\documentclass[11pt, a4paper]{article}\usepackage{amsfonts}
\usepackage{amsmath}
\usepackage{amsthm}
\usepackage{amscd}
\usepackage{amsmath}
\usepackage{amssymb}
\usepackage{amscd}
\usepackage{amsfonts}
\usepackage{amsbsy}
\usepackage{graphicx}
\usepackage{booktabs}
\usepackage{makecell}
\usepackage{enumitem}
\usepackage{tikz}
\usepackage{booktabs}
\usepackage{multirow}
\usepackage{diagbox}
\usepackage{appendix}
\usepackage{mathrsfs}
\usepackage{amsmath,amsfonts,amsthm,amssymb,amscd,color,xcolor}
\usepackage{graphicx,amscd,mathrsfs,wrapfig,mathrsfs,lipsum}
\usepackage{microtype}
\usepackage{float}
\usepackage{tikz}
\usepackage{multicol}
\usepackage{caption}
\usetikzlibrary{arrows}
\usepackage{comment}
\usepackage{framed}
\usepackage{tocloft}
\usepackage{hyperref}

\hypersetup{colorlinks=true,citecolor=blue,%
	linkcolor=blue,urlcolor=blue,bookmarksnumbered=true,
	bookmarksopen=true,bookmarksopenlevel=1,breaklinks=true}

\allowdisplaybreaks[4]

\newcommand{\la}{\lambda}

\newcommand{\dd}{{\textup d}}

\newcommand{\PPPP}{{\mathfrak P}}

\DeclareRobustCommand{\rchi}{{\mathpalette\irchi\relax}}
\newcommand{\irchi}[2]{\raisebox{\depth}{$#1\chi$}}

\let\emptyset\varnothing
\newtheorem*{maintheorem}{Main Theorem}

\newtheorem{theorem}{Theorem}[section]
\newtheorem{proposition}[theorem]{Proposition}
\newtheorem{lemma}[theorem]{Lemma}

\newtheorem{corollary}[theorem]{Corollary}

\theoremstyle{definition}

\theoremstyle{definition}

\numberwithin{equation}{section}

\numberwithin{figure}{section}
\def\9{{\infty}}

\def\({\left(}
\def\){\right)}
\def\<{\langle}
\def\>{\rangle}

\colorlet{darkblue}{blue!50!black}

\hypersetup{
	colorlinks,%
	citecolor=blue,%
	filecolor=red,%
	linkcolor=darkblue,%
	urlcolor=blue,%
	pdfnewwindow=true,%
	pdfstartview={FitH}
}

\def\polhk#1{\setbox0=\hbox{#1}{\ooalign{\hidewidth
			\lower1.5ex\hbox{`}\hidewidth\crcr\unhbox0}}}
\def\polhk#1{\setbox0=\hbox{#1}{\ooalign{\hidewidth
			\lower1.5ex\hbox{`}\hidewidth\crcr\unhbox0}}}
\def\polhk#1{\setbox0=\hbox{#1}{\ooalign{\hidewidth
			\lower1.5ex\hbox{`}\hidewidth\crcr\unhbox0}}} 
\def\polhk#1{\setbox0=\hbox{#1}{\ooalign{\hidewidth
			\lower1.5ex\hbox{`}\hidewidth\crcr\unhbox0}}}

\providecommand{\MR}{\relax\ifhmode\unskip\space\fi MR }

\title{Polynomial mixing for the weakly damped stochastic nonlinear Schr\"odinger equation on the whole space}

\date{\today}
\begin{document}

\author{Vahagn~Nersesyan\,\footnote{NYU-ECNU Institute of Mathematical Sciences at NYU Shanghai, 3663 Zhongshan Road North, Shanghai, 200062, China, e-mail: \href{mailto:vn2111@nyu.edu}{vn2111@nyu.edu}.}         \and 
  Tianyi Pan\,\footnote{NYU-ECNU Institute of Mathematical Sciences at NYU Shanghai, 3663 Zhongshan Road North, Shanghai, 200062, China, e-mail: \href{mailto:tp2861@nyu.edu}{tp2861@nyu.edu}.}         
}

\maketitle

\begin{abstract}
We consider the weakly damped stochastic nonlinear Schr\"odinger (NLS) equation on the real line, driven by a noise that is white in time and smooth in space. Assuming that the noise is sufficiently non-degenerate, we prove that the equation has a unique stationary measure in the class of probability measures concentrated on~$H^2$, and establish
polynomial mixing in the dual-Lipschitz metric over~$H^1$. We do not impose any restriction on the size of the damping.
The proof is based on a coupling argument, whose key ingredient is a
Foia\c{s}--Prodi-type estimate in the $H^1$-norm, derived by means of a Lyapunov functional adapted to the linearized NLS dynamics. To compensate for the loss
of compactness, we combine this estimate with a truncated Poincar\'e
inequality and a space-time weight function quantifying the spatial decay of
solutions.
\end{abstract}



 \tableofcontents

 \section{Introduction}

In this paper, we study the long-time behavior of the weakly damped stochastic nonlinear Schr\"odinger (NLS) equation
	\begin{align}
	\left\{
	\begin{aligned}
		&\partial_t u+\alpha u-i\partial_x^2 {u}+i|u|^2u=b\partial_t W, \quad x\in\mathbb R,\\
		&	{u}(0,x)={u}_0(x),  \label{equation-1}
	\end{aligned}	
	\right.
\end{align}where $\alpha>0$ is the damping coefficient, for which no largeness assumption is made. The equation is driven by a noise of the form \begin{align}\label{04130915} 
b\partial_tW=\sum_{j=1}^\infty b_j\big(\partial_t\beta^1_j(t) +i\partial_t\beta^2_j(t)\big)e_j(x), 
\end{align} where $\{b_j\}_{j\ge1}$ are non-negative constants,
$\{\beta_j^1,\beta_j^2\}_{j\ge1}$ are independent standard real-valued Brownian
motions on a filtered probability space
$(\Omega,\mathcal F,\mathcal F_t,\mathbb P)$  satisfying the usual conditions (see, e.g., Definition~2.25 in \cite{KS14}), and~$\{e_j\}_{j\ge1}$ is an orthonormal basis of $L^2(\mathbb R;\mathbb R)$. We also assume that $e_j\in H^4(\mathbb R;\mathbb R)$ for every~$j\ge1$ and 
\begin{gather}
   \mathcal B_\varphi:=\sum_{j=1}^\infty b_j^2\|\varphi e_j\|^2<\infty,
   \label{03300242}\\
   \mathcal B_4:=\sum_{j=1}^\infty b_j^2\|e_j\|_{H^4}^2<\infty,
   \label{03300242b}
\end{gather}
with $\varphi(x):=\sqrt{1+x^2}$. Both hold, for instance, for the Hermite basis,
provided the coefficients $b_j$ decay sufficiently fast. Here and in what follows, we write $H^k:=H^k(\mathbb R;\mathbb C)$ and $L^2:=L^2(\mathbb R;\mathbb C)$,
with norms $\|\cdot\|_{H^k}$ and $\|\cdot\|$; the real-valued spaces are denoted by
$H^k(\mathbb R;\mathbb R)$ and $L^2(\mathbb R;\mathbb R)$.

Under the above assumptions, problem~\eqref{equation-1} is well posed
in~$H^1$: for every $u_0\in H^1$ it has almost surely a unique solution
$u\in C(\mathbb R_+;H^1)$, depending continuously on~$u_0$, and the solutions
form a Feller family of Markov processes $(u_t,\mathbb P_{u_0})$ parameterized
by the initial condition $u_0\in H^1$, where $\mathbb P_{u_0}$ denotes the
probability corresponding to the solution issued from~$u_0$. Moreover, $H^2$-regularity is propagated:
if $u_0\in H^2$, then almost surely $u\in C(\mathbb R_+;H^2)$ (see, e.g., \cite{DD03}). For a Banach space~$X$, let $\mathcal{P}(X)$ denote the set of Borel probability
measures on~$X$. We write
$P_t(u,\Gamma)=\mathbb{P}_u\{u_t\in\Gamma\}$ for the associated
transition function and denote by 
\begin{align*}
 \PPPP_tf(u)&=\int_{H^1}f(v)\,P_t(u,\dd v),
 \qquad \PPPP_t:C_b(H^1)\to C_b(H^1),\\
 \PPPP^*_t\lambda(\Gamma)&=\int_{H^1}P_t(u,\Gamma)\,\lambda(\dd u),
 \qquad \PPPP^*_t:\mathcal{P}(H^1)\to \mathcal{P}(H^1)
\end{align*}
the corresponding Markov semigroups. Recall that a measure
$\nu \in \mathcal{P}(H^1)$ is stationary for $(u_t, \mathbb{P}_{u_0})$
if $\PPPP_t^* \nu = \nu$ for all~$t \ge 0$. Since $H^2$ is a Borel subset
of~$H^1$, we regard $\mathcal{P}(H^2)$ as a subset of~$\mathcal{P}(H^1)$
in the natural way, so that stationarity of a measure in~$\mathcal{P}(H^2)$
refers to the semigroup~$\PPPP_t^*$ acting on~$\mathcal{P}(H^1)$.
\begin{maintheorem}
There exists an integer $N_0\geq1$, depending only on $\alpha$, on the
basis $\{e_j\}_{j\ge1}$, and on the quantities $\mathcal{B}_\varphi$ and
$\mathcal{B}_4$, such that if
\begin{align}\label{07070547}
     b_j\neq 0 \quad \text{for }j=1,\ldots,N_0,
\end{align}
then the family $(u_t,\mathbb{P}_{u_0})$ admits a unique stationary measure
$\nu$ in~$\mathcal{P}(H^2)$. Moreover, $\nu$ has finite moments of every order
in~$H^2$, and for every $q\geq1$ there is a constant $c_q>0$ such that
\begin{align}\label{05050801}
     \|\PPPP_t^*\la-\nu\|^*_{L(H^1)}
     \leq c_q\left(1+\int_{H^2}\|u\|_{H^2}^{20}\,\la(\dd u)\right)t^{-q}
\end{align}
for any $\la\in\mathcal{P}(H^2)$ and $t>0$. Here $\|\cdot\|^*_{L(H^1)}$ is the
dual-Lipschitz metric defined in~\eqref{E:DL}.
\end{maintheorem}
Several comments on this result are in order. First, no restriction is
imposed on the size of the damping parameter $\alpha$, nor on its relation
to the intensity of the noise. Second, the integer $N_0$ does not depend on
the rate $q$ in~\eqref{05050801}, so the same finitely many non-degenerate
modes ensure polynomial mixing of every order. Third, the convergence holds
in the dual-Lipschitz metric over $H^1$, whereas the quantity appearing on
the right-hand side is the $20$th moment of the initial measure in $H^2$.
This gap between the two spaces is intrinsic to the method: as explained
below, the stabilization and Girsanov-type arguments require control of the
$H^2$-norm of the trajectories. For the same reason, uniqueness is
established in the class $\mathcal P(H^2)$; since the equation has no
smoothing mechanism, whether a stationary measure charging~$H^1\setminus H^2$ can exist remains an open question. Finally, the exponent~$20$ originates from the Lyapunov-type functionals used in the recurrence
estimates and is unlikely to be optimal.

Since the mid-1990s, considerable progress has been made in understanding the ergodic properties of PDEs perturbed by white-in-time noise. The first results were obtained in~\cite{FM95,KS02,EMS01,Ha02,HM06} for the Navier--Stokes system and other parabolic equations on bounded domains; we refer the reader to the book~\cite{KS12} and the surveys~\cite{Deb13,KS17} for a detailed account of the theory. Two features of the underlying dynamics play an essential role in these works: the regularizing effect and strong dissipation provided by the parabolic structure of the equations, and the compactness properties available on bounded domains, such as the compactness of Sobolev embeddings and the discreteness of the spectrum of the Laplacian. Both features fail for problem~\eqref{equation-1}. The equation is
dispersive, so that the linear evolution preserves every $H^k$-norm
and the only dissipation is the zeroth-order damping term~$\alpha u$;
moreover, the equation is posed on the whole real line, so that all
compactness properties are lost. The Main Theorem provides the first uniqueness and mixing result for the stochastic NLS equation in an unbounded domain \emph{without} any largeness assumption on the damping.

Let us recall some previous results on the ergodicity of the stochastic NLS
equation. In the weakly damped case, the existence of a stationary measure
was established in~\cite{Kim06} on a bounded domain and in~\cite{EKZ17} on
the whole space. Much less is known about uniqueness and mixing. The first result in this
direction was obtained in~\cite{DO05}, where the weakly damped NLS equation
was considered on a bounded interval and the uniqueness of the stationary
measure was proved, together with a polynomial rate of convergence to it.
The polynomial rather than exponential rate is closely related to weak damping: in contrast with parabolic problems, the energy-type functionals of the NLS equation admit moments that grow linearly in time, but no exponential moment bounds; this rules out the supermartingale-type arguments on which the proofs of exponential mixing are based. The~same obstruction is present in our setting, which is why our result gives
a polynomial and not an exponential rate; exponential mixing for the weakly
damped equation remains open even on the torus.

Let us also mention~\cite{CHW17}, which proves that a fully discrete scheme
for the weakly damped stochastic NLS equation on a bounded interval is
uniquely ergodic and approximates the stationary measure of the underlying
continuous equation; see also the book~\cite{HW19} for a systematic account. 
 Finally, in the case of a
bounded noise localized in space or in Fourier modes, exponential mixing for
the NLS equation on the torus has recently been obtained
in~\cite{CXZZ25,CXZZ26} via a controllability approach; see also~\cite{CX25}
for related large deviations results.

All the uniqueness results mentioned above concern equations on bounded domains. For randomly forced PDEs in unbounded domains, only a few
works address the uniqueness of the stationary measure and mixing. Most of
them are devoted to Burgers-type equations perturbed by a space-time
homogeneous noise (see~\cite{Bak13,BCK14,BaL19,DGR21}); the methods of these papers exploit structural features specific to
the Burgers equation (the Hopf--Cole transform, variational
representations of solutions, and order-preservation properties) and
do not apply to Schr\"odinger-type equations. For the Navier--Stokes system
driven by a non-homogeneous bounded noise, uniqueness and exponential
mixing in unbounded domains are established in~\cite{Ner22} via a
controllability approach combined with the asymptotic compactness of the
dynamics. In the case of a white-in-time noise, the uniqueness and mixing for
the complex Ginzburg--Landau equation on the real line are proved
in~\cite{NZ24} by introducing a space-time weight function that
quantifies the spatial decay of solutions and compensates for the loss
of compactness; these ideas have been extended to the Navier--Stokes system in the
whole space in~\cite{NZ24b}, to the damped nonlinear wave equation
in~\cite{G24}, and to viscous conservation laws on the line
in~\cite{GL26}. All
these equations possess either a parabolic regularization or a strong
dissipation mechanism that is not available for the weakly damped NLS
equation considered here. Let us also mention the papers~\cite{BFZ23,NS25}, where the ergodicity
of the stochastic NLS equation is established under the assumption that
the damping is sufficiently large.

The proof of the Main Theorem is based on a coupling argument; see
Section~3.1.2 of~\cite{KS12} for a systematic presentation of this method in
the context of randomly forced dissipative PDEs. Such arguments were applied in~\cite{DO05} to the weakly
damped NLS equation on a bounded interval, in~\cite{Mar14} to the damped
nonlinear wave equation, and in~\cite{NZ24,NZ24b} to equations on unbounded
domains. Each of these works, however, relies on a structural feature that is absent in our setting: a spectral gap and compact embeddings on a bounded domain, exponential moment bounds for the energy, or a parabolic regularization supporting the weighted estimates. Equation~\eqref{equation-1} has none of these properties: being posed on the
whole real line, it enjoys neither a spectral gap nor compact embeddings; its
only dissipation is the zeroth-order damping; and no exponential moment bounds
are available for its energy functionals. The main ingredients of the scheme therefore require a different approach.

A central component of the method is a Foia\c{s}--Prodi-type estimate,
which expresses a stabilization property by an additive finite-dimensional
feedback; in the deterministic setting such estimates go back
to~\cite{FP67}, and they were first combined with the Girsanov theorem in
the study of mixing for SPDEs in~\cite{EMS01,KS02,Ha02}. 
 In our setting, the derivation of this estimate
faces two difficulties. The first one comes from the unboundedness of
the domain: as in~\cite{NZ24,NZ24b}, we use a truncated Poincar\'e inequality together
with the space-time weight function
\begin{equation}\label{E:psi-def}
\psi(t,x):=\varphi(x)\Big(1-\exp\Big\{-\frac{t}{\varphi(x)}\Big\}\Big),
\qquad \varphi(x):=\sqrt{1+x^2},
\end{equation}
whose spatial growth controls the behavior of solutions at infinity,
while the vanishing of $\psi$ at $t=0$ keeps the weighted norm
$\|\psi(t)u_t\|$ finite without any weighted integrability assumption
on the initial data. The second
difficulty is specific to the NLS equation. Recall that in the coupling
method one introduces, along with a solution~$u$
of~\eqref{equation-1}, an auxiliary process~$v$ solving the same
equation with the same noise, supplemented by a finite-dimensional
feedback term that drives $v$ toward~$u$ (see~\eqref{auxiliary}). Because the equation has no smoothing, the cubic
nonlinearity cannot be controlled by the $L^2$-norm of the difference
$w=u-v$ alone, and the Foia\c{s}--Prodi estimate must instead be
established in the $H^1$-norm. To this end, we introduce a Lyapunov-type functional, obtained by
adding to the $H^1$-energy of $w$ the cross terms~$\langle |u|^2+|v|^2,|w|^2\rangle$ and $\langle uv,w^2\rangle$ arising
from the energy of the linearized equation. Its It\^o evolution is
controlled by energy-type functionals of $u$ and $v$, including a
second-order functional governing the~$H^2$-norm, together with the weighted
norms $\|\psi(t)u\|$ and $\|\psi(t)v\|$; this yields an averaged
Foia\c{s}--Prodi-type estimate (see Proposition~\ref{fpestimate}).

As explained above, the energy functionals admit no exponential moment bounds; accordingly, the stopping times controlling their growth have only polynomially decaying tails (see the Appendix). Combined with the Foia\c{s}--Prodi-type estimate and the Girsanov theorem,
these bounds produce, for any two initial points in a ball in $H^2$, a
coupling of the corresponding trajectories whose probability of breaking down
decays polynomially in time. Together with an exponential
recurrence property for the $H^2$-norm of solutions, obtained from a
stabilization argument for the unforced equation and Lyapunov-type
estimates, this gives polynomial moment bounds of arbitrary order for
the last decoupling time, from which the polynomial
mixing~\eqref{05050801} follows. This is where the $H^2$-regularity of the initial data enters: as noted above,
controlling the It\^o evolution of the Lyapunov-type functional requires the
second-order energy, and both the Girsanov-type estimates and the recurrence
argument rely on it; correspondingly, the stationary measure is shown to be
concentrated on~$H^2$.

The paper is organized as follows. Section~\ref{S:2} is devoted to the Foia\c{s}--Prodi-type estimates in
$H^1$ and their consequences. In Section~\ref{S:3}, we compare, by means of the Girsanov
theorem, the law of the auxiliary process with that of a solution of
the original equation, and estimate the probability that its energy
functionals grow abnormally. The coupling construction and the proof of
the Main Theorem are carried out in Section~\ref{S:4}. Finally, the
Appendix collects the proofs of several auxiliary results used
throughout the paper.

 \subsection*{Notation}

\emph{Function spaces.}
For $p\in[1,\infty]$, we set $L^p:=L^p(\mathbb{R};\mathbb{C})$, with norm
$\|\cdot\|_{L^p}$. Throughout the paper $L^2$ is regarded as a {\it real}
Hilbert space, with inner product
\begin{align*}
    \langle u,v\rangle:=\mathrm{Re}\int_{\mathbb{R}} u(x)\overline{v}(x)\,\dd x
\end{align*}
and the corresponding norm $\|\cdot\|$.
For an integer $k\ge1$, $H^k:=H^k(\mathbb{R};\mathbb{C})$ is the Sobolev space
of order~$k$, similarly regarded as a real Hilbert space, with norm denoted by
$\|\cdot\|_{H^k}$.

\smallskip
\noindent
\emph{Measures and distances.} For a Banach space~$X$, we denote by $B_X(R)$ the closed ball in~$X$ of
radius $R>0$ centered at the origin, by $\mathcal{B}(X)$ the Borel
$\sigma$-algebra of~$X$, and by $\mathcal{P}(X)$ the set of Borel
probability measures on~$X$.
 We denote
by $L_b(X)$ the space of bounded real-valued Lipschitz functions
$f:X\to\mathbb{R}$, endowed with the norm
\begin{align*}
    \|f\|_{L(X)}:=\sup_{u\in X}|f(u)|
    +\sup_{\substack{u,v\in X\\ u\neq v}}
    \frac{|f(u)-f(v)|}{\|u-v\|_X}.
\end{align*}
For $\mu_1,\mu_2\in\mathcal{P}(X)$, we define the dual-Lipschitz and total
variation distances by
\begin{gather}
    \|\mu_1-\mu_2\|^{*}_{L(X)}:=\sup\Big\{\Big|\int_X f\,\dd\mu_1
    -\int_X f\,\dd\mu_2\Big|:
     \|f\|_{L(X)}\leq 1\Big\},\label{E:DL}
    \\
    \|\mu_1-\mu_2\|_{\mathrm{var}}
    :=\sup_{\Gamma\in\mathcal{B}(X)}|\mu_1(\Gamma)-\mu_2(\Gamma)|.\nonumber
\end{gather}
The set $\mathcal{P}(X)$ is endowed with the dual-Lipschitz metric; when the space~$X$ is clear from the context,
we abbreviate $\|\cdot\|_{L(X)}$ and~$\|\cdot\|^{*}_{L(X)}$ to $\|\cdot\|_{L}$
and~$\|\cdot\|^{*}_{L}$. For $T>0$, we denote by $C([0,T];X)$ the space of continuous
functions from $[0,T]$ to~$X$, and by $C_0([0,T];X)$ the subspace of
functions vanishing at $t=0$; the spaces $C([0,\infty);X)$ and
$C_0([0,\infty);X)$ are defined similarly.

\smallskip
\noindent
\emph{Miscellaneous.}
 For a random variable $\xi$ and an event~$A$, we denote
by~$\mathcal{D}(\xi)$ the law of~$\xi$ and write
$\mathbb{E}[\xi;A]:=\mathbb{E}[\xi\mathbf{1}_A]$, $\mathbb{E}[\xi;A,B]:=\mathbb{E}[\xi\mathbf{1}_{A\cap B}]$ where~$\mathbf{1}_A$ denotes
 the indicator function of~$A$. We write $a\wedge b:=\min\{a,b\}$ and~$a\vee b:=\max\{a,b\}$. Constants denoted $C,C_N,C_{p,q},\ldots$ are positive and inessential, and may
change from line to line, their subscripts indicating the parameters they
depend on. All other constants (for example $\mathcal{C}_i$, $\mathcal{C}_p^\sharp$,
$K$, $K_{i,p}$, $\mathscr{C}$,  and $\rho_*$) are fixed once and for all
at the point where they are introduced.

\section{Foia\c{s}--Prodi estimate}\label{S:2}

The coupling method used in the proof of the Main Theorem relies on a
stabilization mechanism: given two initial conditions $u_0$ and $u_0^\prime$,
one adds to the second equation a finite-dimensional feedback force that drives
the corresponding solution toward the solution issued from~$u_0$. This is implemented through an auxiliary process $v$, defined as the
solution of
 	\begin{align}
	\left\{
	\begin{aligned}
		&\partial_t v+\alpha v -i\partial_{x}^2 {v}+i|v|^2v=i\mathsf{P}_N(|v|^2v-|u|^2u-\partial_x^2(v-u))+b\partial_t W,\\
		&	{v}(0,x)={u}^\prime_0(x),\label{auxiliary}
	\end{aligned}	
	\right.
\end{align}
where $u$ is the solution of \eqref{equation-1} starting from $u_0$, and
$\mathsf{P}_N$ is the orthogonal projection onto the 
low modes:
\begin{align*}
    \mathsf{P}_N\xi:=\sum_{j=1}^N\langle ie_j,\xi\rangle ie_j+\sum_{j=1}^N\langle e_j,\xi\rangle e_j.
\end{align*}
 The projection $\mathsf{P}_N$ commutes with multiplication by $i$ and with complex conjugation; namely $\mathsf{P}_N(i\xi)=i\mathsf{P}_N\xi$ and $\overline{\mathsf{P}_N\xi}=\mathsf{P}_N\overline{\xi}$ for any $\xi\in L^2$. In~particular, 
 \begin{align}\label{08301619}
 \langle i\mathsf{P}_N\xi,\xi\rangle=\langle i\mathsf{Q}_N\xi,\xi\rangle=0,
 \end{align}where $\mathsf{Q}_N:=I-\mathsf{P}_N$.
 
  The well-posedness of~\eqref{auxiliary} in
$H^1$ is established in Appendix~\ref{WPA}. The feedback term
$$
i\mathsf{P}_N(|v|^2v-|u|^2u-\partial_x^2(v-u))
$$
is designed to serve two purposes that pull in opposite directions. It must be
strong enough to stabilize $v$ toward $u$: the difference $w:=u-v$ then decays
in the $H^1$-norm, in the averaged sense made precise in Proposition~\ref{fpestimate}, provided that suitable energy functionals of $u$
and $v$ grow at most linearly in time. At the
same time, it must be supported on finitely many modes, so that the Girsanov
theorem can be used to compare the laws of $v$ and~$u^\prime$, the solution to \eqref{equation-1} starting from $u_0^\prime$, provided that
$b_1,\ldots,b_N\neq0$; the corresponding estimates are established in
Section~\ref{S:3} and exploited in the coupling construction of
Section~\ref{S:4}.
 Throughout
this section, $N\geq1$ is arbitrary, and $u_0,u_0^\prime\in H^2$.

The difference $w=u-v$ solves the equation
\begin{align}\label{02230853-1}
\left\{
	\begin{aligned}
    &\partial_tw+\alpha w-i\mathsf{Q}_N(\partial_x^2w)+i\mathsf{Q}_N\big[\big(|u|^2+|v|^2\big)w+uv\overline{w}\big]=0,\\
    &w(0)=u_0-u^\prime_0,
    \end{aligned}
\right.
\end{align}which implies that
\begin{align}\label{03190251}
    \mathsf{P}_N w_t=e^{-\alpha(t-s)}\mathsf{P}_Nw_s,
    \qquad t\geq s\geq 0.
\end{align}

To compensate for the loss of compactness on the unbounded domain, we shall combine two tools. The first is the following truncated
Poincar\'e inequality from~\cite{NZ24}. For any $A>0$, let $\rchi_A:\mathbb{R}\rightarrow[0,1]$ be a smooth cutoff function satisfying
\begin{equation}
\rchi_A(x)=\left\{ 
\begin{aligned}
    &1,\quad |x|\leq A,\\&
    0,\quad |x|\geq 2A.\label{03180955}
\end{aligned}
\right.
\end{equation}
\begin{lemma}[\cite{NZ24}, Lemma~3.1]\label{03011437}
     For any $\varepsilon>0$ and $A>0$, there exists an integer
     $N_{\varepsilon,A}\geq 1$ such that
    \begin{align*}
        \|\mathsf{Q}_N\rchi_A f\|\leq \varepsilon\|f\|_{H^{1}}
    \end{align*}
    for any  $N\geq N_{\varepsilon,A}$  and $f\in H^1$.
\end{lemma}

 The second tool, in the spirit of the papers~\cite{NZ24,NZ24b}, is the
space-time weight function $\psi$ defined in~\eqref{E:psi-def}. We begin with the following estimate~for~$\|w\|$.
\begin{lemma}\label{03011027}
For any $\varepsilon>0$, there exist $N^0_\varepsilon\ge 1$ and $T^0_\varepsilon>0$ such that for any $N\geq N^0_\varepsilon$,  $T\geq T^0_\varepsilon$, and $t\geq 0$, we have
\begin{align*}
    \frac{1}{2}\|w_{t+T}\|^2&+\frac{3\alpha}{4}\int_{ T}^{t+T}\|w_r\|^2\dd r\leq  \frac{1}{2}\|w_{ T}\|^2 +C_{N,\varepsilon}\int_{ T}^{t+T}\|\mathsf{P}_Nw_r\|^2\dd r\nonumber\\&+\varepsilon\int_{ T}^{t+T}\|w_r\|_{H^1}^2\big(\|u_r\|^2_{H^1}+\|v_r\|^2_{H^1}+\|\psi(r)u_r\|^2+\|\psi(r)v_r\|^2\big)\dd r. 
\end{align*}
\end{lemma}
\begin{proof}
Taking the $L^2$-inner product of~\eqref{02230853-1} with $w$ gives
\begin{align}\label{03070318}
    \nonumber\frac{1}{2}\|w_{t+T}\|^2&-\frac{1}{2}\|w_{ T}\|^2+\alpha\int_{ T}^{t+T}\|w_r\|^2\dd r\\\nonumber&=\int_{ T}^{t+T}\langle i\mathsf{Q}_N(\partial_x^2w_r),w_r\rangle\dd r-\int_{ T}^{t+T}\langle i\mathsf{Q}_N\big[u_rv_r\overline{w}_r\big],w_r\rangle\dd r\\&\quad-\int_{ T}^{t+T}\langle i\mathsf{Q}_N\big[(|u_r|^2+|v_r|^2)w_r\big],w_r\rangle\dd r=: \mathrm{I}+\mathrm{II}+\mathrm{III}.
\end{align}
Since $\langle i\partial_x^2 w,w\rangle = 0$, we have
$$
\langle i\mathsf{Q}_N(\partial_x^2 w),w\rangle
= -\langle i\mathsf{P}_N(\partial_x^2 w),w\rangle=\langle   w,i\partial_x^2 \mathsf{P}_N w\rangle.
$$
Therefore,
\begin{align}\label{04060303}
    \mathrm{I}\notag & = \int_{ T}^{t+T}\langle  w_r,i\partial_x^2\mathsf{P}_N w_r\rangle\dd r\le \int_{ T}^{t+T}\|\partial_x^2\mathsf{P}_N w_r\|\|w_r\|\dd r\\&   \leq \frac{\alpha}{4}\int_{ T}^{t+T}\|w_r\|^2\dd r+C_N\int_{ T}^{t+T}\|\mathsf{P}_Nw_r\|^2\dd r,
\end{align}where the last step uses $\|\partial_x^2\mathsf{P}_N w \|\leq C_N\|\mathsf{P}_Nw\|$ and Young's inequality.

Since $\psi(t,x)\to\infty$ as $t$ and $|x|$ both tend to $\infty$, we choose
$T^\prime_\varepsilon>0$ and $A_\varepsilon>0$ so large that
\begin{align}\label{E:psi}
    \psi(t,x)\geq \frac{1}{\varepsilon}, \qquad  t\geq T^\prime_\varepsilon,\ |x|\geq A_\varepsilon.
\end{align}
Let $\rchi_{A_\varepsilon}$ be the cutoff function defined in~\eqref{03180955}
with $A=A_\varepsilon$. We split $\mathrm{II}$ into a part supported far from
the origin and a localized one:
\begin{align*}
    \mathrm{II}&\le\int_{T}^{t+T}\|w_r\|\|\mathsf{Q}_N(u_rv_r\overline{w}_r)\|\dd r \leq  \int_{ T}^{t+T}\|w_r\|\|\mathsf{Q}_N((1-\rchi_{A_\varepsilon})u_rv_r\overline{w}_r)\|\dd r\\&\quad+\int_{ T}^{t+T}\|w_r\|\|\mathsf{Q}_N(\rchi_{A_\varepsilon} u_rv_r\overline{w}_r)\|\dd r=: \mathrm{II}_1+\mathrm{II}_2.
\end{align*}
Taking $T\geq T^\prime_\varepsilon$ and using \eqref{E:psi}, we get
\begin{align}
    \nonumber \mathrm{II}_1&  \leq  \int_{ T}^{t+T} \|w_r\|\big( \int_{|x|\geq {A_\varepsilon}} |u_r(x)v_r(x)\overline{w}_r(x)|^2\dd x\big)^{\frac{1}{2}}\dd r\nonumber\\&\nonumber\leq {\varepsilon} \int_{ T}^{t+T}\hspace{-1mm}\|w_r\|\|\psi(r)u_rv_r\overline{w}_r\|\dd r\\&\nonumber  \leq{\varepsilon}\int_{ T}^{t+T}\|w_r\|\|w_r\|_{L^\infty}\|\psi(r)u_r\|\|v_r\|_{L^\infty}\dd r.
\end{align}
By the Gagliardo--Nirenberg inequality,  we have
\begin{align}\label{07200244}
    \|u\|_{L^\infty}\leq C\|u\|^{\frac{1}{2}}\|u\|_{H^1}^{\frac{1}{2}}. 
\end{align}
 Thus
\begin{align}\label{07202136}
    \mathrm{II}_1\leq C\varepsilon\int_{ T}^{t+T}\|w_r\|_{H^1}^2\|\psi(r)u_r\|\|v_r\|_{H^1}\dd r.
\end{align}
By Lemma~\ref{03011437}, there is $N_{\varepsilon, A_\varepsilon}\ge1$ such that, for any $N\ge N_{\varepsilon, A_\varepsilon}$,
\begin{align}\label{03290745}
      \nonumber \mathrm{II}_2& \leq \varepsilon\int_{ T}^{t+T}\|w_r\|\|u_rv_r\overline{w}_r\|_{H^1}\dd r\\&\leq\nonumber   C\varepsilon\int_{ T}^{t+T}\|w_r\|\|w_r\|_{L^\infty}\big(\|u_r\|_{H^1}\|v_r\|_{L^\infty}+\|v_r\|_{H^1}\|u_r\|_{L^\infty}\big)\dd r\\\nonumber &\quad+C\varepsilon\int_{ T}^{t+T}\|w_r\|\|w_r\|_{H^1}\|u_r\|_{L^\infty}\|v_r\|_{L^\infty}\dd r\\&\leq C \varepsilon\int_{ T}^{t+T}\|u_r\|_{H^1}\|v_r\|_{H^1}\|w_r\|_{H^1}^2\dd r.
\end{align}
Combining \eqref{07202136} and \eqref{03290745}, we obtain
\begin{align}\label{04060259}
    \mathrm{II}\leq C\varepsilon\int_{ T}^{t+T}(\|u_r\|_{H^1}+\|\psi(r)u_r\|)\|v_r\|_{H^1}\|w_r\|_{H^1}^2\dd r.
\end{align}
The same argument, applied to the terms $|u|^2 w$ and $|v|^2 w$, yields
\begin{align}\label{03290746}
\nonumber \mathrm{III}&\leq\int_{ T}^{t+T}\|\mathsf{Q}_N[(|u_r|^2+|v_r|^2)w_r]\|\|w_r\|\dd r\\&\leq  C \varepsilon\int_{ T}^{t+T}(\|u_r\|_{H^1}+\|\psi(r)u_r\|)\|u_r\|_{H^1}\|w_r\|_{H^1}^2\dd r\nonumber \\&\quad +C \varepsilon\int_{ T}^{t+T}(\|v_r\|_{H^1}+\|\psi(r)v_r\|)\|v_r\|_{H^1}\|w_r\|_{H^1}^2\dd r.
\end{align}
Combining \eqref{03070318}, \eqref{04060303}, \eqref{04060259}
and~\eqref{03290746} and applying Young's inequality, we obtain the required
estimate with $C\varepsilon$ in place of~$\varepsilon$; since $\varepsilon>0$
is arbitrary, this proves the lemma.
\end{proof} 
To formulate the Foia\c{s}--Prodi estimate in the $H^1$-norm, we first introduce
the energy-type functionals $\mathcal{H}:H^1\rightarrow\mathbb{R}$ and
$\mathcal{G}:H^2\rightarrow\mathbb{R}$,
\begin{align}\label{03090247-1}
\mathcal{H}(u)&:=\frac{1}{2}\|\partial_xu\|^2+\frac{1}{4}\|u\|_{L^4}^4,
\\ \mathcal{G}(u)&:=\frac{1}{2}\|\partial_x^2u\|^2+\frac{1}{4}\|\partial_x(|u|^2)\|^2.
\label{03090247-2}
\end{align}
Together with the mass $\|u\|^2$, they control the $H^1$- and $H^2$-norms:
\begin{align}\label{E:HGnorms}
\|u\|_{H^1}^2\leq \|u\|^2+2\mathcal{H}(u),\qquad
\|u\|_{H^2}^2\leq \|u\|^2+2\mathcal{H}(u)+2\mathcal{G}(u).
\end{align}
Next, we introduce two random Lyapunov-type functionals
$$
\mathcal{M},\ \widetilde{\mathcal{M}}:\mathbb{R}_+\times H^1
\rightarrow\mathbb{R},
$$
defined in terms of the solutions $u_s$ and $v_s$ of~\eqref{equation-1}
and~\eqref{auxiliary} by
\begin{align*} 
    \mathcal{M}(s,w)&:=\frac{1}{2}\|\partial_x w\|^2
    +\frac{1}{2}\langle |u_s|^2+|v_s|^2,|w|^2\rangle
    +\frac{1}{2}\langle u_sv_s,{w}^2\rangle,\\
    \widetilde{\mathcal{M}}(s,w)&:=\frac{1}{2}\|w\|^2+\mathcal{M}(s,w).
\end{align*}
The two cross terms arise from the energy of the linearized equation; they are
what make the It\^o evolution of $\mathcal{M}$ tractable, as will become clear
in Steps~1--3 of the proof of Proposition~\ref{fpestimate} below. Since
\begin{align}\label{07200228}
    |\langle u_sv_s,w^2\rangle|\leq
    \frac{1}{2}\langle|u_s|^2+|v_s|^2,|w|^2\rangle,
\end{align}
 we have
\begin{align*}
\mathcal{M}(s,w)\geq\frac{1}{2}\|\partial_xw\|^2,
\end{align*}
and therefore
\begin{align}\label{06190211}
    \widetilde{\mathcal{M}}(s,w)\geq\frac{1}{2}\|w\|_{H^1}^2,
\end{align}
so that $\widetilde{\mathcal{M}}$ controls the $H^1$-norm of $w$. Finally, the functionals of $u$ and $v$ that will appear as coefficients in
the It\^o evolution of $\mathcal{M}$ are collected in the functional
$\mathcal{L}:\mathbb{R}_+\times H^2\times H^2\rightarrow\mathbb{R}$ defined by 
\begin{align*}
\mathcal{L}(r,u,v)&:= \|u\|^{12}+\|v\|^{12}+{\mathcal{H}}^2(u)+{\mathcal{H}}^2(v)
+\mathcal{G}^2(u)+\mathcal{G}^2(v)\\
&\quad+\|\psi(r)u\|^2+\|\psi(r)v\|^2+1.
\end{align*}
 With these notations at hand, we can state the main result of this section.
\begin{proposition}\label{fpestimate}
    For any $\varepsilon>0$, there are~$N_\varepsilon\ge1$ and $T_\varepsilon>0$ such that for any $N\ge N_\varepsilon$, $T\ge T_\varepsilon$, and finite stopping time $\tau$, we have
    \begin{align}\label{07210127}
    e^{\alpha (\tau\vee T-T)-\varepsilon\int_{T}^{\tau\vee T}\mathcal{L}(s,u_s,v_s)\dd s}\widetilde{\mathcal{M}}(\tau\vee T,w_{\tau\vee T})\in L^1(\Omega)
    \end{align}
    and 
    \begin{align}\label{04020327}
    \nonumber\mathbb{E}&\big[e^{\alpha (\tau-T)-\varepsilon\int_{T}^{\tau}\mathcal{L}(s,u_s,v_s)\dd s}\widetilde{\mathcal{M}}(\tau,w_{\tau})\big|\mathcal{F}_{T}\big]\mathbf{1}_{\tau\geq T}\\&\quad\quad\quad\quad\quad\quad\quad\quad\quad\leq \big[\widetilde{\mathcal{M}}(T,w_{T})+C_{N,\varepsilon}\|w_{T}\|^2\big]\mathbf{1}_{\tau\geq T}.
\end{align}
\end{proposition}
\begin{proof}
Let us fix any $\varepsilon>0$ and estimate each term of $\mathcal{M}(t,w_t)$ separately.  

\noindent{\it Step 1: first term}.
To estimate $\|\partial_x w\|^2$, we use \eqref{02230853-1} and integration by~parts:
\begin{align}\label{03291134}
    \frac12\|\partial_x w_{t+T}\|^2-\frac12\|\partial_x w_{T}\|^2&+\alpha\int_{T}^{t+T} \|\partial_xw_r\|^2\dd r\nonumber\\\nonumber&=\int_{T}^{t+T}\langle i\mathsf{Q}_N\big[\big(|u_r|^2+|v_r|^2\big)w_r\big],\partial_x^2 w_r\rangle\dd r \\&\quad+\int_{T}^{t+T} \langle i\mathsf{Q}_N(u_rv_r\overline{w}_r),\partial_x^2w_r\rangle \dd r.
\end{align}
\noindent{\it Step 2: second term}. It follows from the product rule that
\begin{align*}
\frac{1}{2}\langle(|u_{t+T}|^2+|v_{t+T}|^2)&,|w_{t+T}|^2\rangle-\frac{1}{2}\langle(|u_{T}|^2+|v_{T}|^2),|w_{T}|^2\rangle\\&=\frac{1}{2}\int_{ T}^{t+T}\langle |u_r|^2+|v_r|^2, \dd |w_r|^2\rangle+\frac{1}{2}\int_{ T}^{t+T}\langle |w_r|^2, \dd |u_r|^2\rangle\\&\quad+\frac{1}{2}\int_{ T}^{t+T}\langle |w_r|^2, \dd |v_r|^2\rangle=:\mathrm{I+II+III}.
\end{align*}
To estimate $\mathrm{I}$, we use \eqref{02230853-1} and \eqref{08301619} to deduce 
\begin{align*} 
    \nonumber\mathrm{I}=&\ \mathrm{Re}\int_{ T}^{t+T}\int_{\mathbb{R}}\big[(|u_r|^2+|v_r|^2)\overline{w}_r\big(-\alpha w_r+i\mathsf{Q}_N(\partial_x^2w_r)\big)
\big]\dd x\dd r\\\nonumber& -\mathrm{Re}\int_{ T}^{t+T}\int_{\mathbb{R}}i(|u_r|^2+|v_r|^2)\overline{w}_r\mathsf{Q}_N((|u_r|^2+|v_r|^2)w_r)\dd x\dd r\\&-\mathrm{Re}\int_{ T}^{t+T}\int_{\mathbb{R}} (|u_r|^2+|v_r|^2)\overline{w}_ri\mathsf{Q}_N(u_rv_r\overline{w}_r)\dd x\dd r.
\end{align*}
Note that the second term on the right-hand side of this equality vanishes. 
The terms $\mathrm{II}$ and $\mathrm{III}$ are treated using It\^o's formula together with \eqref{equation-1} and~\eqref{auxiliary} and assumption \eqref{03300242b}:
\begin{align}\label{07200057}
    \nonumber\mathrm{II}=&\ \mathrm{Re}\int_{ T}^{t+T}\int_{\mathbb{R}} \overline{u}_r|w_r|^2(-\alpha u_r+i\partial_x^2u_r-i|u_r|^2u_r)\dd x\dd r\\&\nonumber+\mathrm{Re}\sum_{j=1}^\infty b_j\int_{ T}^{t+T}\int_{\mathbb{R}} \overline{u}_r|w_r|^2e_j\dd x\big(\dd \beta^1_j(r)+i\dd \beta^2_j(r)\big)\\&\nonumber+\sum_{j=1}^\infty b^2_j\int_{ T}^{t+T}\int_{\mathbb{R}} |w_r|^2e_j^2\dd x\dd r\nonumber\\\leq&\nonumber-\alpha\int_{ T}^{t+T}\int_{\mathbb{R}}|{u}_r|^2|w_r|^2\dd x\dd r+C\int_{ T}^{t+T}\|w_r\|_{L^\infty}^2\|\partial_x^2u_r\|\|u_r\|\dd r\nonumber\\& +\mathrm{Re}\sum_{j=1}^\infty b_j\int_{ T}^{t+T}\hspace{-0.5mm}\int_{\mathbb{R}} \overline{u}_r|w_r|^2e_j\dd x\big(\dd \beta^1_j(r)\hspace{-0.5mm}+i\dd \beta^2_j(r)\big)+C\int_{ T}^{t+T}\|w_r\|^2\dd r
\end{align}
and
\begin{align}\label{07200054}
    \nonumber\mathrm{III}= &\ \mathrm{Re}\int_{T}^{t+T}\int_{\mathbb{R}} \overline{v}_r|w_r|^2\big(-\alpha v_r+i\mathsf{Q}_N(\partial_x^2v_r)-i\mathsf{Q}_N(|v_r|^2v_r)\\&\nonumber-i\mathsf{P}_N(|u_r|^2u_r)+i\mathsf{P}_N(\partial_x^2u_r)\big)\dd x\dd r\\&\nonumber+ \mathrm{Re}\sum_{j=1}^\infty b_j\int_{ T}^{t+T} \int_{\mathbb{R}} \overline{v}_r|w_r|^2e_j\dd x\big(\dd \beta^1_j(r)+i\dd \beta^2_j(r)\big)\\&+\sum_{j=1}^\infty b^2_j\int_{ T}^{t+T} \int_{\mathbb{R}} |w_r|^2e_j^2\dd x\dd r\nonumber\\\nonumber\leq&-\alpha\int_{ T}^{t+T} \int_{\mathbb{R}}|v_r|^2|w_r|^2\dd x\dd r \\&+ \int_{ T}^{t+T} \|w_r\|_{L^\infty}^2\|v_r\|(\|\partial_x^2u_r\|+ \|\partial_x^2v_r\|  + \|u_r\|^3_{L^6}+\|v_r\|_{L^6}^3)\dd r\nonumber\\&+ \mathrm{Re}\sum_{j=1}^\infty b_j\int_{ T}^{t+T}\!\! \int_{\mathbb{R}} \overline{v}_r|w_r|^2e_j\dd x\big(\dd \beta^1_j(r) + i\dd \beta^2_j(r)\big)\!+\!C\int_{ T}^{t+T}\!\!\|w_r\|^2\dd r.
\end{align}
For the second term on the right-hand side of \eqref{07200057}, we have
\begin{align*}
    \|w_r\|_{L^\infty}^2\|\partial_x^2u_r\|\,\|u_r\|
    &\leq \|w_r\|_{L^\infty}^2\big(\|u_r\|^2+\mathcal{G}(u_r)\big)\\
    &\leq C\|w_r\|_{L^\infty}^2\mathcal{L}^{\frac{1}{2}}(r,u_r,v_r).
\end{align*}
Similarly, for the counterpart in \eqref{07200054}, the Gagliardo--Nirenberg
and Young inequalities give
\begin{align*}
\|v_r\|\,\|u_r\|_{L^6}^3
&\leq C\|v_r\|\,\|u_r\|^2\|u_r\|_{H^1}
 \leq C\|v_r\|\big(\|u_r\|^3+\|u_r\|^2\mathcal{H}^{\frac{1}{2}}(u_r)\big)\\
&\leq C\big(\|u_r\|^4+\|v_r\|^4+\|u_r\|^6+\|v_r\|^6+\mathcal{H}(u_r)\big)
 \leq C\mathcal{L}^{\frac{1}{2}}(r,u_r,v_r).
\end{align*}
Applying similar estimates to the terms $\|v\|\|\partial_x^2u\|$ and $\|v\|\|\partial_x^2v\|$, we have
\begin{gather}
    \nonumber\int_{ T}^{t+T} \|w_r\|_{L^\infty}^2\|v_r\| (\|\partial_x^2u_r\|+ \|\partial_x^2v_r\|  + \|u_r\|^3_{L^6}+\|v_r\|_{L^6}^3)\dd r\\\leq   C\int_T^{t+T}\|w_r\|_{L^\infty}^2\mathcal{L}^{\frac{1}{2}}(r,u_r,v_r)\dd r.\label{07200146}
\end{gather}
Combining the above estimates for $\mathrm{I}$--$\mathrm{III}$, we obtain
\begin{align}\label{03291137}
    \nonumber\frac{1}{2}\langle(|&u_{t+T}|^2+|v_{t+T}|^2),|w_{t+T}|^2\rangle-\frac{1}{2}\langle(|u_{ T}|^2+|v_{ T}|^2),|w_{ T}|^2\rangle\\\le &\nonumber-2\alpha\int_{ T}^{t+T}\int_{\mathbb{R}} |w_r|^2(|u_r|^2+|v_r|^2)\dd x\dd r \\&+\nonumber \mathrm{Re}\int_{ T}^{t+T}\int_{\mathbb{R}}i (|u_r|^2+|v_r|^2)\overline{w}_r   \mathsf{Q}_N(\partial_x^2 w_r)\dd x\dd r\\&-\nonumber \mathrm{Re}\int_{ T}^{t+T}\hspace{-1.5mm}\int_{\mathbb{R}} (|u_r|^2+|v_r|^2)\overline{w}_r i\mathsf{Q}_N(u_rv_r\overline{w}_r)\dd x\dd r\\&+\mathrm{Re}\sum_{j=1}^\infty b_j\int_{ T}^{t+T}\hspace{-1.5mm}\int_{\mathbb{R}} (\overline{u}_r+\overline{v}_r)\nonumber |w_r|^2e_j\dd x\big(\dd \beta^1_j(r)+i\dd \beta^2_j(r)\big)\\&+C\int_T^{t+T}\|w_r\|_{L^\infty}^2\mathcal{L}^{\frac{1}{2}}(r,u_r,v_r)\dd r+C\int_{ T}^{t+T}\!\!\!\|w_r\|^2\dd r.
\end{align}
\noindent{\it Step 3: third term}. By the product rule,
\begin{align}\label{03291100}
    \nonumber\frac{1}{2}\langle u_{t+T}&v_{t+T},w_{t+T}^2\rangle-\frac{1}{2}\langle u_{ T}v_{ T},w_{ T}^2\rangle\\=&\ \nonumber\int_{ T}^{t+T}\langle u_rv_r, w_r\dd w_r\rangle+\frac{1}{2}\int_{ T}^{t+T}\langle w_r^2, v_r\dd u_r\rangle+\frac{1}{2}\int_{ T}^{t+T}\langle w_r^2, u_r\dd v_r\rangle\\&\quad +\frac{1}{2}\int_{ T}^{t+T}\langle w_r^2, \dd \langle\!\langle u,v\rangle\!\rangle_r\rangle=:\mathrm{I^\prime+II^\prime+III^\prime+IV^\prime},
\end{align}
where $\langle\!\langle u,v\rangle\!\rangle_\cdot:\mathbb{R}_+\times\Omega\rightarrow L^2$ is the cross-variation of $u$ and $v$; since $u$ and $v$ are driven by the same noise $bW$,
\begin{align*}
    \langle\!\langle u,v\rangle\!\rangle_t=\langle\!\langle bW\rangle\!\rangle_t=t\sum_{j=1}^\infty b_j^2e_j^2+t\sum_{j=1}^\infty b_j^2(ie_j)^2=0.
\end{align*}
The latter implies that
\begin{align*}
\mathrm{IV}^\prime=0.
\end{align*}We treat the terms $\mathrm{I}^\prime$--$\mathrm{III}^\prime$ as in Step~2.
For the term $\mathrm{I}^\prime$, we substitute \eqref{02230853-1}
 and use \eqref{08301619}:
\begin{align*} 
    \nonumber\mathrm{I}^\prime=&\ \mathrm{Re}\int_{ T}^{t+T}\hspace{-1mm}\int_{\mathbb{R}} u_rv_r\overline{w}_r\big(-\alpha \overline{w}_r-i\mathsf{Q}_N(\partial_x^2\overline{w}_r)+i\mathsf{Q}_N\big((|u_r|^2+|v_r|^2)\overline{w}_r\big)\\&+i\mathsf{Q}_N(\overline{u}_r\overline{v}_r{w}_r)\big)\dd x\dd r\nonumber\\=& -\alpha\mathrm{Re}\int_{ T}^{t+T}\hspace{-2mm}\int_{\mathbb{R}}u_rv_r\overline{w}_r^2\dd x\dd r-\mathrm{Re}\int_{ T}^{t+T}\langle i\mathsf{Q}_N(u_rv_r\overline{w}_r),\partial_x^2{w}_r\rangle\dd r\nonumber\\&+ \mathrm{Re}\int_{ T}^{t+T}\hspace{-2mm}\int_{\mathbb{R}}(|u_r|^2+|v_r|^2)\overline{w}_ri\mathsf{Q}_N(u_rv_r\overline{w}_r)\dd x\dd r.\nonumber
\end{align*}
Applying It\^o's formula and using \eqref{07200146}, we obtain
\begin{align*} 
    \nonumber\mathrm{II}^\prime&=\frac{1}{2}\mathrm{Re}\int_{ T}^{t+T}\hspace{-1mm}\int_{\mathbb{R}} v_r\overline{w}_r^2\big(-\alpha u_r+i\partial_x^2u_r-i|u_r|^2u_r\big)\dd x\dd r\\&\quad +\frac{1}{2}\mathrm{Re}\sum_{j=1}^\infty b_j\int_{ T}^{t+T}\hspace{-1mm}\int_{\mathbb{R}}
    v_r\overline{w}_r^2e_j\dd x(\dd \beta_j^1(r)+i\dd \beta_j^2(r))\nonumber\\&\leq-\frac{\alpha}{2}\mathrm{Re}\int_{ T}^{t+T}\hspace{-2mm}\int_{\mathbb{R}}u_rv_r\overline{w}_r^2\dd x\dd r+\frac{1}{2}\int_{ T}^{t+T}\|w_r\|_{L^\infty}^2\|v_r\|(\|\partial_x^2u_r\|+\|u_r\|_{L^6}^3)\dd r\nonumber\\&\quad +\frac{1}{2}\mathrm{Re}\sum_{j=1}^\infty b_j\int_{ T}^{t+T}\hspace{-1mm}\int_{\mathbb{R}}
    v_r\overline{w}_r^2e_j\dd x(\dd \beta_j^1(r)+i\dd \beta_j^2(r))\nonumber\\&\leq -\frac{\alpha}{2}\mathrm{Re}\int_{ T}^{t+T}\hspace{-2mm}\int_{\mathbb{R}}u_rv_r\overline{w}_r^2\dd x\dd r+C\int_{ T}^{t+T}\|w_r\|_{L^\infty}^2\mathcal{L}^{\frac{1}{2}}(r,u_r,v_r)\dd r\nonumber\\&\quad +\frac{1}{2}\mathrm{Re}\sum_{j=1}^\infty b_j\int_{ T}^{t+T}\hspace{-1mm}\int_{\mathbb{R}}
    v_r\overline{w}_r^2e_j\dd x(\dd \beta_j^1(r)+i\dd \beta_j^2(r)).
\end{align*}
Similarly,  
\begin{align*} 
 \nonumber\mathrm{III}^\prime&=\frac{1}{2}\mathrm{Re}\sum_{j=1}^\infty b_j\int_{ T}^{t+T}\int_{\mathbb{R}}
    u_r\overline{w}_r^2e_j\dd x(\dd \beta_j^1(r)+i\dd \beta_j^2(r))\\&\quad+\frac{1}{2}\mathrm{Re}\int_{ T}^{t+T}\int_{\mathbb{R}} u_r\overline{w}_r^2\nonumber  \big(-\alpha v_r+i\mathsf{Q}_N(\partial_x^2 v_r)-i\mathsf{Q}_N(|v_r|^2v_r)\\&\quad-i\mathsf{P}_N(|u_r|^2u_r)+i\mathsf{P}_N(\partial_x^2u_r)\big)\dd x\dd r\nonumber\\&\leq \frac{1}{2}\mathrm{Re}\sum_{j=1}^\infty b_j\int_{ T}^{t+T}\int_{\mathbb{R}}
    u_r\overline{w}_r^2e_j\dd x(\dd \beta_j^1(r)+i\dd \beta_j^2(r))\nonumber\\&\quad-\frac{\alpha}{2}\mathrm{Re}\int_{ T}^{t+T}\int_{\mathbb{R}}u_rv_r\overline{w}_r^2\dd x\dd r\nonumber \\&\quad+ C\int_{ T}^{t+T}\|w_r\|_{L^\infty}^2\|u_r\|(\|\partial_x^2u_r\|+\|\partial_x^2v_r\|+\|u_r\|_{L^6}^3+\|v_r\|_{L^6}^3)\dd r\nonumber\\&\leq \frac{1}{2}\mathrm{Re}\sum_{j=1}^\infty b_j\int_{ T}^{t+T}\int_{\mathbb{R}}
    u_r\overline{w}_r^2e_j\dd x(\dd \beta_j^1(r)+i\dd \beta_j^2(r))\nonumber\\&\quad-\frac{\alpha}{2}\mathrm{Re}\int_{ T}^{t+T}\int_{\mathbb{R}}u_rv_r\overline{w}_r^2\dd x\dd r+ C\int_{ T}^{t+T}\|w_r\|_{L^\infty}^2\mathcal{L}^{\frac{1}{2}}(r,u_r,v_r)\dd r.\nonumber
\end{align*}
Substituting the estimates for $\mathrm{I}^\prime$--$\mathrm{IV}^\prime$  into \eqref{03291100}, we derive
\begin{align}\label{03291138}
  &  \nonumber\frac{1}{2}\langle u_{t+T}v_{t+T},w_{t+T}^2\rangle-\frac{1}{2}\langle u_{ T}v_{ T},w_{ T}^2\rangle\\\nonumber&\leq-2\alpha\mathrm{Re}\int_{T}^{t+T} \int_{\mathbb{R}} u_r v_r\overline{w}_r^2\dd x\dd r-\mathrm{Re}\int_{T}^{t+T}\langle i\mathsf{Q}_N(u_r v_r\overline{w}_r),\partial_x^2{w}_r\rangle \dd r\\&\quad +\nonumber \mathrm{Re} \int_{T}^{t+T} \int_{\mathbb{R}} (|u_r|^2 + |v_r|^2)\overline{w}_ri\mathsf{Q}_N(u_r v_r\overline{w}_r)\dd x\dd r\\& \quad+\nonumber\frac{1}{2}\mathrm{Re}\sum_{j=1}^\infty b_j\int_{T}^{t+T}\int_{\mathbb{R}}
    (u_r+v_r)\overline{w}_r^2e_j\dd x(\dd \beta_j^1(r)+i\dd \beta_j^2(r))\\&\quad +C\int_{T}^{t+T}\|w_r\|_{L^\infty}^2\mathcal{L}^{\frac{1}{2}}(r,u_r,v_r)\dd r.
\end{align}
\noindent{\it Step 4: a pathwise estimate}.
In this step we prove that there exist $N_\varepsilon\ge 1$ and $T_\varepsilon>0$ such that, almost surely, for all~$t\geq0$,
\begin{align}\label{07210028}
    \nonumber  e^{\alpha t-\varepsilon\int_T^{t+T}\mathcal{L}(s,u_s,v_s)\dd s}&\widetilde{\mathcal{M}}(t+T,w_{t+T})-\widetilde{\mathcal{M}}(T,w_T)\\&\leq\nonumber  C_\varepsilon\|w_T\|^2+C_{N,\varepsilon}\|\mathsf{P}_Nw_T\|^2\\&\quad+\int_T^{t+T}e^{\alpha(r-T)-\varepsilon\int_T^r\mathcal{L}(s,u_s,v_s)\dd s}\dd M_r^T,
\end{align}
where  
\begin{align*}
    \nonumber M^T_{t}=&\ \frac{1}{2}\mathrm{Re}\sum_{j=1}^\infty b_j\int_{T}^{t}\hspace{-1mm}\int_\mathbb{R}
    (u_r+v_r)\overline{w}_r^2e_j\dd x\big(\dd \beta_j^1(r)+i\dd \beta_j^2(r)\big) \\&+\mathrm{Re}\sum_{j=1}^\infty b_j\int_{T}^{t}\hspace{-1.5mm}\int_\mathbb{R}(\overline{u}_r+\overline{v}_r)|w_r|^2e_j\dd x\big(\dd \beta^1_j(r)+i\dd \beta^2_j(r)\big).
\end{align*}
We first combine \eqref{03291134},  \eqref{03291137}, and \eqref{03291138} to obtain
\begin{align}\label{03300316}
    \nonumber &\mathcal{M}(t+T,w_{t+T})-\mathcal{M}( T,w_{ T})\leq -2\alpha \int_{ T}^{t+T}\mathcal{M}(r,w_r)\dd r\\&\nonumber -\alpha \int_{ T}^{t+T}\hspace{-1mm}\int_{\mathbb{R}} |w_r|^2(|u_r|^2+|v_r|^2)\dd x\dd r-\alpha\mathrm{Re}\int_{ T}^{t+T}\hspace{-1mm}\int_{\mathbb{R}} u_r v_r\overline{w}_r^2\dd x\dd r\\&+C\int_{T}^{t+T}\|w_r\|_{L^\infty}^2\mathcal{L}^{\frac{1}{2}}(r,u_r,v_r)\dd r+C\int_T^{t+T}\|w_r\|^2\dd r+M^{T}_{t+T}.
\end{align}
By \eqref{07200228}, the sum of the second and third terms on the
right-hand side is non-positive; namely,
\begin{align*}
    -\alpha\int_T^{t+T}\int_\mathbb{R}|w_r|^2(|u_r|^2+|v_r|^2)\dd x\dd r-\alpha\mathrm{Re}\int_T^{t+T}\int_{\mathbb{R}}u_rv_r\overline{w}_r^2\dd x\dd r\leq 0.
\end{align*}
Combining this with \eqref{03300316}, \eqref{07200244} and Young's inequality, we obtain that, for any~$\varepsilon>0$, there exists $C_\varepsilon>0$ such that
\begin{align}\label{03300349}
\nonumber \mathcal{M}(t+T,w_{t+T})-\mathcal{M}( T,w_{ T}) &\leq -2\alpha\int_{ T}^{t+T}\mathcal{M}(r,w_r)\dd 
r\\&\quad\nonumber+\frac{\varepsilon}{6}\int_{ T}^{t+T} \|w_r\|_{H^1}^2\mathcal{L}(r,u_r,v_r)\dd r\\&\quad+C_\varepsilon\int_{ T}^{t+T}\|w_r\|^2\dd r+M^T_{t+T}.
\end{align}
Applying the product rule to
\begin{align*}
    e^{\alpha t-\varepsilon\int_T^{t+T}\hspace{-0.5mm}\mathcal{L}(s,u_s,v_s)\dd s}\mathcal{M}(t+T,w_{t+T}),
\end{align*}
we get 
\begin{align}\label{06250504}
    \nonumber & e^{\alpha t-\varepsilon\int_{T}^{t+T} \mathcal{L}(s,u_s,v_s)\dd s}\mathcal{M}(t+T,w_{t+T})-\mathcal{M}(T,w_T)\\\nonumber&\leq  C_\varepsilon\int_T^{t+T} e^{\alpha (r-T)-\varepsilon\int_T^{r} \mathcal{L}(s,u_s,v_s)\dd s}\|w_r\|^2\dd r \nonumber\\&\nonumber\quad+\int_T^{t+T}  e^{\alpha (r-T)-\varepsilon\int_T^{r}\mathcal{L}(s,u_s,v_s)\dd s}\dd M_{r}^T\\\nonumber&\quad+\frac{\varepsilon}{3}\int_T^{t+T}e^{\alpha (r-T)-\varepsilon\int_T^{r} \mathcal{L}(s,u_s,v_s)\dd s}\widetilde{\mathcal{M}}(r,w_r)\mathcal{L}(r,u_r,v_r)\dd r\\&\quad-\varepsilon \int_T^{t+T}e^{\alpha (r-T)-\varepsilon\int_T^{r} \mathcal{L}(s,u_s,v_s)\dd s}{\mathcal{M}}(r,w_r)\mathcal{L}(r,u_r,v_r)\dd r.
\end{align}
To estimate the first term on the right-hand side, we obtain from
Lemma~\ref{03011027} that for any $\varepsilon_1>0$, $T\geq T^0_{\varepsilon_1}$,
and $N\geq N_{\varepsilon_1}^0$,\begin{align*}
   \nonumber \frac{1}{2}\|w_{t+T}\|^2+ \frac{3\alpha}{4} \int_{ T}^{t+T}\|w_r\|^2\dd r & \leq \frac{1}{2}\|w_{ T}\|^2 +  C_{N,\varepsilon_1} \int_{ T}^{t+T}\|\mathsf{P}_Nw_r\|^2\dd r\\&\quad+  \varepsilon_1\int_{ T}^{t+T}\|w_r\|_{H^1}^2 \mathcal{L}(r,u_r,v_r)\dd r.
\end{align*}
  Applying the product rule to
   \begin{align*}
     \frac{1}{2}e^{\alpha t-\varepsilon\int_T^{t+T}\mathcal{L}(s,u_s,v_s)\dd s}\|w_{t+
T}\|^2,
   \end{align*}
  we obtain
\begin{align}\label{06250608}
    \nonumber \frac{1}{2}&e^{\alpha t-\varepsilon\int_T^{t+T} \mathcal{L}(s,u_s,v_s)\dd s}\|w_{t+T}\|^2+\frac{\alpha}{4}\int_T^{t+T} e^{\alpha (r-T)-\varepsilon\int_{T}^{r}\mathcal{L}(s,u_s,v_s)\dd s}\|w_{r}\|^2\dd r\\&\leq \nonumber \frac{1}{2}\|w_T\|^2+C_{N,\varepsilon_1}\int_{T}^{t+T} e^{\alpha (r-T)}\|\mathsf{P}_Nw_r\|^2\dd r\\&\quad\nonumber+2\varepsilon_1\int_T^{t+T} e^{\alpha (r-T)-\varepsilon\int_{T}^{r}\mathcal{L}(s,u_s,v_s)\dd s}\widetilde{\mathcal{M}}(r,w_r)\mathcal{L}(r,u_r,v_r)\dd r\\&\quad-\frac{\varepsilon}{2}\int_T^{t+T}e^{\alpha(r-T)-\varepsilon\int_T^r\mathcal{L}(s,u_s,v_s)\dd s}\|w_r\|^2\mathcal{L}(r,u_r,v_r)\dd r.
\end{align}
In particular, this inequality implies that  
\begin{align*}
     \int_{T}^{t+T}&e^{\alpha (r-T)-\varepsilon\int_{T}^{r} \mathcal{L}(s,u_s,v_s)\dd s}\|w_r\|^2\dd r\\&\leq  \frac{2 }{\alpha}\|w_T\|^2+C_{N,\varepsilon_1}\int_T^{t+T} e^{\alpha (r-T)}\|\mathsf{P}_Nw_r\|^2\dd r\\&\quad+\frac{8\varepsilon_1}{\alpha}\int_{T}^{t+T}e^{\alpha (r-T)-\varepsilon\int_{T}^{r}\mathcal{L}(s,u_s,v_s)\dd s}\widetilde{\mathcal{M}}(r,w_r)\mathcal{L}(r,u_r,v_r)\dd r.
\end{align*}
Combining this with \eqref{06250504}, we get
\begin{align*}
    & e^{\alpha t-\varepsilon\int_T^{t+T}\mathcal{L}(s,u_s,v_s)\dd s}\mathcal{M}(t+T,w_{t+T})-\mathcal{M}(T,w_T)\\&\leq  \frac{2C_\varepsilon}{\alpha}\|w_T\|^2+C_{N,\varepsilon_1,\varepsilon}\int_T^{t+T}e^{\alpha (r-T)}\|\mathsf{P}_Nw_r\|^2\dd r\\&\quad+\int_T^{t+T}  e^{\alpha (r-T)-\varepsilon\int_T^{r}\mathcal{L}(s,u_s,v_s)\dd s}\dd M_{r}^T\\&\quad+\big(\frac{\varepsilon}{3}+\frac{8C_\varepsilon\varepsilon_1}{\alpha}\big)\int_{T}^{t+T} e^{\alpha (r-T)-\varepsilon\int_{T}^{r}\mathcal{L}(s,u_s,v_s)\dd s}\widetilde{\mathcal{M}}(r,w_r)\mathcal{L}(r,u_r,v_r)\dd r\\&\quad-\varepsilon\int_{T}^{t+T}e^{\alpha (r-T)-\varepsilon\int_{T}^{r}\mathcal{L}(s,u_s,v_s)\dd s}{\mathcal{M}}(r,w_r)\mathcal{L}(r,u_r,v_r)\dd r.
\end{align*}
Summing this and \eqref{06250608}, we obtain
\begin{align*}
    &\nonumber  e^{\alpha t-\varepsilon\int_T^{t+T}\mathcal{L}(s,u_s,v_s)\dd s}\widetilde{\mathcal{M}}(t+T,w_{t+T})-\widetilde{\mathcal{M}}(T,w_T)\\&\leq\nonumber  (\frac{2C_\varepsilon}{\alpha}+\frac{1}{2})\|w_T\|^2+C_{N,\varepsilon_1,\varepsilon}\int_T^{t+T}e^{\alpha (r-T)}\|\mathsf{P}_Nw_r\|^2\dd r\\&\quad\nonumber+\hspace{-0.5mm}\big(\frac{8C_\varepsilon\varepsilon_1}{\alpha}\hspace{-0.5mm}+\hspace{-0.5mm}2\varepsilon_1\hspace{-0.5mm}-\hspace{-0.5mm}\frac{2\varepsilon}{3}\big)\int_{T}^{t+T}\hspace{-1.5mm}e^{\alpha (r-T)-\varepsilon\int_{T}^{r}\mathcal{L}(s,u_s,v_s)\dd s}\widetilde{\mathcal{M}}(r,w_r)\mathcal{L}(r,u_r,v_r)\dd r\\&\quad+\int_T^{t+T}e^{\alpha(r-T)-\varepsilon\int_T^r\mathcal{L}(s,u_s,v_s)\dd s}\dd M_r^T.
\end{align*}
Choosing
\begin{align*}
    \varepsilon_1:=\frac{\varepsilon\alpha}{24C_\varepsilon}\wedge \frac{\varepsilon}{6},
\end{align*}
 defining $T_\varepsilon:=T^0_{\varepsilon_1}$ and $N_\varepsilon:=N_{\varepsilon_1}^0$, and using \eqref{03190251}, we arrive at \eqref{07210028}.

\noindent{\it Step 5: conclusion}.
 For any $n\in{\mathbb{N}}$, by taking $t:=(\tau\wedge n-T)_+$ in \eqref{07210028}, we obtain that
\begin{align}\label{07210028-1}
    \nonumber  &e^{\alpha (\tau\wedge n-T)_+-\varepsilon\int_T^{T\vee(\tau\wedge n)}\mathcal{L}(s,u_s,v_s)\dd s}\widetilde{\mathcal{M}}(T\vee(\tau\wedge n),w_{T\vee(\tau\wedge n)})\\&\hspace{3cm}\leq\nonumber  \widetilde{\mathcal{M}}(T,w_T)+C_\varepsilon\|w_T\|^2+C_{N,\varepsilon}\|\mathsf{P}_Nw_T\|^2\\&\hspace{3.5cm}+\int_T^{T\vee(\tau\wedge n)}e^{\alpha(r-T)-\varepsilon\int_T^r\mathcal{L}(s,u_s,v_s)\dd s}\dd M_r^T.
\end{align}
The martingale term on the right-hand side of  this inequality is integrable. Indeed, Lemmas \ref{02280441}(ii), \ref{02280436}(ii), \ref{WPA2} and It\^o's isometry imply that
\begin{align}\label{07210154}
    &\nonumber\mathbb{E}\big[\big|\int_T^{T\vee(\tau\wedge n)}\hspace{-1mm}e^{\alpha(r-T)-\varepsilon\int_T^r\mathcal{L}(s,u_s,v_s)\dd s}\dd M_r^T\big|^2\big]\leq\hspace{1mm} e^{2\alpha(T\vee n-T)}\mathbb{E}\big[\langle M^T\rangle_{T\vee n}\big]\\&\qquad\leq\hspace{1mm} Ce^{2\alpha n}\sum_{j=1}^\infty b_j^2\|e_j\|_{H^2}^2\mathbb{E}\big[\int_T^{T\vee n}(\|u_r\|_{H^1}^6+\|v_r\|_{H^1}^6)\dd r\big]<\infty.
\end{align}
Taking expectations in \eqref{07210028-1}, we obtain from Lemmas \ref{02280441}(ii), \ref{02280436}(ii), and~\ref{WPA2} that
\begin{align*}
    &\mathbb{E}\big[e^{\alpha(\tau\wedge n-T)_+-\varepsilon\int_T^{T\vee(\tau\wedge n)}\mathcal{L}(s,u_s,v_s)\dd s}\widetilde{\mathcal{M}}(T\vee(\tau\wedge n),w_{T\vee(\tau\wedge n)})\big]\\&\qquad\leq \hspace{1mm}C_\varepsilon\mathbb{E}\big[\|w_T\|^2\big]+C_{N,\varepsilon}\mathbb{E}\big[\|\mathsf{P}_Nw_T\|^2\big]+\mathbb{E}\big[\widetilde{\mathcal{M}}(T,w_T)\big]<\infty.
\end{align*}
Since $\tau$ is a finite stopping time, by letting $n\rightarrow\infty$ and using Fatou's lemma, we finish proving \eqref{07210127}.\par  Notice that $\big\{\tau\wedge n\geq T\big\}\in\mathcal{F}_T$. Based on \eqref{07210127} and \eqref{07210154}, we take the conditional expectation $\mathbb{E}[\hspace{0.5mm}\cdot\hspace{0.5mm}|\mathcal{F}_T]$ on both sides of \eqref{07210028-1} and restrict it to the event $\big\{\tau\wedge n\geq T\big\}$. Then we arrive at
\begin{align*}
     \mathbb{E} \big[&e^{\alpha (\tau\wedge n-T) -\varepsilon\int_{ T}^{\tau\wedge n}\mathcal{L}(s,u_s,v_s)\dd s}\widetilde{\mathcal{M}}(\tau\wedge n,w_{\tau\wedge n})\big|\mathcal{F}_{ T}\big]\mathbf{1}_{\tau
     \wedge n\geq  T}\\&\qquad\leq \big(C_{\varepsilon}\|w_T\|^2+C_{N,\varepsilon}\|\mathsf{P}_Nw_T\|^2+\widetilde{\mathcal{M}}(T,w_{T})\big)\mathbf{1}_{\tau\wedge n\geq T}\\&\qquad\leq \big(C_{N,\varepsilon}\|w_T\|^2+\widetilde{\mathcal{M}}(T,w_{T})\big)\mathbf{1}_{\tau\wedge n\geq T}.
\end{align*}
Letting $n\to\infty$ and applying Fatou's lemma proves \eqref{04020327}.
\end{proof}
The following estimate complements Proposition~\ref{fpestimate}: it is
valid for every~$N\ge1$ and from time $0$, at the price of a fixed
constant in place of the small parameter $\varepsilon$. Both estimates will be applied in the proof of Proposition~\ref{03100846}.
\begin{proposition}\label{03100330}
    There exists a constant $\mathcal{C}_2>0$ such that for any $N\ge1$ and $t\geq 0$, we have
    \begin{align*}
    &\mathbb{E}\big[e^{\alpha t-\mathcal{C}_2\int_{0}^{t}\mathcal{L}(s,u_s,v_s)\dd s}\widetilde{\mathcal{M}}(t,w_{t})\big]\leq  {\widetilde{\mathcal{M}}}(0,w_{0})+C_N\|w_0\|^2.
\end{align*}
\end{proposition}
\begin{proof}
Since the derivation of~\eqref{03300349} depends neither on the lower bound
on~$N$ nor on that on~$T$, we may set $\varepsilon=1$ and $T=0$ in it; using
that $\mathcal{L}\geq1$, we obtain   \begin{align}\label{05041049}
       \nonumber {\mathcal{M}}(t,w_{t})& - {\mathcal{M}}(0,w_{0})\\\leq &-2\alpha\int_0^t{\mathcal{M}}(r,w_r)\dd r+C\int_0^t\|w_r\|_{H^1}^2\mathcal{L}(r,u_r,v_r)\dd r+M_t^0\nonumber\\\leq & -2\alpha\int_0^t\mathcal{M}(r,w_r)\dd r+C\int_0^t\widetilde{\mathcal{M}}(r,w_r)\mathcal{L}(r,u_r,v_r)\dd r+M_t^0,
   \end{align}
    where the last inequality follows from \eqref{06190211}. To estimate
the $L^2$-norm of $w$, we return to the decomposition \eqref{03070318}
(with $T=0$) and bound the terms~$\mathrm{II}$ and $\mathrm{III}$
directly, without splitting off the low frequencies: by the Sobolev
embedding $H^1\hookrightarrow L^\infty$ and H\"older's inequality,
   \begin{align*}
       \mathrm{II}+\mathrm{III}\leq C\int_0^t(\|u_r\|_{H^1}^2+\|v_r\|_{H^1}^2)\|w_r\|^2\dd r\leq C\int_0^t\|w_r\|^2\mathcal{L}(r,u_r,v_r)\dd r.
   \end{align*}
   Combining this with \eqref{03070318} and \eqref{04060303}, we obtain that
   \begin{align*}
       \frac{1}{2}{\|w_t\|^2}-\frac{1}{2}{\|w_0\|^2}\leq& -\frac{\alpha}{2}\int_{0}^t\|w_r\|^2\dd r+C_N\int_0^t\|\mathsf{P}_Nw_r\|^2\dd r\\&+ C\int_0^t\mathcal{L}(r,u_r,v_r)\|w_r\|^2\dd r.
   \end{align*}
 The latter and \eqref{05041049} imply that 
 \begin{align*}
     \widetilde{\mathcal{M}}(t,w_t)-\widetilde{\mathcal{M}}(0,w_0)\leq& -\alpha \int_0^t\widetilde{\mathcal{M}}(s,w_s)\dd s+\mathcal{C}_2\int_0^t\widetilde{\mathcal{M}}(r,w_r)\mathcal{L}(r,u_r,v_r)\dd r\\&+C_N\int_0^t\|\mathsf{P}_Nw_r\|^2\dd r+M_t^0,
 \end{align*}
 for some $\mathcal{C}_2>0$. Applying the product rule to
   \begin{align*}
       e^{\alpha t-\mathcal{C}_2\int_0^t\mathcal{L}(s,u_s,v_s)\dd s}\widetilde{\mathcal{M}}(t,w_t),
   \end{align*}
   we find that for any $t\geq 0$,
 \begin{align}\label{06250355}
     \nonumber e^{\alpha t-\mathcal{C}_2\int_0^t\mathcal{L}(s,u_s,v_s)\dd s}&\widetilde{\mathcal{M}}(t,w_t)-\widetilde{\mathcal{M}}(0,w_0)\\\leq&\nonumber \hspace{1mm} C_N\int_0^te^{\alpha r}\|\mathsf{P}_Nw_r\|^2\dd r+\int_0^t e^{\alpha r-\mathcal{C}_2\int_0^r\mathcal{L}(s,u_s,v_s)\dd s}\dd M_r^0\\\leq&\hspace{1mm} C_N\|\mathsf{P}_Nw_0\|^2+\int_0^t e^{\alpha r-\mathcal{C}_2\int_0^r\mathcal{L}(s,u_s,v_s)\dd s}\dd M_r^0,
 \end{align}
 where the last inequality follows from \eqref{03190251}. Thus taking the expectation in \eqref{06250355}, we complete the proof.   
\end{proof}
The following estimate for the process $w$ plays a central role in
controlling the probability that the coupling fails (Section~\ref{S:3}). To formulate it, we introduce the functional
$$
F(u):=\|u\|^{20}+\mathcal{H}^{10}(u)+\mathcal{G}^{2}(u).
$$
We shall use the stopping time $\tau_1^u$ defined in \eqref{04080050} below,
with the constants $K$ and $\mathscr{C}$ fixed in
\eqref{04271005-1} and \eqref{04271005-2} and with a parameter $\rho\ge\rho_*$
as in~\eqref{04271005-3}; the constants in the estimates below do not depend
on~$\rho$, unless this is indicated by a subscript. We denote by $\tau^v_1$
the analogue of $\tau_1^u$ for the auxiliary process $v$, with the same
constants and the same parameter~$\rho$, and we set
\begin{align}\label{06230835-1}
    \tau_1^{u,v}:=\tau_1^u\wedge\tau_1^v.
\end{align}
\begin{proposition}\label{03100846}
There exist $\mathcal{C}_3>0$ and an integer $N_0\ge1$ such that, for
$N=N_0$ in \eqref{auxiliary}, any $\rho\ge\rho_*$, and any
$u_0,u_0^\prime\in H^2$, we have
$$
    \mathbb{E}\Big[\!\int_0^{\tau_1^{u,v}}\!\!\big(1+\mathcal{H}(u_s)
    +\mathcal{H}(v_s)\big)\widetilde{\mathcal{M}}(s,w_s)\dd s\Big]
    \leq \mathcal{C}_3\widetilde{\mathcal{M}}(0,w_0)\,
    e^{\mathcal{C}_3\rho+\mathcal{C}_3\mathscr{C}(F(u_0)+F(u^\prime_0))}.
$$
\end{proposition}
\begin{proof}
 \noindent{\it Step 1: Reduction}.
By the definitions \eqref{06230835-1} and
\eqref{04080050} of the stopping times, we have
\begin{align*}
    \mathcal{H}^2(u_s)\leq\rho+Ks+\mathscr{C}\mathcal{H}^2(u_0),
    \qquad 0\le s\leq\tau_1^{u,v},
\end{align*}
and similarly for $v$, with $u_0^\prime$ in place of $u_0$. Since
$\mathcal{H}^2\leq C(1+F)$ and $\rho\geq1$, it follows that
\begin{align}\label{E:R-bound}
    1+\mathcal{H}(u_s)+\mathcal{H}(v_s)\leq Cg(s),
    \qquad 0\le s\leq\tau_1^{u,v},
\end{align}  where 
   $$ 
   g(s):=\big(\rho+Ks+\mathscr{C}(F(u_0)+F(u_0^\prime))\big)^{\frac12}.
   $$
Writing $\int_0^{\tau_1^{u,v}}=\int_0^\infty\mathbf{1}_{s\le\tau_1^{u,v}}
\dd s$ and applying Tonelli's theorem on $\Omega\times\mathbb{R}_+$ to the
non-negative function
$(\omega,s)\mapsto\mathbf{1}_{s\le\tau_1^{u,v}(\omega)}\,g(s)\,
\widetilde{\mathcal{M}}(s,w_s(\omega))$, we
obtain from \eqref{E:R-bound} that
\begin{align}\label{07312348}
\nonumber\mathbb{E}\Big[\!\int_0^{\tau_1^{u,v}}\!\!\!\!\!\big(1\!+\!\mathcal{H}(u_s)
    \!+\!\mathcal{H}(v_s)\big)\widetilde{\mathcal{M}}(s,w_s)\dd s\Big]
   & \leq C\mathbb{E}\Big[\int_0^{\infty}\!\!\!\!
    \mathbf{1}_{s\le\tau_1^{u,v}}g(s)
    \widetilde{\mathcal{M}}(s,w_s)\dd s\Big]\\
    &= C\!\int_0^\infty\!\!\! g(s)
    \mathbb{E}\big[\widetilde{\mathcal{M}}(s,w_s);
    \tau_1^{u,v}\geq s\big]\dd s.
\end{align}
 Thus it suffices to establish a
moment estimate for $\widetilde{\mathcal{M}}(s,w_s)$ on
$\{\tau_1^{u,v}\geq s\}$, for any $s\geq0$.

\noindent{\it Step 2: Moment estimate for $\widetilde{\mathcal{M}}(s,w_s)$ on $\{{\tau_1^{u,v}}\geq s\}$}. In this step, we prove the existence of $\widetilde{\mathcal{C}}_3>0$ such that
\begin{equation}\label{07312337}
    \mathbb{E}\big[\widetilde{\mathcal{M}}(s,w_s);{\tau_1^{u,v}}\geq s\big]\leq C e^{-\frac{3\alpha s}{4}+\widetilde{\mathcal{C}}_3(\rho+\mathscr{C}(F(u_0)+F(u_0^\prime)))}\widetilde{\mathcal{M}}(0,w_0),\qquad   s\geq0.
\end{equation}
We first note that, by the definitions \eqref{06230835-1} and \eqref{04080050} of the stopping times, there exists $\mathcal{C}_1>0$ such that
\begin{align}\label{04261050}
    \int_0^s\mathcal{L}(r,u_r,v_r)\dd r\leq \mathcal{C}_1\big(\rho+Ks+\mathscr{C}(F(u_0)+F(u^\prime_0)) \big)
\end{align}for  $0\le s\leq \tau_1^{u,v}$.
   Recall that $w=u-v$. Choosing
\begin{align}\label{07312311}
    \varepsilon=\frac{\alpha}{4\mathcal{C}_1K},
\end{align}
we obtain from Proposition~\ref{fpestimate} an integer $N_0\geq1$ and a time
$T_0:=T_\varepsilon>0$ such that \eqref{07210127} and \eqref{04020327} hold
with $N=N_0$ and $T=T_0$. By \eqref{04261050} and \eqref{07312311}, this
choice of $\varepsilon$ gives
\begin{align}\label{E:eps-L}
    \varepsilon\int_0^s\mathcal{L}(r,u_r,v_r)\dd r
    \leq\frac{\alpha}{4}s+\frac{\alpha}{4K}
    \big(\rho+\mathscr{C}(F(u_0)+F(u_0^\prime))\big)
\end{align}for  $0\le s\leq \tau_1^{u,v}$,
so that this term is strictly dominated by the factor $\alpha(s-T_0)$
appearing in the exponent in \eqref{04020327}; the effective decay rate
$\tfrac{3\alpha}{4}$ in \eqref{05200459} below is what remains after this
compensation.

Let $s\geq T_0$. We apply \eqref{04020327} with $N=N_0$ and $T=T_0$ at the
deterministic time~$s$, multiply both sides by the indicator of the
$\mathcal{F}_{T_0}$-measurable event $\{\tau_1^{u,v}\geq T_0\}$, and take
expectations. Enlarging the event of integration by means of
$\{\tau_1^{u,v}\geq s\}\subseteq\{\tau_1^{u,v}\geq T_0\}$ and bounding the
resulting deterministic factor by \eqref{E:eps-L}, we obtain
\begin{align}\label{04110235-1}
    \nonumber\mathbb{E}\big[\widetilde{\mathcal{M}}(s,w_s);
    \tau_1^{u,v}\geq s\big]
    \leq&\ e^{-\alpha(s-T_0)+\frac{\alpha}{4}s+\frac{\alpha}{4K}
(\rho+\mathscr{C}(F(u_0)+F(u_0^\prime)))}\\
    &\times\mathbb{E}\big[\widetilde{\mathcal{M}}(T_0,w_{T_0})
    +C_{N_0}\|w_{T_0}\|^2;\tau_1^{u,v}\geq T_0\big],
\end{align}where $C_{N_0}$ is the constant $C_{N,\varepsilon}$ from \eqref{04020327},
with $N=N_0$ and $\varepsilon$ as in \eqref{07312311}.
By \eqref{06190211}, \eqref{04261050}, and Proposition~\ref{03100330}
applied at time $T_0$, and dropping the event $\{\tau_1^{u,v}\geq T_0\}$,
the integrand being non-negative, it follows that
\begin{align}\label{05200456}
    \nonumber\mathbb{E}\big[\widetilde{\mathcal{M}}(T_0,w_{T_0})
    +&C_{N_0}\|w_{T_0}\|^2;\tau_1^{u,v}\geq T_0\big]
    \leq C_{N_0}\mathbb{E}\big[\widetilde{\mathcal{M}}(T_0,w_{T_0});
    \tau_1^{u,v}\geq T_0\big]\\
    \nonumber\leq&\ Ce^{-\alpha T_0+\mathcal{C}_2\mathcal{C}_1
    (\rho+KT_0+\mathscr{C}(F(u_0)+F(u^\prime_0)))}\\
    \nonumber&\times\mathbb{E}\big[e^{\alpha T_0-\mathcal{C}_2
    \int_0^{T_0}\mathcal{L}(r,u_r,v_r)\dd r}
    \widetilde{\mathcal{M}}(T_0,w_{T_0})\big]\\
    \leq&\ Ce^{-\alpha T_0+\mathcal{C}_2\mathcal{C}_1
    (\rho+KT_0+\mathscr{C}(F(u_0)+F(u^\prime_0)))}\,
    \widetilde{\mathcal{M}}(0,w_0).
\end{align}
Combining this with \eqref{04110235-1}, and absorbing
into $C$ the factors depending only on $T_0$, we arrive at
\begin{align}\label{05200459}
    \mathbb{E}\big[\widetilde{\mathcal{M}}(s,w_s);\tau_1^{u,v}\geq s\big]
    \leq C e^{-\frac{3\alpha s}{4}
    +(\mathcal{C}_2\mathcal{C}_1+\frac{\alpha}{4K})
    (\rho+\mathscr{C}(F(u_0)+F(u_0^\prime)))}\,
    \widetilde{\mathcal{M}}(0,w_0)
\end{align}for $s\geq T_0$.
Now we turn to the case $s\leq T_0$. Arguing as in \eqref{05200456}, with
$T_0$ replaced by $s$, and using $s\le T_0$ to absorb the term
$\mathcal{C}_2\mathcal{C}_1Ks$ into the constant, we obtain
\begin{align}\label{07260616}
    \mathbb{E}\big[\widetilde{\mathcal{M}}(s,w_s);\tau_1^{u,v}\geq s\big]
    \leq Ce^{-\alpha s+\mathcal{C}_2\mathcal{C}_1
    (\rho+\mathscr{C}(F(u_0)+F(u^\prime_0)))}\,
    \widetilde{\mathcal{M}}(0,w_0)
\end{align}for $0\le s\leq T_0$.
Inequalities
\eqref{05200459} and \eqref{07260616} imply that \eqref{07312337} holds with
$\widetilde{\mathcal{C}}_3:=\mathcal{C}_2\mathcal{C}_1
+\frac{\alpha}{4K}$.

\noindent{\it Step 3: Conclusion}.
By the inequality
\begin{equation*}
    \sqrt{x}\leq Ce^{\frac{\alpha}{4K}x},\qquad x\geq0,
\end{equation*}
applied with $x=\rho+Ks+\mathscr{C}(F(u_0)+F(u_0^\prime))$, the function
$g$ from Step~1 satisfies
\begin{align*}
    g(s)\leq Ce^{\frac{\alpha}{4}s+\frac{\alpha}{4K}
    (\rho+\mathscr{C}(F(u_0)+F(u_0^\prime)))},\qquad s\ge0.
\end{align*}
Combining this with \eqref{07312348} and \eqref{07312337}, we obtain
\begin{align}\label{08010003}
    \nonumber\mathbb{E}\Big[\!\int_0^{\tau_1^{u,v}}\!\!\!\!\!
    \big(1\!+\!\mathcal{H}(u_s)\!+\!\mathcal{H}(v_s)\big)
    &\widetilde{\mathcal{M}}(s,w_s)\dd s\Big]
    \leq  C e^{(\widetilde{\mathcal{C}}_3+\frac{\alpha}{4K})
    (\rho+\mathscr{C}(F(u_0)+F(u_0^\prime)))}
    \\
    \nonumber&\qquad\qquad\qquad\quad\times \widetilde{\mathcal{M}}(0,w_0)\int_0^\infty
    e^{\frac{\alpha}{4}s-\frac{3\alpha}{4}s}\dd s\\
    \leq&\ C e^{(\widetilde{\mathcal{C}}_3+\frac{\alpha}{4K})
    (\rho+\mathscr{C}(F(u_0)+F(u_0^\prime)))}\,
    \widetilde{\mathcal{M}}(0,w_0).
\end{align}
 Setting
$\mathcal{C}_3:=\big(\widetilde{\mathcal{C}}_3+\frac{\alpha}{4K}\big)
\vee C$, with $C>0$ from the last line of \eqref{08010003}, we complete the
proof.
\end{proof}
 
From now on we fix $N=N_0$, where $N_0\ge1$ is as in Proposition~\ref{03100846}, and assume that \eqref{07070547} holds.

\section{Growth estimate for the auxiliary process}\label{S:3}

In this section, we use the Girsanov theorem to estimate the probability
$\mathbb{P}(\tau_1^v<T)$ that the energy functionals of the auxiliary
process~$v$ grow abnormally before time~$T$. The parameter $a$ is fixed later, in the proof of Lemma~\ref{03151059}.
\begin{proposition}\label{03150137}
    There exists $\mathcal{C}_4\geq\mathcal{C}_3$ such that for any $d_0\in(0,1)$, $u_0,u^\prime_0\in B_{H^2}(d_0)$, $a>0$, $\rho\ge\rho_*$, and $T>0$, we have
   $$
        \mathbb{P}(\tau^{{v}}_1<T)\leq \mathcal{C}_4\rho^{-1}+e^{-a}+ \frac{1}{2}\sqrt{\mathrm{exp}\big\{\mathcal{C}_4C^{*}(d_0)\mathrm{exp}\{\mathcal{C}_4\rho+a\}\big\}-1},
$$
where
    $C^{*}(d_0):=(d_0^2+d_0^4)e^{\mathcal{C}_4\mathscr{C}d_0}$.
\end{proposition}
\begin{proof}
For the purposes of this proof we assume, without loss of generality, that
$\Omega=C_0([0,T];H^4)$ and that $\omega_\cdot$ is the coordinate process
under $\mathbb{P}=\mathcal{D}(bW)$. The proof is divided into five steps.
 
\noindent{\it Step 1: reduction}.
If $u_0=u^\prime_0$, then $u$ solves \eqref{auxiliary}, so $v=u=u^\prime$ by
the uniqueness in Lemma~\ref{WPA2}. Taking $q=4$ and $l=0$ in Lemma~\ref{03120055} and using
$\|u_0\|_{H^2}\le d_0<1$, we obtain a constant $\mathcal{C}_4^0>0$, independent of $d_0$ and~$\rho$, such that 
\begin{align}\label{08021543}
    \mathbb{P}\big(\tau_1^{v}<T\big)=\mathbb{P}\big(\tau_1^{u}<T\big)
    \leq\mathbb{P}\big(\tau_1^{u}<\infty\big)\leq\mathcal{C}_4^0\rho^{-1},
\end{align}
which implies the desired result. 

When $u_0\neq u_0^\prime$, we reduce the estimate for $\tau_1^v$ to that of a truncated problem. Namely, let us consider the truncated
processes  $\hat{u}^\prime_\cdot=\hat{u}^\prime(\hspace{0.5mm}\cdot\hspace{0.5mm},\omega)$ and $\hat{v}_\cdot=\hat{v}(\hspace{0.5mm}\cdot\hspace{0.5mm},\omega)$ solving
	\begin{align}
	\left\{
	\begin{aligned}
		&\partial_t \hat{u}^\prime+\alpha \hat{u}^\prime-\big[i\partial_x^2 {\hat{u}}^\prime-i|\hat{u}^\prime|^2\hat{u}^\prime\big]\mathbf{1}_{t\leq\tau^{{u}^\prime}_1}=\mathbf{1}_{t\leq\tau^{u^\prime}_1}\partial_t \omega,\\
		&	\hat{u}^\prime(0,x)={u}^\prime_0(x)\label{trunc eqn 2}
	\end{aligned}	
	\right.
\end{align}
and
	\begin{align}
	\left\{
	\begin{aligned}
		&\partial_t \hat{v}+\alpha \hat{v} -\big[i\partial_{x}^2 \hat{v}-i|\hat{v}|^2\hat{v}\big]\mathbf{1}_{t\leq\tau_1^{v}}=  \big[\mathsf{P}_N(i|\hat{v}|^2\hat{v}-i|{u}|^2{u}\\& \qquad\qquad\qquad\qquad\qquad\qquad-i\partial_x^2(\hat{v}-{u}))+\partial_t \omega\big]\mathbf{1}_{t\leq\tau_1^{v}},\\
		&	\hat{v}(0,x)={u}^\prime_0(x).\label{auxiliary 1} 
	\end{aligned}	
	\right.
\end{align}
Since the truncation in \eqref{auxiliary 1} is inactive on $[0,\tau_1^{v}]$,
we have $\tau_1^{\hat v}=\tau_1^{v}$, and it suffices to estimate
$\mathbb{P}\big(\tau_1^{\hat{v}}<T\big)$. To this end, we introduce the stopping time
\begin{align}\label{04160408}
    \nonumber\tau_0^{u,v}:=& \inf\big\{t\ge 0:\int_0^{t}\big(1+\mathcal{H}(u_s)+\mathcal{H}(v_s)\big)\widetilde{\mathcal{M}}(s,w_s)\dd s\\&\qquad\qquad\quad \geq \mathcal{C}_3\widetilde{\mathcal{M}}(0,w_0)e^{\mathcal{C}_3\rho+\mathcal{C}_3\mathscr{C}(F(u_0)+F(u_0^\prime))+a}\big\},
\end{align}with the usual convention $\inf\emptyset=\infty$, 
and define measurable maps $\Phi^{u_0,u_0^\prime}, \widetilde{\Phi}^{u_0,u_0^\prime}:\Omega\rightarrow \Omega$ by 
\begin{align*}
    &\Phi^{u_0,u_0^\prime}(\omega)_t=\omega_t+\int_0^t\mathbf{1}_{s\leq \tau_1^v}\mathsf{P}_N\big[i|\hat{v}|^2\hat{v}-i|{u}|^2{u}-i\partial_x^2(\hat{v}-{u})\big]\dd s,\\
    &\widetilde{\Phi}^{u_0,u_0^\prime}(\omega)_t=\omega_t+\int_0^t\mathbf{1}_{s\leq \tau_1^{u,v}\wedge\tau_0^{u,v}}\mathsf{P}_N\big[i|\hat{v}|^2\hat{v}-i|{u}|^2{u}-i\partial_x^2(\hat{v}-{u})\big]\dd s.
\end{align*} The first reproduces the feedback term in \eqref{auxiliary 1}, so that
\eqref{04240259} below holds; the second truncates it at the earlier time
$\tau_1^{u,v}\wedge\tau_0^{u,v}$, which makes its drift bounded by a
deterministic constant and the Girsanov theorem applicable in Step~4, at the
price of the error term estimated in Step~3.

  Notice that
\begin{align}\label{04240259}
\hat{v}(\hspace{0.5mm}\cdot\hspace{0.5mm},\omega)
=\hat{u}^\prime\big(\hspace{0.5mm}\cdot\hspace{0.5mm},
\Phi^{u_0,u_0^\prime}(\omega)\big)
\qquad\text{for $\mathbb{P}$-a.e.\ }\omega.
\end{align}
 Indeed, let us set $\sigma:=\tau_1^{v}(\omega)\wedge
\tau_1^{u^\prime}\big(\Phi^{u_0,u_0^\prime}(\omega)\big)$. On $[0,\sigma]$
the indicators in~\eqref{trunc eqn 2} and \eqref{auxiliary 1} both
equal~$1$, so the two processes solve the same equation with the same
initial datum, and pathwise uniqueness gives their coincidence there. If
$\sigma\ge T$, this already gives \eqref{04240259}; otherwise, as the
stopping time~\eqref{04080050} is a functional of the trajectory alone, it
is attained at $\sigma$ by both, so that after $\sigma$ both indicators
vanish and the two processes solve $\partial_t\zeta+\alpha\zeta=0$ from the
same value, whence \eqref{04240259}.
As a result
\begin{align}\label{03140952}
    \mathbb{P}(\tau_1^{\hat{v}}<T)&= \Phi^{u_0,u_0^\prime}_*\mathbb{P}(\tau_1^{\hat{u}^\prime}<T)\leq \nonumber\mathbb{P}(\tau_1^{\hat{u}^\prime}<T)+\|\Phi^{u_0,u_0^\prime}_*\mathbb{P}-\mathbb{P}\|_{\mathrm{var}}\\&\leq \mathbb{P}(\tau_1^{\hat{u}^\prime}<T)+\|\Phi^{u_0,u_0^\prime}_*\mathbb{P}-\widetilde{\Phi}^{u_0,u_0^\prime}_*\mathbb{P}\|_{\mathrm{var}}\nonumber\\&\quad+\|\widetilde{\Phi}^{u_0,u_0^\prime}_*\mathbb{P}-\mathbb{P}\|_{\mathrm{var}}.
\end{align}
The three terms on the right-hand side of this inequality are estimated in Steps~2--4,
respectively.
 
\noindent {\it Step 2: estimate for $\mathbb{P}(\tau_1^{\hat{u}^\prime}<T)$}.
Since the truncation in \eqref{trunc eqn 2} is inactive on~$[0,\tau_1^{u^\prime}]$, the process $\hat u^\prime$ coincides there with
$u^\prime$, and therefore $\tau_1^{\hat u^\prime}=\tau_1^{u^\prime}$. 
Thus, as in \eqref{08021543}, we have
\begin{align}\label{04230821}
    \mathbb{P}\big(\tau_1^{\hat{u}^\prime}<T\big)\leq \mathbb{P}\big(\tau_1^{u^\prime}<\infty\big)\leq \mathcal{C}_4^0\rho^{-1}.
\end{align}

\noindent
{\it Step 3: estimate for $\|\Phi^{u_0,u_0^\prime}_*\mathbb{P}-\widetilde{\Phi}^{u_0,u_0^\prime}_*\mathbb{P}\|_{\mathrm{var}}$.} 
According to the definitions of $\Phi^{u_0,u_0^\prime}$ and $\widetilde{\Phi}^{u_0,u_0^\prime}$, we have
\begin{align*}
   \big\{\omega\in\Omega:\Phi^{u_0,u_0^\prime}(\omega)&\neq \widetilde{\Phi }^{u_0,u_0^\prime}(\omega)\big\}\\&\subseteq\big\{\tau_1^u\wedge\tau_0^{u,v}<\tau_1^v\wedge T\big\}\\&=\big\{\tau_1^u<\tau_1^v\wedge T,\tau_1^u\leq \tau_0^{u,v}\big\}{\scalebox{1.3}{$\cup$}} \big\{\tau_0^{u,v}<\tau_1^v\wedge T,\tau_0^{u,v}<\tau_1^u \big\}\\&\subseteq \big\{\tau_1^u<T\big\}{\scalebox{1.3}{$\cup$}}\big\{\tau_0^{u,v}<\tau_1^{u,v}\wedge T\big\}.
\end{align*}
Consequently, 
\begin{align}\label{04270250}
    \nonumber\|\Phi^{u_0,u_0^\prime}_*\mathbb{P}-\widetilde{\Phi}^{u_0,u_0^\prime}_*\mathbb{P}\|_{\mathrm{var}}&\leq\mathbb{P}\big(\Phi^{u_0,u_0^\prime}(\omega)\neq \widetilde{\Phi}^{u_0,u_0^\prime}(\omega)\big)\\&\leq \mathbb{P}\big(\tau_1^{u}< T\big)+\mathbb{P}\big(\tau_0^{u,v}<\tau_1^{u,v}\wedge T\big).
\end{align}
The first term on the right-hand side is estimated as in \eqref{08021543}: 
$$\mathbb{P}(\tau_1^u<T)\leq\mathbb{P}(\tau_1^u<\infty)\leq \mathcal{C}_4^0\rho^{-1}.
$$ 
For the second term, note that, by the
definition \eqref{04160408} of $\tau_0^{u,v}$, on the event
$\{\tau_0^{u,v}<\tau_1^{u,v}\wedge T\}$ one has
\begin{align*}
    \int_0^{\tau_1^{u,v}}\!\!\!\!\big(1+\mathcal{H}(u_t)+\mathcal{H}(v_t)\big)
    \widetilde{\mathcal{M}}(t,w_t)\dd t\geq  \mathcal{C}_3\widetilde{\mathcal{M}}(0,w_0)\,
    e^{\mathcal{C}_3\rho+\mathcal{C}_3\mathscr{C}(F(u_0)+F(u_0^\prime))+a}.
\end{align*}
By Chebyshev's inequality and Proposition~\ref{03100846},
\begin{align}\label{07310407}
    \mathbb{P}\big(\tau_0^{u,v}<\tau_1^{u,v}\wedge T\big)\leq e^{-a}.
\end{align}
Thus,
\begin{align}\label{04230822}
    \|\Phi^{u_0,u_0^\prime}_*\mathbb{P}-\widetilde{\Phi}^{u_0,u_0^\prime}_*\mathbb{P}\|_{\mathrm{var}}\leq \mathcal{C}^0_4\rho^{-1}+e^{-a}.
\end{align}

\noindent{\it Step 4: estimate for $\|\widetilde{\Phi}^{u_0,u_0^\prime}_*\mathbb{P}-\mathbb{P}\|_{\mathrm{var}}$}.
Let us write $\Omega=\Omega_N\oplus\Omega_N^\perp$, where
\[
\Omega_N:=C_0([0,T];\mathsf{P}_NH^4),\qquad
\Omega_N^\perp:=C_0([0,T];\mathsf{Q}_NH^4).
\]
Setting $\tau_2^{u,v}:=\tau_1^{u,v}\wedge\tau_0^{u,v}\wedge T$ and
\begin{align*}
    \mathcal{A}\big(s,(\omega^{(1)},\omega^{(2)})\big):=
    \mathbf{1}_{s\leq\tau_2^{u,v}}\,\mathsf{P}_N
    \big[i|v|^2v-i|u|^2u-i\partial_x^2(v-u)\big],
\end{align*}
we introduce the map $\Psi:\Omega\to\Omega_N$ given by
\begin{align*}
    \Psi(\omega^{(1)},\omega^{(2)})_t:=\omega_t^{(1)}
    +\int_0^{t}\mathcal{A}\big(s,(\omega^{(1)},\omega^{(2)})\big)\dd s.
\end{align*}
Since $\hat v_t=v_t$ for $t\le\tau_1^{v}$ and
$\tau_2^{u,v}\le\tau_1^{u,v}\le\tau_1^{v}$, the drift of $\Psi$ coincides
with that of $\widetilde\Phi^{u_0,u_0^\prime}$, so that
\[
\widetilde{\Phi}^{u_0,u_0^\prime}(\omega^{(1)},\omega^{(2)})
=\big(\Psi(\omega^{(1)},\omega^{(2)}),\omega^{(2)}\big).
\]
Note also that $u$ and $v$ are adapted and $\tau_2^{u,v}$ is a stopping time,
so~$\mathcal{A}$ is adapted. 
 Lemma~3.3.13 in~\cite{KS12} implies that 
\begin{align*}
\|\widetilde{\Phi}^{u_0,u_0^\prime}_*\mathbb{P}-\mathbb{P}\|_{\mathrm{var}}\nonumber\leq \int_{\Omega_N^\perp}\|\Psi_*(\hspace{0.5mm}\cdot\hspace{0.5mm},\omega^{(2)})\mathbb{P}_N-\mathbb{P}_N\|_{\mathrm{var}}\,\mathbb{P}_N^\perp(\dd \omega^{(2)}),
\end{align*}
where $\mathbb{P}_N:=(\mathsf{P}_N)_*\mathbb{P}\in\mathcal{P}(\Omega_N)$ and $\mathbb{P}_N^{\perp}:=(\mathsf{Q}_N)_*\mathbb{P}\in\mathcal{P}(\Omega^\perp_N)$. By  Theorem~A.10.1 in~\cite{KS12}, we have
\begin{align}\label{06140347}
      \nonumber\|\Psi_*(\hspace{0.5mm}\cdot\hspace{0.5mm},&\hspace{0.5mm}\omega^{(2)})\mathbb{P}_N-\mathbb{P}_N\|_{\mathrm{var}}\\\leq&\  \frac{1}{2}\sqrt{\mathbb{E}_N\Big[\mathrm{exp}\big\{6\sup_{1\leq j\leq N} b_j^{-2}\int_0^T \|\mathcal{A}(t,(\hspace{0.5mm}\cdot\hspace{0.5mm},\omega^{(2)}))\|^2\dd t\big\}\Big]^{\frac{1}{2}}-1},
\end{align}where $\mathbb{E}_N$ is the expectation corresponding to $\mathbb{P}_N$, provided Novikov's condition
\begin{align}\label{06140343}
    \mathbb{E}_N\mathrm{exp}\Big\{C\int_0^T \|\mathcal{A}(t,(\hspace{0.5mm}\cdot\hspace{0.5mm},\omega^{(2)}))\|^2\dd t\Big\}<\infty
\end{align}
is satisfied for any $C>0$ and $\omega^{(2)}\in\Omega_N^{\perp}$, which we verify below.
Notice that, by \eqref{07080243-1},
\begin{align*}
\nonumber\|\mathcal{A}(t,\omega)\|^2&\leq\mathbf{1}_{t\leq\tau_2^{u,v}}\|\mathsf{P}_N\big(|v_t|^2v_t-|u_t|^2u_t-\partial_x^2(v_t-u_t)\big)\|^2\\&\leq \nonumber C\mathbf{1}_{t\leq\tau_2^{u,v}}(\|u_t\|_{L^4}^4+\|v_t\|_{L^4}^4)\|u_t-v_t\|^2_{H^1}\\&\quad\nonumber+C\mathbf{1}_{t\leq\tau_2^{u,v}}\|u_t-v_t\|^2\\&\leq   C\mathbf{1}_{t\leq\tau_2^{u,v}}\big(1+\mathcal{H}(u_t)+\mathcal{H}(v_t)\big)\widetilde{\mathcal{M}}(t,u_t-v_t),
\end{align*}
where the last inequality follows from \eqref{06190211}. Thus, since $\tau_2^{u,v}\le\tau_0^{u,v}$, the definition \eqref{04160408}
of $\tau_0^{u,v}$ gives
\begin{align}\label{03151055}
     \int_0^T\|\mathcal{A}(t,\omega)\|^2\dd t
     &\leq C\int_0^{\tau^{u,v}_2}\big(1+\mathcal{H}(u_t)+\mathcal{H}(v_t)\big)
     \widetilde{\mathcal{M}}(t,u_t-v_t)\dd t\nonumber\\
     &\leq C\mathcal{C}_{3}\widetilde{\mathcal{M}}(0,u_0-u^\prime_0)
     e^{\mathcal{C}_3\rho+\mathcal{C}_3\mathscr{C}(F(u_0)+F(u_0^\prime))+a}
\end{align}
for any $\omega=(\omega^{(1)},\omega^{(2)})\in\Omega$.
Since the last estimate in \eqref{03151055} is a deterministic constant not
depending on $\omega^{(2)}$, it satisfies Novikov's condition \eqref{06140343}
for every $C>0$, justifying \eqref{06140347}, and the same bound on the
exponent also holds with $\mathbb{E}_N$ replaced by $\mathbb{E}$; integrating
\eqref{06140347} against $\mathbb{P}_N^\perp(\dd\omega^{(2)})$ therefore~gives
\begin{align*}
     \|\widetilde{\Phi}^{u_0,u_0^\prime}_*\mathbb{P}-\mathbb{P}\|_{\mathrm{var}}\leq   \frac{1}{2}\sqrt{\mathbb{E}\Big[\mathrm{exp}\big\{6\sup_{1\leq j\leq N} b_j^{-2}\int_0^T \|\mathcal{A}(t,\hspace{0.5mm}\cdot\hspace{0.5mm})\|^2\dd t\big\}\Big]^{\frac{1}{2}}-1}.
\end{align*}
Combining this with \eqref{03151055} and using that $u_0,u_0^\prime\in B_{H^2}(d_0)$ with $d_0\leq 1$,  we obtain the existence of   $\widetilde{\mathcal{C}}_4>0$ such that
\begin{align}\label{03200856}
    \|\widetilde{\Phi}^{u_0,u_0^\prime}_*\mathbb{P}-\mathbb{P}\|_{\mathrm{var}}\leq&\ \frac{1}{2}\sqrt{\mathrm{exp}\big\{C \mathcal{C}_3(d_0^2+d_0^4)e^{\widetilde{\mathcal{C}}_4\mathscr{C}d_0+\mathcal{C}_3\rho+a}\big\}-1}.
\end{align}

\noindent{\it Step 5: conclusion.} Combining \eqref{03140952}, \eqref{04230821}, \eqref{04230822}, and \eqref{03200856},  we complete the proof of the proposition with $\mathcal{C}_4:=(2\mathcal{C}_4^0)\vee \widetilde{\mathcal{C}}_4\vee \mathcal{C}_3\vee(C\mathcal{C}_3)$, where
$C>0$ is the constant in \eqref{03200856}.
\end{proof}

\section{Proof of the Main Theorem}\label{S:4}

This section is devoted to the proof of the Main Theorem. In
Subsection~\ref{S:4.1}, we construct a coupling of two solutions
of~\eqref{equation-1} and define the corresponding coupling events. In
Subsection~\ref{S:4.2}, we estimate the probability that a coupling
attempt breaks down at a given step. The proof of the Main Theorem is
completed in Subsection~\ref{S:4.3}.

\subsection{Construction of the coupling process}\label{S:4.1}

We fix a time step $T>0$, to be chosen later, and denote by
\begin{gather*}
 u:H^2\times C_0([0,T];H^4)\rightarrow C([0,T];H^2),\\
 v:H^2\times H^2\times C_0([0,T];H^4)\rightarrow C([0,T];H^2)
\end{gather*}
the solution maps of \eqref{equation-1} and \eqref{auxiliary}: for any $u_0,u_0^\prime\in H^2$ and any $C_0([0,T];H^4)$-valued random
variable $\eta$ with the same law as~$bW,$ the trajectories $u(u_0,\eta)$
and $v(u_0,u_0^\prime,\eta)$ are the unique solutions of \eqref{equation-1} and
\eqref{auxiliary} driven by~$\eta$ and issued from $u_0$ and~$u_0^\prime$
(see Proposition~3.5 in~\cite{DD03} and Lemma~\ref{07101918}); their values at
time~$t$ are written $u(t;u_0,\eta)$ and $v(t;u_0,u_0^\prime,\eta)$. The same
letters denote the solution maps on the half-line, defined on
$H^2\times C_0([0,\infty);H^4)$ and $H^2\times H^2\times C_0([0,\infty);H^4)$
and taking values in $C([0,\infty);H^2)$; this creates no ambiguity, since the time interval will always be clear
from the context.

To construct the maximally coupled noises that will drive the
coupling process, we apply Lemma~\ref{L:max-coupling} with
$$
X=C_0([0,T];H^4),\qquad Y=C([0,T];H^2),\qquad Z=H^2\times H^2,
$$
where $\mu=\mathcal{D}(bW)$ is the law of the noise on~$X$ and, for
$z=(u_0,u_0^\prime)\in Z$ and $\eta\in X$,
$$
f_1(z,\eta):=v(\hspace{0.5mm}\cdot\hspace{0.5mm}\,;u_0,u_0^\prime,\eta),
\qquad
f_2(z,\eta):=u(\hspace{0.5mm}\cdot\hspace{0.5mm}\,;u_0^\prime,\eta),
$$
which are measurable in $(z,\eta)$; the corresponding measures are
$$
\nu_1(z,\hspace{0.5mm}\cdot\hspace{0.5mm})=\mathcal{D}\big(v(\hspace{0.5mm}\cdot\hspace{0.5mm}\,;u_0,u_0^\prime,bW)\big),
\qquad
\nu_2(z,\hspace{0.5mm}\cdot\hspace{0.5mm})=\mathcal{D}\big(u(\hspace{0.5mm}\cdot\hspace{0.5mm}\,;u_0^\prime,bW)\big).
$$
The lemma provides a probability space, which we again denote by
$(\Omega,\mathcal{F},\mathbb{P})$, and a family of
$C_0([0,T];H^4)$-valued random variables
$b\widetilde{W}^i(u_0,u_0^\prime)$, $i=1,2$, depending measurably on
$(u_0,u_0^\prime)\in H^2\times H^2$, with the following two
properties. First, each $b\widetilde{W}^i(u_0,u_0^\prime)$ has the same law as $bW$; in
particular, $v(\hspace{0.5mm}\cdot\hspace{0.5mm}\,;u_0,u_0^\prime,
b\widetilde{W}^1(u_0,u_0^\prime))$ and
$u(\hspace{0.5mm}\cdot\hspace{0.5mm}\,;u_0^\prime,
b\widetilde{W}^2(u_0,u_0^\prime))$ solve \eqref{auxiliary} and
\eqref{equation-1}, respectively. Second, for
any $(u_0,u_0^\prime)\in Z$, these two processes form a maximal
coupling of $\nu_1(z,\hspace{0.5mm}\cdot\hspace{0.5mm})$ and
$\nu_2(z,\hspace{0.5mm}\cdot\hspace{0.5mm})$; in particular,
\begin{align}\label{04280357}
\mathbb{P}\Big(
u\big(t;u_0^\prime,&b\widetilde{W}^2(u_0,u_0^\prime)\big)
\neq
v\big(t;u_0,u_0^\prime,b\widetilde{W}^1(u_0,u_0^\prime)\big)
\ \text{ for some }t\in[0,T]\Big)\nonumber\\
&=\big\|\mathcal{D}\big(u(\hspace{0.5mm}\cdot\hspace{0.5mm}\,;u_0^\prime,bW)\big)
-\mathcal{D}\big(v(\hspace{0.5mm}\cdot\hspace{0.5mm}\,;u_0,u_0^\prime,bW)\big)\big\|_{\mathrm{var}}.
\end{align}
 Thus the two solutions coincide on $[0,T]$ with the maximal probability allowed
by the distance between their laws. 

Without loss of generality, we assume that the underlying probability space
supports a family
\begin{align}\label{04112204}
\big\{\big(b\widetilde{W}^{1n},b\widetilde{W}^{2n}\big)\big\}_{n\geq0}
\end{align}
of independent copies of the pair of maps
$\big(b\widetilde{W}^{1},b\widetilde{W}^{2}\big)$, together with random
variables $bW^{in}$, $i=1,2$, $n\geq0$, with values in $C_0([0,T];H^4)$, each
having the law of $bW,$ such that the family consisting of the pairs in \eqref{04112204} and
the variables $bW^{in}$, $i=1,2$, $n\geq0$, is mutually independent.

We now turn to the construction of the coupling process, beginning with the
coupling event. Given a triple $(\widetilde{u},\widetilde{u}^\prime,
\widetilde{v})$ of $H^2$-valued processes and integers $0\leq l\leq k$, we define the event
 \begin{equation*}
    (P_{l,k}):\quad\left\{
    \begin{aligned}
        &\widetilde{u}^\prime_t=\widetilde{v}_t
          \qquad\text{for all } t\in[lT,kT],\\
        &\big(\widetilde{u}_{lT},\widetilde{u}^\prime_{lT}\big)
          \in B_{H^2}(d_0)\times B_{H^2}(d_0),\\
        &{\tau}_1^{\widetilde{u},\widetilde{u}^\prime}\circ\theta_{lT}
          \geq (k-l)T,
    \end{aligned}
    \right.
\end{equation*}
where $d_0>0$ is specified later in Corollary~\ref{07311420},
$\tau_1^{\widetilde{u},\widetilde{u}^\prime}
:=\tau_1^{\widetilde{u}}\wedge\tau_1^{\widetilde{u}^\prime}$ is as in~\eqref{06230835-1}, and $\theta_t$ denotes the time shift, acting on trajectory functionals by
\begin{align}\label{08120353}
f(\widetilde{u},\widetilde{u}^\prime)\circ\theta_t
:=f\big(\widetilde{u}(t+\cdot),\widetilde{u}^\prime(t+\cdot)\big)
\end{align}
for any measurable $f:C([0,\infty);H^2)^2\rightarrow\mathbb{R}$ and $t\geq0$,
and similarly for a finite stopping time in place of~$t$.  In words, $(P_{l,k})$ says that a coupling attempt launched at time~$lT$ is
still in progress at step~$k$: both data are small at the launch time, the
processes $\widetilde{u}^\prime$ and $\widetilde{v}$ have not separated since,
and the energy functionals of $\widetilde{u}$ and $\widetilde{u}^\prime$ have
not grown abnormally. For $k=l$ the event reduces to the ball condition.

Accordingly, we introduce an integer-valued process $l_0(\cdot)$, defined
recursively by $l_0(-1):=\infty$ and, for $k\ge0$,
\begin{align*}
 l_0(k):=
 \begin{cases}
 l_0(k-1) & \text{on }\{l_0(k-1)<\infty\}\cap(P_{l_0(k-1),k}),\\[1mm]
 \infty   & \text{on }\{l_0(k-1)<\infty\}\cap(P_{l_0(k-1),k})^c,\\[1mm]
 k        & \text{on }\{l_0(k-1)=\infty\}\cap(P_{k,k}),\\[1mm]
 \infty   & \text{on }\{l_0(k-1)=\infty\}\cap(P_{k,k})^c.
 \end{cases}
\end{align*}
In particular, $l_0(0)=0$ if $u_0,u_0^\prime\in B_{H^2}(d_0)$, and
$l_0(0)=\infty$ otherwise.

  Thus an attempt launched at step~$l$ has $l_0(l)=l$, and $l_0(k)=l$ for~$k>l$ provided that $(P_{l,j})$ holds for every $j=l,\dots,k$; $l_0$ is reset to~$\infty$ at the first step at
which it fails, and $l_0(k)=\infty$ means that no attempt is in progress at
step~$k$. In particular, $l_0$ is not monotone. Our definition of $l_0$ differs from the one in~\cite{DO05}, where $l_0(k)$ is taken to be the
smallest $l\le k$ for which~$(P_{l,k})$ holds; with that convention a new
attempt may be launched at the same step at which the previous one failed.
Our recursive definition depends only on~$l_0(k-1)$ and on one of the events
$(P_{l_0(k-1),k})$ and $(P_{k,k})$; this is what makes the Markov-type property
of Lemma~\ref{L1:l0-markov} available and yields property~(i) below.

The triple $(\widetilde{u},\widetilde{u}^\prime,\widetilde{v})$ is defined
recursively, starting from
$(\widetilde{u}_0,\widetilde{u}^\prime_0,\widetilde{v}_0)
=(u_0,u_0^\prime,u_0^\prime)$. Given its values at time $nT$, $n\geq0$, we set
\begin{equation}\label{E:recursion}
\left\{
\begin{aligned}
    &\widetilde{u}_t=u\big(t-nT;\widetilde{u}_{nT},\eta^{1n}\big),\\
    &\widetilde{u}^\prime_t=u\big(t-nT;\widetilde{u}^\prime_{nT},\eta^{2n}\big),\\
    &\widetilde{v}_t=v\big(t-nT;\widetilde{u}_{nT},\widetilde{u}^\prime_{nT},
      \eta^{1n}\big),
\end{aligned}
\right.
\end{equation} where the first two relations hold for $t\in[nT,(n+1)T]$ and the third
for $t\in[nT,(n+1)T)$. The driving noises are
\begin{align}\label{08070318}
\big(\eta^{1n},\eta^{2n}\big):=
\begin{cases}
\big(b\widetilde{W}^{1n}(\widetilde{u}_{nT},\widetilde{u}^\prime_{nT}),\,
 b\widetilde{W}^{2n}(\widetilde{u}_{nT},\widetilde{u}^\prime_{nT})\big),
 & \text{if } l_0(n)<\infty,\\[1mm]
\big(bW^{1n},\,bW^{2n}\big), & \text{otherwise}.
\end{cases}
\end{align}
Thus $\widetilde{u}$ and $\widetilde{v}$ are always driven by the same noise
$\{\eta^{1n}\}_{n\geq0}$, while $\widetilde{u}^\prime$ is driven by
$\{\eta^{2n}\}_{n\geq0}$: when $l_0(n)<\infty$, the pair
$(\eta^{1n},\eta^{2n})$ is the maximally coupled one, and otherwise
$\eta^{1n}$ and $\eta^{2n}$ are independent.
 The following two properties follow from this construction:
\begin{itemize}
 \item[(i)] 
  $l_0(\vartheta_k)=\vartheta_k$ on $\{\vartheta_k<\infty\}$, for the
recurrence times $\vartheta_k$ defined in Subsection~\ref{S:4.3}.  
\item[(ii)]   On each interval $[nT,(n+1)T)$, the difference
$w:=\widetilde{u}-\widetilde{v}$ solves the pathwise equation
\eqref{02230853-1} with initial condition
$\widetilde{u}_{nT}-\widetilde{u}^\prime_{nT}$. Substituting
$\widetilde v=\widetilde u-w$ there, we obtain an equation for $w$ in which
$\widetilde{u}|_{[nT,(n+1)T]}$ is the only external datum. By the same
fixed-point argument as in Lemma~\ref{WPA2}, that equation is well posed and
its solution depends measurably on
$\big(\widetilde{u}^\prime_{nT},\widetilde{u}|_{[nT,(n+1)T]}\big)$; hence so
does $\widetilde{v}|_{[nT,(n+1)T)}$, and $\widetilde{v}$ is a measurable
functional of $(\widetilde{u},\widetilde{u}^\prime)$.  
\end{itemize}

Concatenating the increments $\eta^{in}$ over successive intervals, we
define the noise maps
$\widetilde{\eta}^i:H^2\times H^2\times\Omega\rightarrow
C_0([0,\infty);H^4)$, $i=1,2$, recursively in~$n$ by
$\widetilde{\eta}^i_0:=0$ and
\begin{equation}\label{E:noise-concat}
\widetilde{\eta}_t^i(u_0,u_0^\prime)
:=\widetilde{\eta}^i_{nT}(u_0,u_0^\prime)+\eta^{in}_{t-nT},
\qquad t\in[nT,(n+1)T]
\end{equation}
for any $u_0,u_0^\prime\in H^2$.   Since $\widetilde{\eta}^{1}$ and $\widetilde{\eta}^{2}$ are
$C_0([0,\infty);H^4)$-valued, by \eqref{E:recursion},
\eqref{E:noise-concat}, and pathwise uniqueness, we have
\begin{align}\label{08130433}
\widetilde{u}=u(\hspace{0.5mm}\cdot\hspace{0.5mm}\,;u_0,\widetilde{\eta}^{1})
\quad\text{and}\quad
\widetilde{u}^\prime=u(\hspace{0.5mm}\cdot\hspace{0.5mm}\,;u_0^\prime,\widetilde{\eta}^{2}).
\end{align}
Let $\widetilde{\mathcal{F}}_{nT}$ be
the $\sigma$-algebra generated by
$\big\{\widetilde{\eta}^i_t:\,   t\leq nT,\, i=1,2\big\}$. From \eqref{E:recursion} and \eqref{08130433} it follows that the restriction
$(\widetilde{u},\widetilde{u}^\prime,\widetilde{v})|_{[0,nT]}$ is $\widetilde{\mathcal{F}}_{nT}$-measurable, and so are the events
$(P_{l,k})$, $l\leq k\leq n$, and the random variables $l_0(k)$,
$k\leq n$.

To prove that $(\widetilde{u},\widetilde{u}^\prime)$ is a coupling of $(u,u^\prime)$, it suffices to show that each of $\widetilde{\eta}^{1}$ and $\widetilde{\eta}^{2}$ has the same distribution as~$bW$. For the noise processes $\widetilde{\eta}^i$, $i=1,2$, we write $\theta_t\widetilde{\eta}^{i}$ for the shifted process $\widetilde{\eta}^i_{t+\cdot}-\widetilde{\eta}^i_t$, which differs from the shift introduced in~\eqref{08120353}.
\begin{lemma}\label{08080341}
For any $u_0,u_0^\prime\in H^2$ and $i=1,2$, the process
$\widetilde{\eta}^{i}(u_0,u_0^\prime)$ has the same distribution as~$bW$. Moreover, for any $k\geq1$, the increment
$\theta_{kT}\widetilde{\eta}^{i}(u_0,u_0^\prime)$ is
independent of $\widetilde{\mathcal{F}}_{kT}$.
\end{lemma}
 \begin{proof}
  It suffices to consider the case $i=1$. We prove by induction on $n\geq1$ that
$\widetilde{\eta}^{1}(u_0,u_0^\prime)|_{[0,nT]}$ has the distribution
$\mathcal{D}(bW|_{[0,nT]})$, which gives the first assertion. In the rest of the proof
we drop the fixed initial data from the notation. For the base case $n=1$, note that the noise selection in
\eqref{08070318} is decided by the initial data alone:
$\eta^{10}=b\widetilde{W}^{10}(u_0,u_0^\prime)$ if
$u_0,u_0^\prime\in B_{H^2}(d_0)$, and $\eta^{10}=bW^{10}$ otherwise. In
both cases $\widetilde{\eta}^{1}|_{[0,T]}=\eta^{10}$ has the
distribution $\mathcal{D}(bW|_{[0,T]})$ and is independent of the trivial $\sigma$-algebra $\widetilde{\mathcal{F}}_0$.

Let us assume that $\widetilde{\eta}^{1}|_{[0,nT]}$ has the distribution
$\mathcal{D}(bW|_{[0,nT]})$ for some $n\geq1$, and prove it for $n+1$. We claim that the noise $\eta^{1n}$ in \eqref{08070318} is
independent of $\widetilde{\mathcal{F}}_{nT}$.
 Indeed, for any $B\in\mathcal{B}(C_0([0,nT];H^4)^2)$ and bounded measurable function $h:C_0([0,T];H^4)\rightarrow\mathbb{R}$, we have
     \begin{align*}
         \mathbb{E}\big[h(&\eta^{1n});(\widetilde{\eta}^1,\widetilde{\eta}^2)\big|_{[0,nT]}\in B\big]\\&= \mathbb{E}\big[h(b\widetilde{W}^{1n}(\widetilde{u}_{nT},\widetilde{u}^\prime_{nT}));(\widetilde{\eta}^1,\widetilde{\eta}^2)\big|_{[0,nT]} \in  B,\ l_0(n)<\infty\big]\\&\quad+\mathbb{E}\big[h(b{W}^{1n});(\widetilde{\eta}^1,\widetilde{\eta}^2)\big|_{[0,nT]}
\in B,\ l_0(n)=\infty\big]=:\mathrm{I}+\mathrm{II}.
     \end{align*}
     From the construction it follows that the event
$\{(\widetilde{\eta}^1,\widetilde{\eta}^2)|_{[0,nT]}\in B\}$ and the
random variable $l_0(n)$ are both $\widetilde{\mathcal{F}}_{nT}$-measurable. Moreover, the map
$b\widetilde{W}^{1n}$ and the variable $bW^{1n}$ are independent of
$\widetilde{\mathcal{F}}_{nT}$. Thus we obtain that
     \begin{align}\label{08080233-1}
         \nonumber\mathrm{I}&=\mathbb{E}\big[\mathbb{E}\big[h(b\widetilde{W}^{1n}(w,w^\prime))\big]\big|_{(w,w^\prime)=(\widetilde{u}_{nT},\widetilde{u}^\prime_{nT})};(\widetilde{\eta}^1,\widetilde{\eta}^2)\big|_{[0,nT]}\in B, \ l_0(n)<\infty\big]\\& =\mathbb{E}\big[h(bW)\big]\mathbb{P}\big((\widetilde{\eta}^1,\widetilde{\eta}^2)\big|_{[0,nT]}\in B, \ l_0(n)<\infty\big).
     \end{align}
     Similarly, we have
     \begin{align}\label{08080233-2}
         \mathrm{II}=\mathbb{E}\big[h(bW)\big]\mathbb{P}\big((\widetilde{\eta}^1,\widetilde{\eta}^2)\big|_{[0,nT]} \in  B, \ l_0(n)= \infty\big).
     \end{align}
Summing \eqref{08080233-1} and \eqref{08080233-2}, we get
     \begin{align*}
         \mathbb{E}\big[h(\eta^{1n});(\widetilde{\eta}^1,\widetilde{\eta}^2)\big|_{[0,nT]}\in B\big]=\mathbb{E}\big[h(bW)\big]\mathbb{P}\big((\widetilde{\eta}^1,\widetilde{\eta}^2)\big|_{[0,nT]}\in B\big).
     \end{align*}
Taking $B$ to be the full space gives
$\mathcal{D}(\eta^{1n})=\mathcal{D}(bW|_{[0,T]})$; the arbitrariness of
$h$ and $B$ then proves that $\eta^{1n}$ is independent of
$\widetilde{\mathcal{F}}_{nT}$. Thus, by \eqref{E:noise-concat} and the
induction assumption, the process $\widetilde{\eta}^{1}|_{[0,(n+1)T]}$
is the extension of a path with distribution
$\mathcal{D}(bW|_{[0,nT]})$ by an independent increment with
distribution $\mathcal{D}(bW|_{[0,T]})$, and therefore has the
distribution $\mathcal{D}(bW|_{[0,(n+1)T]})$.

  For the second assertion, let $k\geq1$. By the claim established above,
for any $M\geq0$, any bounded measurable
$h_0,\ldots,h_M:C_0([0,T];H^4)\rightarrow\mathbb{R}$ and any
$A\in\widetilde{\mathcal{F}}_{kT}$, conditioning successively on
$\widetilde{\mathcal{F}}_{(k+M)T},\ldots,\widetilde{\mathcal{F}}_{(k+1)T}$
gives
$$
\mathbb{E}\Big[\prod_{m=0}^{M}h_m\big(\eta^{1,k+m}\big);A\Big]
=\prod_{m=0}^{M}\mathbb{E}\big[h_m(bW)\big]\,\mathbb{P}(A).
$$
Since, by \eqref{E:noise-concat}, the noises $\eta^{1,k+m}$ are the increments
of $\theta_{kT}\widetilde{\eta}^1$ over the intervals $[mT,(m+1)T]$, and $M$
is arbitrary, the above identity proves the second assertion.
 \end{proof}
 
By Lemma~\ref{08080341} and the uniqueness of solutions
to~\eqref{equation-1}, the pair $(\widetilde{u},\widetilde{u}^\prime)$
is a coupling of $(u,u^\prime)$. However, this pair is not a Markov
process: by \eqref{08070318}, the noise driving the $n$-th interval
depends on whether a coupling attempt is in progress, and this is not
determined by the position at time~$nT$. We will establish two
Markov-type properties (Lemmas~\ref{local markov}
and~\ref{L1:l0-markov}) that substitute for the Markov property in the
proof of polynomial mixing.

\subsection{Decoupling estimates}\label{S:4.2}

For $w,w^\prime\in H^2$, we denote by $\mathbf{P}_{w,w^\prime}$ the
distribution in $C([0,T];H^2)^2$ of the pair
$$
\big(u(\cdot\,;w,b\widetilde{W}^{1}(w,w^\prime)),\
u(\cdot\,;w^\prime,b\widetilde{W}^{2}(w,w^\prime))\big),
$$
that is, of a pair of solutions on one time block driven by the
maximally coupled noises, and by $\mathbf{E}_{w,w^\prime}$ the
corresponding expectation. Then we have the following Markov-type
property.
\begin{lemma}\label{local markov}
For any non-negative measurable function
$g:C([0,T];H^2)^2\rightarrow\mathbb{R}$ and $k\ge0$, we have
\begin{align}\label{E:local-markov}
\mathbb{E}\big[g(\widetilde{u},\widetilde{u}^\prime)\circ\theta_{k T}
\,\big|\,\widetilde{\mathcal{F}}_{k T}\big]
=\mathbf{E}_{\widetilde{u}_{k T},\widetilde{u}^\prime_{k T}}
\big[g(\widetilde{u},\widetilde{u}^\prime)\big]
\qquad\text{on }\{l_0(k)<\infty\}.
\end{align}
\end{lemma}
\begin{proof}
    On $\{l_0(k)<\infty\}$, for any $t\in[0,T]$,
    \begin{align*}(\widetilde{u},\widetilde{u}^\prime)\circ\theta_{kT}(t)=\big(u(t;\widetilde{u}_{kT},b\widetilde{W}^{1k}(\widetilde{u}_{kT},\widetilde{u}_{kT}^\prime)),u(t;\widetilde{u}^\prime_{kT},b\widetilde{W}^{2k}(\widetilde{u}_{kT},\widetilde{u}_{kT}^\prime))\big).
    \end{align*}
Notice that the pair of maps
$\big(b\widetilde{W}^{1k},b\widetilde{W}^{2k}\big)$ is independent of
$\widetilde{\mathcal{F}}_{kT}$, while, for any
$B\in\mathcal{B}(C_0([0,kT];H^4)^2)$, the event
$\{(\widetilde{\eta}^1,\widetilde{\eta}^2)|_{[0,kT]}\in B\}$ and the
pair $(\widetilde{u}_{kT},\widetilde{u}^\prime_{kT})$ are
$\widetilde{\mathcal{F}}_{kT}$-measurable. Moreover, $l_0(k)$ being
$\widetilde{\mathcal{F}}_{kT}$-measurable, there exists
$B_k\in\mathcal{B}(C_0([0,kT];H^4)^2)$ such that
$$
\{l_0(k)<\infty\}=\{(\widetilde{\eta}^1,\widetilde{\eta}^2)|_{[0,kT]}
\in B_k\}.
$$
Hence, by the same argument as in the proof of
Lemma~\ref{08080341}, for any non-negative measurable
$h:H^2\times H^2\times C_0([0,T];H^4)^2\rightarrow\mathbb{R}$,
    \begin{align}\label{E:freeze-k}
        & \nonumber\mathbb{E}\big[h\big(\widetilde{u}_{kT},\widetilde{u}^\prime_{kT},b\widetilde{W}^{1k}(\widetilde{u}_{kT},\widetilde{u}_{kT}^\prime),\hspace{-0.5mm}b\widetilde{W}^{2k}(\widetilde{u}_{kT},\widetilde{u}_{kT}^\prime)\big);\hspace{-0.5mm}(\widetilde{\eta}^1,\widetilde{\eta}^2)|_{[0,kT]}\!\in\!\hspace{-0.5mm} B, l_0(k)\!<\!\infty\big]\\& =\int_{(B\cap B_k)\times H^2\times H^2}\!\!\!\!\!\!\!\!\!\mathbb{E}\big[h\big(w,w^\prime,b\widetilde{W}^{1k}(w,w^\prime),b\widetilde{W}^{2k}(w,w^\prime)\big)\big] \Lambda_k(\dd \zeta_1,\dd \zeta_2,\dd w,\dd w^\prime),
    \end{align}where $\Lambda_k:=\mathcal{D}(\widetilde{\eta}^1|_{[0,kT]},\widetilde{\eta}^2|_{[0,kT]},\widetilde{u}_{kT},\widetilde{u}_{kT}^\prime).$
We apply this equality to the function
\begin{align*}
h(w,w^\prime,\zeta_1,\zeta_2)
:=g\big(u(\cdot\,;w,\zeta_1),u(\cdot\,;w^\prime,\zeta_2)\big),
\end{align*} where $w,w^\prime\in H^2$, $ \zeta_1,\zeta_2\in C_0([0,T];H^4)$.
With this choice, on $\{l_0(k)<\infty\}$,
$$
    h\big(\widetilde{u}_{kT},\widetilde{u}^\prime_{kT},b\widetilde{W}^{1k}(\widetilde{u}_{kT},\widetilde{u}_{kT}^\prime),b\widetilde{W}^{2k}(\widetilde{u}_{kT},\widetilde{u}_{kT}^\prime)\big)=g(\widetilde{u},\widetilde{u}^\prime)\circ\theta_{kT},
$$   
and 
\begin{align*}
\mathbb{E}\big[&h\big(w,w^\prime,b\widetilde{W}^{1k}(w,w^\prime),b\widetilde{W}^{2k}(w,w^\prime)\big)\big]\\&=\mathbb{E}\big[g\big(u(\cdot;w,b\widetilde{W}^1(w,w^\prime)),u(\cdot;w^\prime,b\widetilde{W}^2(w,w^\prime))\big)\big]=\mathbf{E}_{w,w^\prime}\big[g(\widetilde{u},\widetilde{u}^\prime)\big].
\end{align*}
Substituting these two identities into \eqref{E:freeze-k}, we obtain
\begin{align*}
\mathbb{E}\big[g(\widetilde{u}&,\widetilde{u}^\prime)\circ\theta_{k T};
(\widetilde{\eta}^1,\widetilde{\eta}^2)|_{[0,kT]}\in B,\ l_0(k)<\infty \big]\\
&=\int_{(B\cap B_k)\times H^2\times H^2}
\mathbf{E}_{w,w^\prime}\big[g(\widetilde{u},\widetilde{u}^\prime)\big]\,
\Lambda_k(\dd \zeta_1,\dd \zeta_2,\dd w,\dd w^\prime)\\
&=\mathbb{E}\big[\mathbf{E}_{\widetilde{u}_{kT},\widetilde{u}_{kT}^\prime}
[g(\widetilde{u},\widetilde{u}^\prime)];
(\widetilde{\eta}^1,\widetilde{\eta}^2)|_{[0,kT]}\in B,\ l_0(k)<\infty \big].
\end{align*}
The arbitrariness of $B$ and the definition of
$\widetilde{\mathcal{F}}_{kT}$ then give \eqref{E:local-markov}.
\end{proof}
The next lemma shows that, for a suitable choice of $d_0=d_0(\rho)$, the two
trajectories remain coupled on $[0,T]$ with probability at least $1/2$, provided
both initial conditions lie in $B_{H^2}(d_0)$. Here $\rho$ is the truncation level entering the
stopping time $\tau_1$, see \eqref{06230835-1}, and $\rho_*$ is 
fixed in \eqref{04271005-3}.
\begin{lemma}\label{03151059}
There is $\rho_0\geq\rho_*$ such that, for any $\rho\geq\rho_0$, one can find
$d_0=d_0(\rho)\in(0,1)$ with the following property: for any $T>0$ and any
$u_0,u_0^\prime\in B_{H^2}(d_0)$,
\begin{align*}
    \mathbb{P}\big(l_0(1)\neq 0\big)\leq \frac{1}{2}.
\end{align*}
\end{lemma}
\begin{proof}
The proof is divided into four steps.

\noindent{\it Step 1: Reduction}.
We introduce the stopping time
\begin{align*}
    \widetilde{\sigma}=\inf\{t\geq0:\widetilde{u}^\prime_t\neq\widetilde{v}_t\}.
\end{align*}
By the definition of $l_0(1)$,
\begin{align}\label{03150320}
  \mathbb{P}\big(l_0(1)\neq 0\big)
  &=\mathbb{P}\big(\widetilde{\sigma}\wedge
  \tau_1^{\widetilde{u},\widetilde{u}^\prime}<T\big)
  \leq \mathbb{P}\big(\tau_1^{\widetilde{u},\widetilde{u}^\prime}<T\big)
  +\mathbb{P}\big(\widetilde{\sigma}<T\big),
\end{align}
and the first term on the right-hand side is at most $2\mathcal{C}_4\rho^{-1}$,
by \eqref{04230821} applied to $\widetilde{u}$ and $\widetilde{u}^\prime$
together with $\mathcal{C}_4^0\leq\mathcal{C}_4$. It remains to estimate
$\mathbb{P}(\widetilde{\sigma}<T)$.

\noindent{\it Step 2: Truncation}.
Since $u_0,u_0^\prime\in B_{H^2}(d_0)$, the event $(P_{0,0})$ holds, so that~$l_0(0)=0$ and the first interval is driven by the maximally coupled pair;
hence \eqref{04280357} applies:
\begin{align}\label{04250421}
   \nonumber\mathbb{P}\big(\widetilde{\sigma}<T\big)
   &=\|\mathcal{D}(\widetilde{u}^\prime)-\mathcal{D}(\widetilde{v})\|_{\mathrm{var}}
   =\sup_{\Lambda}\big|\mathbb{P}(\widetilde{u}^\prime\in\Lambda)
   -\mathbb{P}(\widetilde{v}\in\Lambda)\big|\\
   &\leq\sup_{\Lambda}\big|\mathbb{P}\big(\widetilde{u}^\prime\in\Lambda,
   \tau_1^{\widetilde{u}^\prime}\geq T\big)
   -\mathbb{P}\big(\widetilde{v}\in\Lambda,\tau_1^{\widetilde{v}}\geq T\big)\big|
   \nonumber\\&\quad+\mathbb{P}\big(\tau_1^{\widetilde{u}^\prime}<T\big)
   +\mathbb{P}\big(\tau_1^{\widetilde{v}}<T\big),
\end{align}
the supremum being taken over $\Lambda\in\mathcal{B}(C([0,T];H^2))$.

Let $\hat{u}^\prime$ and $\hat{v}$ denote the solutions of \eqref{trunc eqn 2}
and \eqref{auxiliary 1} with $bW$ replaced by~$b\widetilde{W}^2$ and
$b\widetilde{W}^1$, respectively, as in the proof of
Proposition~\ref{03150137}. Then~$\hat{u}^\prime=\widetilde{u}^\prime$ on
$[0,T]$ on the event $\{\tau_1^{\widetilde{u}^\prime}\geq T\}$, and likewise for
$\hat{v}$ and $\widetilde{v}$. Moreover, $\tau_1$ is a measurable functional of the trajectory, so there is
$A\in\mathcal{B}(C([0,T];H^2))$ such that $\{\tau_1^w\geq T\}=\{w\in A\}$ for
any $C([0,T];H^2)$-valued process $w$. Taking~$\hat{\Lambda}:=\Lambda\cap A$, we
obtain
\begin{align*}
\big\{\hat{u}^\prime\in\Lambda,\tau_1^{\hat{u}^\prime}\geq T\big\}
=\big\{\hat{u}^\prime\in\hat{\Lambda}\big\},
\qquad
\big\{\hat{v}\in\Lambda,\tau_1^{\hat{v}}\geq T\big\}
=\big\{\hat{v}\in\hat{\Lambda}\big\},
\end{align*}
whence the first term on the right-hand side of the inequality
\eqref{04250421} is bounded by
$\|\mathcal{D}(\hat{u}^\prime)-\mathcal{D}(\hat{v})\|_{\mathrm{var}}$.

\noindent{\it Step 3: Estimate of the three terms}.
Since $b\widetilde{W}^1$ and $b\widetilde{W}^2$ have the law of~$bW,$
we deduce from \eqref{04240259}, \eqref{04230822}, and \eqref{03200856}
that
\begin{align}\label{04270423}
 \nonumber\|\mathcal{D}(\hat{u}^\prime)-\mathcal{D}(\hat{v})\|_{\mathrm{var}}
 &\leq\|\Phi^{u_0,u_0^\prime}_*\mathcal{D}(bW)-\mathcal{D}(bW)\|_{\mathrm{var}}\\
 &\leq\mathcal{C}_4\rho^{-1}+e^{-a}
 +\frac{1}{2}\sqrt{\exp\big\{\mathcal{C}_4C^{*}(d_0)
 \exp\{\mathcal{C}_4\rho+a\}\big\}-1},
\end{align}
while \eqref{04230821}, the inequality $\mathcal{C}_4^0\leq\mathcal{C}_4$, and
Proposition~\ref{03150137} give
\begin{align}\label{05200947}
 \nonumber\mathbb{P}\big(\tau_1^{\widetilde{u}^\prime}<T\big)
 +\mathbb{P}\big(\tau_1^{\widetilde{v}}<T\big)
 &\leq 2\mathcal{C}_4\rho^{-1}+e^{-a}\\
 &\quad+\frac{1}{2}\sqrt{\exp\big\{\mathcal{C}_{4}C^{*}(d_0)
 \exp\{\mathcal{C}_4\rho+a\}\big\}-1}.
\end{align}

\noindent{\it Step 4: Choice of the parameters}.
Substituting \eqref{04270423} and \eqref{05200947} into \eqref{04250421}, and
adding the bound obtained in Step~1, we arrive at
\begin{align}\label{08091111}
    \mathbb{P}\big(l_0(1)\neq0\big)\leq 5\mathcal{C}_4\rho^{-1}+2e^{-a}
    +\sqrt{\exp\big\{\mathcal{C}_{4}C^{*}(d_0)
    \exp\{\mathcal{C}_4\rho+a\}\big\}-1}.
\end{align}
  We first choose $a>0$
and $\rho_0\ge\rho_*$ so large that
\begin{align}\label{08091114}
    5\mathcal{C}_4\rho_0^{-1}\leq\frac{1}{6},\qquad 2e^{-a}\leq\frac{1}{6},
\end{align}
and then, for any $\rho\geq\rho_0$, we choose $d_0=d_0(\rho)\in(0,1)$ so small
that
\begin{align}\label{04160442}
    \sqrt{\exp\big\{\mathcal{C}_4C^*(d_0)\exp\{\mathcal{C}_4\rho+a\}\big\}-1}
    \leq\frac{1}{6},
\end{align}
which is possible since $C^{*}(d_0)\to0$ as $d_0\to0$. Combining \eqref{08091111}--\eqref{04160442}, we arrive at the required inequality.
\end{proof}

From now on, $d_0=d_0(\rho)$ denotes the function constructed in
Lemma~\ref{03151059}. The next lemma provides an estimate for the probability that the coupling attempt
launched at time~$0$ survives up to step~$k$ and fails at
step~$k+1$.
\begin{lemma}\label{03160423}
For any $q\geq5$, $\rho\geq\rho_0$, $T>0$,  $k\geq1$, and
$u_0,u_0^\prime\in B_{H^2}(d_0)$ with $d_0=d_0(\rho)$, we have
\begin{align*}
\mathbb{P}\big(l_0(k+1)\neq0,\ l_0(k)=0\big)
\leq C_\rho\exp\Big\{-\frac{\alpha kT}{8}\Big\}
+C_q\big(\rho+kT\big)^{-\frac{q}{2}+1},
\end{align*}
with some positive constants $C_\rho$ and $C_q$. 
\end{lemma}
\begin{proof}
    {\it Step 1: Reduction.} For any $u_0,u_0^\prime\in B_{H^2}(d_0)$, we have
    \begin{gather*}
          \big\{l_0(k+1)\neq 0,l_0(k)=0\big\}=  \big\{{\tau}_1^{\widetilde{u},\widetilde{u}^\prime}\in[kT,(k+1)T),\widetilde{u}^\prime|_{[0,kT]}=\widetilde{v}|_{[0,kT]}\big\}\\{\scalebox{1.3}{$\cup$}}\big\{\widetilde{u}^\prime_t\neq\widetilde{v}_t \   \exists\hspace{1mm} t\in[kT,(k+1)T),\widetilde{u}^\prime|_{[0,kT]}=\widetilde{v}|_{[0,kT]},\tau_1^{\widetilde{u},\widetilde{u}^\prime}\geq (k+1)T\big\}.
    \end{gather*}Therefore,
     \begin{gather*}
         \mathbb{P}\Big(l_0(k+1)\neq 0,l_0(k)=0\Big)  \leq\mathbb{P}\Big({\tau}_1^{\widetilde{u},\widetilde{u}^\prime}\in[kT,(k+1)T)\Big)\\+\,\mathbb{P}\Big(kT\leq\widetilde{\sigma}<(k+1)T,{\tau}_1^{\widetilde{u},\widetilde{u}^\prime}\geq (k+1)T\Big).
    \end{gather*}
For the first term on the right-hand side, since $u_0,u_0^\prime\in B_{H^2}(d_0)$ with $d_0\leq1$, Lemma~\ref{03120055} implies that    \begin{align}\label{04271129}
        \mathbb{P}\Big({\tau}_1^{\widetilde{u},\widetilde{u}^\prime}\in[kT,(k+1)T)\Big)\leq C_q(\rho+kT)^{-\frac{q}{2}+1}.
    \end{align}
To estimate the second term, we define, as in \eqref{04080050} and
\eqref{06230835-1}, the stopping times $\tau_1^{u,(k)}$, $\tau_1^{v,(k)},$ and
$\tau_1^{u,v,(k)}$ associated with the processes $u$ and $v$ issued from
$w,w^\prime\in H^2$:
\begin{align*}
    \tau_{1}^{u,(k)}&:=\min\{ \tau_{0,6}^{u,(k)}, \tau_{0,15}^{u,(k)}, \tau_{1,2}^{u,(k)}, \tau_{1,15}^{u,(k)}, \tau_{2,2}^{u,(k)}, \tau_{2,3}^{u,(k)}, \tau_{\psi}^{u,(k)}\},\\\ \tau_{1}^{u,v,(k)}&:=\tau_{1}^{u,(k)}\wedge\tau_{1}^{v,(k)},
\end{align*}    
   where $\tau_{0,6}^{u,(k)}$ is defined by
  \begin{align}\label{03190353-k}
\tau_{0,6}^{u,(k)}:=\inf\big\{t\geq 0:\mathcal{E}_{0,6}^u(t)\geq
\rho+\epsilon kT+Kt+\mathscr{C}\|w\|^{12}\big\},\end{align}
    with a constant $\epsilon\in(0,1]$ to be chosen, and
    $\tau_{0,15}^{u,(k)}$, $\tau_{1,2}^{u,(k)}$, $\tau_{1,15}^{u,(k)}$,
    $\tau_{2,2}^{u,(k)}$, $\tau_{2,3}^{u,(k)}$, and $\tau^{u,(k)}_{\psi}$ are
    obtained from \eqref{03190353-1}, \eqref{03190353-2}, \eqref{03190353-3},
    and \eqref{03190353-4} by replacing~$\rho$ with $\rho+\epsilon kT$.
Moreover, as in \eqref{04160408}, we define
    \begin{gather*}
 {\tau}_{0}^{{u},{v},(k)}:=\inf\big\{0\leq t\leq T :\int_{0}^{t\wedge{\tau}^{{u},{v},(k)}_{1}}\big(1+\mathcal{H}({u}_s)+\mathcal{H}({v}_s)\big)\widetilde{\mathcal{M}}(s,{u}_s-{v}_s)\dd s\\\geq  \mathcal{C}_3\widetilde{\mathcal{M}}(0,w-w^\prime)e^{\mathcal{C}_3(\rho+\epsilon { kT})+\mathcal{C}_3\mathscr{C}(F(w)+F(w^\prime))+ \frac{\alpha kT}{8}}\big\}\wedge T.\end{gather*}
Since $\widetilde{v}$ is a measurable functional of
$(\widetilde{u},\widetilde{u}^\prime)$ by property~(ii) of
Subsection~\ref{S:4.1}, the indicator $\mathbf{1}_{\widetilde{\sigma}<T}$ is of the form
$\widetilde{g}(\widetilde{u},\widetilde{u}^\prime)$ for some non-negative
measurable
$\widetilde{g}:C([0,T];H^2)^2\rightarrow\mathbb{R}$. Moreover,
$$
\{\widetilde{\sigma}\wedge\tau_1^{\widetilde{u},\widetilde{u}^\prime}
\geq kT\}=\{l_0(k)=0\}\subseteq\{l_0(k)<\infty\},
$$
so
Lemma~\ref{local markov} applies to $\widetilde{g}$ on this event, and
we obtain   
    \begin{align}\label{03150923}
          \nonumber\mathbb{P}\Big(kT\leq\widetilde{\sigma}<(k+1)T; &\hspace{1mm}{\tau}_1^{\widetilde{u},\widetilde{u}^\prime}\geq (k+1)T\Big)\\\leq&\ \mathbb{E}\big[\mathbb{P}\big(\widetilde{\sigma}\circ\theta_{kT}\in[0,T)\big|\widetilde{\mathcal{F}}_{kT}\big);\widetilde{\sigma}\wedge\tau_1^{\widetilde{u},\widetilde{u}^\prime}\geq kT\big]\nonumber\\=&\hspace{1mm}\mathbb{E}\big[\textbf{P}_{\widetilde{u}_{kT},\widetilde{u}^\prime_{kT}}\big(\widetilde{\sigma}\in[0,T)\big);\widetilde{\sigma}\wedge\tau_1^{\widetilde{u},\widetilde{u}^\prime}\geq kT\big].
    \end{align}
 To estimate $\mathbf{P}_{\widetilde{u}_{kT},\widetilde{u}^\prime_{kT}}
\big(\widetilde{\sigma}\in[0,T)\big)$, we argue as in
Proposition~\ref{03150137}. We denote 
    $$
    \tau_2^{u,v,(k)}:=\tau_1^{u,v,(k)}\wedge\tau_0^{u,v,(k)}
    $$ 
and define the transformations
       $$
    {\Phi}_{(k)}^{w,w^\prime},   \widetilde{\Phi}_{(k)}^{w,w^\prime}: C_0([0,T];H^4) \rightarrow C_0([0,T];H^4),
    $$  by
\begin{align}
    &\nonumber{\Phi}_{(k)}^{w,w^\prime}(\omega)_t:= \omega_t+\int_0^t\mathbf{1}_{s\leq\tau_1^{v,(k)}}\mathsf{P}_N\big[i|\hat{v}|^2\hat{v}-i|{u}|^2{u}-i\partial_x^2(\hat{v}-{u})\big]\dd s,\\
    &\widetilde{\Phi}_{(k)}^{w,w^\prime}(\omega)_t:= \omega_t+\int_0^t\mathbf{1}_{s\leq\tau^{u,v,(k)}_2}\mathsf{P}_N\big[i|\hat{v}|^2\hat{v}-i|{u}|^2{u}-i\partial_x^2(\hat{v}-{u})\big]\dd s,\nonumber
\end{align}
where $u$ is the solution to \eqref{equation-1} and $\hat{v}$ is the one for \eqref{auxiliary 1} but with $\tau_1^v$ replaced by $\tau_1^{v,(k)}$. Both of them are driven by $bW:=\omega$.

 We also denote by $\hat{u}^{k,\prime}$ and $\hat{v}^{k}$ the solutions to
\eqref{trunc eqn 2} and \eqref{auxiliary 1}, both issued from
$\widetilde{u}^\prime_{kT}$ and driven by
$b\widetilde{W}^{2}(\widetilde{u}_{kT},\widetilde{u}^\prime_{kT})$ and
$b\widetilde{W}^{1}(\widetilde{u}_{kT},\widetilde{u}^\prime_{kT})$,
respectively. In~\eqref{auxiliary 1} the feedback term is determined by the solution issued
from $\widetilde{u}_{kT}$ and driven by the same noise, and $\tau_1^{u^\prime}$,
$\tau_1^{v}$ are replaced by $\tau_1^{u^\prime,(k)}$, $\tau_1^{v,(k)}$. As in \eqref{04250421},
we have
\begin{align}\label{04270418}
\nonumber\mathbf{P}_{\widetilde{u}_{kT},\widetilde{u}^\prime_{kT}}\big(\widetilde{\sigma}\in[0,T)\big)& \leq  \mathbf{P}_{\widetilde{u}_{kT},\widetilde{u}^\prime_{kT}}\big(\tau_1^{u^\prime,(k)}<T\big)+\mathbf{P}_{\widetilde{u}_{kT},\widetilde{u}^\prime_{kT}}\big(\tau_1^{v,(k)}<T\big)\\&\quad+\|\mathcal{D}(\hat{u}^{k,\prime})-\mathcal{D}(\hat{v}^{k})\|_{\mathrm{var}}=:\mathrm{I+II+III}.
\end{align}
Thus it suffices to estimate the expectations of $\mathrm{I}$,
$\mathrm{II}$ and $\mathrm{III}$ over the event
$\{\widetilde{\sigma}\wedge\tau_1^{\widetilde{u},\widetilde{u}^\prime}
\geq kT\}$.

\noindent{\it Step 2: Estimates for $\mathrm{I}$--$\mathrm{III}$.}
We first consider the event $\{\widetilde{u}_{kT}=\widetilde{u}^\prime_{kT}\}$,
on which the maximal coupling is diagonal and
$\widetilde{u}_t=\widetilde{u}^\prime_t=\widetilde{v}_t$ for any
$t\in[kT,(k+1)T]$, as in Step~1 of the proof of
Proposition~\ref{03150137}. Hence $\mathrm{I}=\mathrm{II}$.
Lemma~\ref{03120055} gives
\begin{align}\label{08101219}
\nonumber\mathrm{I}=\mathrm{II}&=\mathbf{P}_{\widetilde{u}_{kT},\widetilde{u}^\prime_{kT}}\big(\tau_1^{u^\prime,(k)}<T\big)\\&\leq  C_q\big(1+\|\widetilde{u}_{kT}^\prime\|^{36q}+\mathcal{H}^{18q}(\widetilde{u}^\prime_{kT})+\mathcal{G}^{\frac{7q}{2}}(\widetilde{u}_{kT}^\prime)\big)\big(\rho+\epsilon kT\big)^{-\frac{q}{2}+1}.
\end{align}
Moreover, on this event the feedback term in \eqref{auxiliary 1} vanishes
identically, so that $\hat{u}^{k,\prime}$ and $\hat{v}^{k}$ are the solutions
of the same equation \eqref{trunc eqn 2}, with the same initial condition
$\widetilde{u}^\prime_{kT}$, driven respectively by $b\widetilde{W}^{2}$ and
$b\widetilde{W}^{1}$. Since these two noises have the same law, so do
$\hat{u}^{k,\prime}$ and $\hat{v}^{k}$, and therefore $\mathrm{III}=0$. Combining this with \eqref{08101219}, we obtain, on~$\{\widetilde{u}_{kT}= \widetilde{u}^\prime_{kT}\}$,
\begin{align}\label{08101219-1}
\nonumber\mathbf{P}_{\widetilde{u}_{kT},\widetilde{u}^\prime_{kT}}\big(\widetilde{\sigma}\in[0,T)\big)\leq&\hspace{1mm}  C_q\big(1+\|\widetilde{u}_{kT}^\prime\|^{36q}+\mathcal{H}^{18q}(\widetilde{u}^\prime_{kT})+\mathcal{G}^{\frac{7q}{2}}(\widetilde{u}_{kT}^\prime)\big)\\&\times\big(\rho+\epsilon kT\big)^{-\frac{q}{2}+1}.
\end{align}
We now turn to the event
$\{\widetilde{u}_{kT}\neq\widetilde{u}^\prime_{kT}\}$. As in
\eqref{08101219}, Lemma~\ref{03120055} gives \begin{align}\label{04270310}
\mathrm{I}&\leq  C_q\big(1+\|\widetilde{u}_{kT}^\prime\|^{36q}+\mathcal{H}^{18q}(\widetilde{u}^\prime_{kT})+\mathcal{G}^{\frac{7q}{2}}(\widetilde{u}_{kT}^\prime)\big)\big(\rho+\epsilon kT\big)^{-\frac{q}{2}+1}.
    \end{align}
For the term $\mathrm{II}$, the argument leading to \eqref{03140952} gives
\begin{align*}
\mathbf{P}_{\widetilde{u}_{kT},\widetilde{u}^\prime_{kT}}\big(\tau_1^{{v},(k)}< T\big)&\hspace{-0.5mm}\leq \hspace{-0.5mm} \mathbf{P}_{\widetilde{u}_{kT},\widetilde{u}^\prime_{kT}}(\tau_1^{{u}^\prime,(k)}\hspace{-0.5mm}<\hspace{-0.5mm}T)\hspace{-0.5mm}+\hspace{-0.5mm}\|\widetilde{\Phi}_{(k),*}^{\widetilde{u}_{kT},\widetilde{u}^\prime_{kT}}\mathcal{D}(bW)-\mathcal{D}(bW)\|_{\mathrm{var}}\\&\quad\hspace{-0.5mm}+\hspace{-0.5mm}\|\widetilde{\Phi}_{(k),*}^{\widetilde{u}_{kT},\widetilde{u}^\prime_{kT}}\mathcal{D}(bW)\hspace{-0.5mm}-\hspace{-0.5mm}{\Phi}_{(k),*}^{\widetilde{u}_{kT},\widetilde{u}^\prime_{kT}}\mathcal{D}(bW)\|_{\mathrm{var}}\\&=: \mathrm{II_1+II_2+II_3}.
\end{align*}
Notice that $\mathrm{II}_1=\mathrm{I}$. For $\mathrm{II}_2$, arguing as
in the derivation of \eqref{03151055} and \eqref{03200856}, we obtain
\begin{align*}
    \mathrm{II}_2\leq &\hspace{1mm}\frac{1}{2}\sqrt{\mathrm{exp}\big\{\mathcal{C}_4\widetilde{\mathcal{M}}(kT,\widetilde{u}_{kT}\hspace{-0.5mm}-\hspace{-0.5mm}\widetilde{u}^\prime_{kT})e^{\frac{\alpha kT}{8}+\mathcal{C}_4(\rho+\epsilon kT)+\mathcal{C}_4\mathscr{C}(F(\widetilde{u}_{kT})+F(\widetilde{u}_{kT}^\prime))}\big\}-1}.
\end{align*}
Define $F_*:H^2\rightarrow\mathbb{R}$ by
\begin{align*}
F_*(u):=\|u\|^{30}+\mathcal{H}^{15}(u)+\mathcal{G}^{3}(u).
\end{align*}
Then, with $\epsilon\in(0,1]$ the constant from \eqref{03190353-k},
still to be chosen, $\mathrm{II}_2$ can be bounded by
\begin{align*}
 \Big\{\frac{1}{2}\sqrt{\mathrm{exp}\big\{\hspace{-0.5mm}C_{\epsilon}\widetilde{\mathcal{M}}(kT,\widetilde{u}_{kT}\hspace{-0.5mm}-\hspace{-0.5mm}\widetilde{u}^\prime_{kT})e^{\frac{\alpha kT}{8}+\mathcal{C}_4(\rho+\epsilon kT)+\epsilon\mathscr{C}({F}_*(\widetilde{u}_{kT})+{F}_*(\widetilde{u}_{kT}^\prime))}\big\}\hspace{-0.5mm}-\hspace{-0.5mm}1}\Big\}\wedge 1
\end{align*}
for some $C_\epsilon>0$. Notice that, on the event $\{\widetilde{\sigma}\wedge\tau_1^{\widetilde{u},\widetilde{u}^\prime}\geq kT\}$, we have
\begin{align*}
   {F}_*(\widetilde{u}_{kT})+{F}_*(\widetilde{u}_{kT}^\prime)\leq C\big(\rho+KkT+d_0\big),
\end{align*}
where $C>0$ does not depend on $\rho$, $k$, $T$ or $\epsilon$.
Choosing
$\epsilon:=\frac{\alpha}{8[\mathcal{C}_4+C\mathscr{C}K]}\wedge1$, we obtain, on
$\{\widetilde{\sigma}\wedge\tau_1^{\widetilde{u},\widetilde{u}^\prime}
\geq kT\}$, 
\begin{align}\label{08101227}
    \mathrm{II}_2\leq \Big\{\frac{1}{2}\sqrt{\mathrm{exp}\big\{C_\rho\widetilde{\mathcal{M}}(kT,\widetilde{u}_{kT}-\widetilde{u}^\prime_{kT})e^{\frac{\alpha kT}{4}}\big\}-1}\Big\}\wedge 1.
\end{align}
To estimate $\mathrm{II}_3$, we argue as in the derivation of \eqref{04270250}
and obtain
\begin{align*}
   \mathrm{II}_3&\leq   \mathbf{P}_{w,w^\prime}\big({\Phi}_{(k)}^{w,w^\prime}(bW)\neq\widetilde{\Phi}_{(k)}^{w,w^\prime}(bW) \big)\big|_{(w,w^\prime)=(\widetilde{u}_{kT},\widetilde{u}_{kT}^\prime)}\\&\leq \mathbf{P}_{\widetilde{u}_{kT},\widetilde{u}^\prime_{kT}}\big(\tau_1^{{u},(k)}< T\big)+\mathbf{P}_{\widetilde{u}_{kT},\widetilde{u}^\prime_{kT}}\big(\tau_{0}^{{u},{v},(k)}<\tau_{1}^{{u},{v},(k)}\wedge T\big).
\end{align*}
Applying Lemma~\ref{03120055} to the first term, and Proposition~\ref{03100846} together with Chebyshev's inequality to the second, we obtain
\begin{align*}
    \mathrm{II}_3\leq \hspace{0.5mm} C_q\big(1+\|\widetilde{u}_{kT}\|^{36q}+\mathcal{H}^{18q}(\widetilde{u}_{kT})+\mathcal{G}^{\frac{7q}{2}}(\widetilde{u}_{kT})\big)\big(\rho+\epsilon kT\big)^{-\frac{q}{2}+1}+e^{-\frac{\alpha kT}{8}}.
\end{align*}
Combining this with \eqref{04270310}, \eqref{08101227}, and the identity
$\mathrm{II}_1=\mathrm{I}$, we obtain
\begin{align}\label{08101117}
    \mathrm{II}\leq&\nonumber\hspace{1mm} C_q\big(1+\|\widetilde{u}_{kT}\|^{36q}+\|\widetilde{u}^\prime_{kT}\|^{36q}+\mathcal{H}^{18q}(\widetilde{u}_{kT})+\mathcal{H}^{18q}(\widetilde{u}^\prime_{kT})\\&\hspace{0.9cm}+\mathcal{G}^{\frac{7q}{2}}(\widetilde{u}_{kT})+\mathcal{G}^{\frac{7q}{2}}(\widetilde{u}_{kT}^\prime)\big)\big(\rho+\epsilon kT\big)^{-\frac{q}{2}+1}+e^{-\frac{\alpha kT}{8}}\nonumber\\&\hspace{0.9cm}+\Big\{\frac{1}{2}\sqrt{\mathrm{exp}\big\{C_\rho\widetilde{\mathcal{M}}(kT,\widetilde{u}_{kT}-\widetilde{u}^\prime_{kT})e^{\frac{\alpha kT}{4}}\big\}-1}\Big\}\wedge 1.
\end{align}
For the term $\mathrm{III}$, arguing as in the derivation
of \eqref{04270423}, we find
\begin{align}\label{08100126}
    \nonumber\mathrm{III}&=\|\mathcal{D}(\hat{u}^{k,\prime})-\mathcal{D}(\hat{v}^{k})\|_{\mathrm{var}}\leq \|\Phi_{(k),*}^{\widetilde{u}_{kT},\widetilde{u}_{kT}^\prime}\mathcal{D}(bW)-\mathcal{D}(bW)\|_{\mathrm{var}}\\&\nonumber\le \|\widetilde{\Phi}_{(k),*}^{\widetilde{u}_{kT},\widetilde{u}_{kT}^\prime}\mathcal{D}(bW)\hspace{-0.5mm}-\hspace{-0.5mm}\mathcal{D}(bW)\|_{\mathrm{var}}\hspace{-0.5mm}+\hspace{-0.5mm}\|\widetilde{\Phi}_{(k),*}^{\widetilde{u}_{kT},\widetilde{u}_{kT}^\prime}\mathcal{D}(bW)\hspace{-0.5mm}-\hspace{-0.5mm}{\Phi}_{(k),*}^{\widetilde{u}_{kT},\widetilde{u}_{kT}^\prime}\mathcal{D}(bW)\|_{\mathrm{var}}\\&=\mathrm{II}_2+\mathrm{II}_3,
\end{align}
so that $\mathrm{III}$ is bounded by the sum of \eqref{08101227} and the estimate for $\mathrm{II}_3$; adding these bounds to \eqref{04270310} and \eqref{08101117}, we find on the event~$\{\widetilde{u}_{kT}\neq\widetilde{u}^\prime_{kT}\}$,
\begin{align}\label{08100149}
    \nonumber \mathbf{P}_{\widetilde{u}_{kT},\widetilde{u}^\prime_{kT}}\big(\widetilde{\sigma}\in[0,T)\big)\leq&\nonumber\hspace{1mm} C_q\big(1\hspace{-0.5mm}+\hspace{-0.5mm}\|\widetilde{u}_{kT}\|^{36q}\hspace{-0.5mm}+\hspace{-0.5mm}\|\widetilde{u}^\prime_{kT}\|^{36q}\hspace{-0.5mm}+\hspace{-0.5mm}\mathcal{H}^{18q}(\widetilde{u}_{kT})\hspace{-0.5mm}+\hspace{-0.5mm}\mathcal{H}^{18q}(\widetilde{u}^\prime_{kT})\\&+\mathcal{G}^{\frac{7q}{2}}(\widetilde{u}_{kT})+\mathcal{G}^{\frac{7q}{2}}(\widetilde{u}_{kT}^\prime)\big)\big(\rho+\epsilon kT\big)^{-\frac{q}{2}+1}+2e^{-\frac{\alpha kT}{8}}\nonumber\\&+\Big\{\sqrt{\mathrm{exp}\big\{C_\rho\widetilde{\mathcal{M}}(kT,\widetilde{u}_{kT}-\widetilde{u}^\prime_{kT})e^{\frac{\alpha kT}{4}}\big\}-1}\Big\}\wedge 2.
\end{align}  Together with \eqref{08101219-1}, this shows that \eqref{08100149}
holds on the whole sample space~$\Omega$.

\noindent{\it Step 3: Conclusion.}
We denote the first and third terms on the right-hand side of \eqref{08100149}
by $\mathrm{I}^*$ and $\mathrm{III}^*$. By \eqref{03150923}, it remains to
estimate their expectations over the event
$\{\widetilde{\sigma}\wedge\tau_1^{\widetilde{u},\widetilde{u}^\prime}\geq kT\}$.
It follows from Lemmas~\ref{02280441}(i),~\ref{02280436}(i),
and~\ref{02231236}(i), together with the fact that
$u_0,u_0^\prime\in B_{H^2}(d_0)$ with $d_0\leq1$, that
\begin{align}\label{08100238}
\mathbb{E}\big[\mathrm{I^*};\widetilde{\sigma}\wedge\tau_1^{\widetilde{u},\widetilde{u}^\prime}\geq kT\big]\leq \mathbb{E}\big[\mathrm{I^*}\big]\leq C_q(\rho+\epsilon kT)^{-\frac{q}{2}+1}.
\end{align}
To estimate $\mathrm{III}^*$, notice that on the event
$\{\widetilde{\sigma}\wedge\tau_1^{\widetilde{u},\widetilde{u}^\prime}\geq kT\}$
we have $\widetilde{u}^\prime=\widetilde{v}$ throughout $[0,kT]$, and that this
event is contained in $\{\tau_1^{\widetilde{u},\widetilde{v}}\geq kT\}$. Hence,  
\begin{align*}
    \mathbb{E}\big[\widetilde{\mathcal{M}}(kT,\widetilde{u}_{kT}-\widetilde{u}^\prime_{kT});\widetilde{\sigma}\hspace{-0.5mm}\wedge\hspace{-0.5mm}\tau^{\widetilde{u},\widetilde{u}^\prime}_1\geq kT\big]&\hspace{-0.5mm}\leq\hspace{-0.5mm}\mathbb{E}\big[\widetilde{\mathcal{M}}(kT,\widetilde{u}_{kT}-\widetilde{v}_{kT});{\widetilde{\sigma}}\wedge\tau^{\widetilde{u},\widetilde{v}}_1\geq kT\big].
\end{align*}
Since on $\{\widetilde{\sigma}\wedge\tau_1^{\widetilde{u},\widetilde{v}}\geq kT\}$, we have for any $t\in[0,kT]$,
\begin{align*}
    &\widetilde{u}_t=u(t;u_0,\widetilde{\eta}^1)=:u^1_t,\ \widetilde{v}_t=v(t;u_0,u_0^\prime,\widetilde{\eta}^1)=:v^1_t,\text{ and } \widetilde{u}^\prime_t=v_t^1.
\end{align*}
Thus we have by \eqref{07312337} that
\begin{align*}
\mathbb{E}\big[\widetilde{\mathcal{M}}(kT,\widetilde{u}_{kT}-& \widetilde{u}^\prime_{kT});\widetilde{\sigma}\hspace{-0.5mm}\wedge\hspace{-0.5mm}\tau^{\widetilde{u},\widetilde{u}^\prime}_1\geq kT\big]\\\leq&\hspace{1mm}\mathbb{E}\big[\widetilde{\mathcal{M}}(kT,u^1_{kT}-v^1_{kT});\tau_1^{u^1,v^1}\geq kT\big]\\\leq&  C_\rho \widetilde{\mathcal{M}}(0,u_0-u_0^\prime)e^{-\frac{3\alpha kT}{4}+\mathcal{C}_3\mathscr{C}(F(u_0)+F(u_0^\prime))}\leq C_\rho e^{-\frac{3\alpha kT}{4}}C^{*}(d_0).
\end{align*}
An application of Chebyshev's inequality implies that for $d_0=d_0(\rho)$,
\begin{align*}
    \mathbb{P}\big(\widetilde{\mathcal{M}}(kT,\widetilde{u}_{kT}-\widetilde{u}^\prime_{kT})\geq e^{-\frac{\alpha kT}{2}};\widetilde{\sigma}\wedge\tau_1^{\widetilde{u},\widetilde{v}}\geq kT\big)\leq C_\rho e^{-\frac{\alpha kT}{4}}C^*(d_0).
\end{align*}
Splitting the event $\{\widetilde{\sigma}\wedge\tau_1^{\widetilde{u},\widetilde{v}}\geq kT\}$
according to whether
$\widetilde{\mathcal{M}}(kT,\widetilde{u}_{kT}-\widetilde{u}^\prime_{kT})$
is larger or smaller than $e^{-\frac{\alpha kT}{2}}$, we obtain
\begin{align}\label{04270358}
    \nonumber \mathbb{E}\big[\mathrm{III}^*;\widetilde{\sigma}\wedge\tau_1^{\widetilde{u},\widetilde{u}^\prime}\geq kT]\leq&\hspace{1mm}  \sqrt{\mathrm{exp}\{C_\rho e^{-\frac{\alpha kT}{4}}\}-1}+C_\rho e^{-\frac{\alpha kT}{4}}C^*(d_0)\\\leq& \ C_\rho e^{-\frac{\alpha kT}{8}}.
\end{align}
Combining \eqref{08100149}--\eqref{04270358}, we
arrive at
\begin{align*}
    \mathbb{P}\Big(kT\leq\widetilde{\sigma}<(k+1)T;\hspace{1mm}{\tau}_1^{\widetilde{u},\widetilde{u}^\prime}\geq (k+1)T\Big)\leq C_{q}(\rho+ \epsilon kT)^{-\frac{q}{2}+1}+C_\rho e^{-\frac{\alpha k T}{8}}.
\end{align*}
This, together with \eqref{04271129}, completes the proof of the lemma.
\end{proof}
As a consequence of Lemmas~\ref{03151059} and~\ref{03160423}, for a suitable
choice of $\rho$ and $T$ the trajectories $\widetilde{u}^\prime$ and
$\widetilde{v}$ coincide for all times with probability at least~$1/4$. To
formulate this precisely, we set
\begin{align*}
    l_0(\infty):=\limsup_{k\rightarrow\infty} l_0(k).
\end{align*}
\begin{corollary}\label{07311420}
There exist $\rho^\sharp\ge\rho_0$ and $T^\sharp\geq1$ such that, for
$\rho= \rho^\sharp$, any $T\geq T^\sharp$, and any
$u_0,u_0^\prime\in B_{H^2}(d_0)$ with $d_0=d_0(\rho^\sharp)$, we have
\begin{align*}
\mathbb{P}\big(l_0(\infty)\neq 0\big)\leq \frac{3}{4}.
\end{align*}
\end{corollary}
\begin{proof}
Since the events $\{l_0(k)=0\}=(P_{0,k})$ are non-increasing in~$k$,
decomposing $\{l_0(\infty)\neq0\}$ according to the first step at which
the attempt that was launched at time~$0$ fails, we obtain
$$
\big\{l_0(\infty)\neq0\big\}
=\big\{l_0(1)\neq0\big\}{\scalebox{1.3}{$\cup$}}
\bigcup_{k=1}^{\infty}\big\{l_0(k+1)\neq 0,\ l_0(k)=0\big\}.
$$
Applying Lemma \ref{03151059} and taking $q=5$ and $T\ge1$ in Lemma \ref{03160423}, we obtain
\begin{align}\label{03160235}
    \nonumber\mathbb{P}\Big(l_0(\infty)\neq 0\Big)&= \mathbb{P}\big(l_0(1)\neq 0\big)+\sum_{k=1}^\infty\mathbb{P}\big(l_0(k+1)\neq 0,l_0(k)=0\big)\\&\leq \frac{1}{2}+C_\rho\sum_{k=1}^\infty\mathrm{exp}\big\{-\frac{\alpha kT}{8}\big\}+C\sum_{k=1}^\infty(\rho+k)^{-\frac{3}{2}}.
\end{align}
We first choose $\rho=\rho^\sharp\ge \rho_0$ so large that
\begin{align}\label{03160324}
    C\sum_{k=1}^\infty(\rho^\sharp+k)^{-\frac{3}{2}}<\frac{1}{8}.
\end{align}
We then choose $T^\sharp\geq1$ so large that
\begin{align}\label{03160325}
C_{\rho^\sharp}\sum_{k=1}^\infty
\exp\big\{-\frac{\alpha kT^\sharp}{8}\big\}<\frac{1}{8},
\end{align}
where $C_{\rho^\sharp}$ denotes the constant in \eqref{03160235}.  
Combining \eqref{03160235}--\eqref{03160325}, we complete the proof.
\end{proof}

The parameters $N$, $K$, $\mathscr{C}$, $\rho=\rho^\sharp$ and
$d_0=d_0(\rho^\sharp)$ are now fixed; the time step $T$ will be chosen
larger than $T^\sharp$ in the proof of the Main Theorem, which we give in the next subsection.

\subsection{Conclusion of the proof}\label{S:4.3}

In this subsection, we prove the Main Theorem. The strategy is to
follow the coupling process through successive coupling attempts, to
handle each attempt by the estimates of Subsection~\ref{S:4.2}, and to
combine the resulting bounds by means of the Markov-type property
established below.

We denote by $\mathbb{P}_{w,w^\prime}$ the distribution on
$C([0,\infty);H^2)^2$ of the coupling process constructed in
Subsection~\ref{S:4.1} and issued from $(w,w^\prime)$, and by
$\mathbb{E}_{w,w^\prime}$ the corresponding expectation. Note that, for $w,w^\prime\in B_{H^2}(d_0)$, the restriction of
$\mathbb{P}_{w,w^\prime}$ to $[0,T]$ coincides with the law
$\mathbf{P}_{w,w^\prime}$ introduced in Subsection~\ref{S:4.2}.
\begin{lemma}\label{L1:l0-markov} 
Let $\varsigma$ be a stopping time with respect to
$(\widetilde{\mathcal{F}}_{nT})_{n\geq0}$ satisfying $\{\varsigma<\infty\}\subseteq\{l_0(\varsigma)=\varsigma\}$. Then, for any non-negative measurable function
$g:C([0,\infty);H^2)^2\rightarrow\mathbb{R}$, we have
\begin{align}\label{E:l0-markov}
\mathbb{E}\big[g(\widetilde{u},\widetilde{u}^\prime)\circ\theta_{\varsigma T}
\,\big|\,\widetilde{\mathcal{F}}_{\varsigma T}\big]
=\mathbb{E}_{\widetilde{u}_{\varsigma T},\widetilde{u}^\prime_{\varsigma T}}
\big[g(\widetilde{u},\widetilde{u}^\prime)\big]
\text{ on }\{\varsigma<\infty\}.
\end{align}
\end{lemma}
\begin{proof}
By the monotone convergence theorem, it suffices to prove
\eqref{E:l0-markov} for bounded~$g$.  Since $\{\varsigma=n\}\in\widetilde{\mathcal{F}}_{\varsigma T}$, we have
\begin{align*}
   \mathbb{E}\big[g(\widetilde{u},\widetilde{u}^\prime)\circ\theta_{\varsigma T}
\,\big|\,\widetilde{\mathcal{F}}_{\varsigma T}\big]=\mathbb{E}\big[g(\widetilde{u},\widetilde{u}^\prime)\circ\theta_{n T}
\,\big|\,\widetilde{\mathcal{F}}_{nT}\big]\text{ on \{$\varsigma=n$\}}.
\end{align*}
Thus it suffices to prove that, for each $n\geq0$,
\begin{align}\label{08070655}
    \mathbb{E}\big[g(\widetilde{u},\widetilde{u}^\prime)\circ\theta_{n T}
\,\big|\,\widetilde{\mathcal{F}}_{nT}\big]=\mathbb{E}_{\widetilde{u}_{n T},\widetilde{u}^\prime_{n T}}
\big[g(\widetilde{u},\widetilde{u}^\prime)\big]\text{ on $\{\varsigma=n\}$}.
\end{align}
Recall from \eqref{08130433} that
$\widetilde{u}=u(\cdot\,;u_0,\widetilde{\eta}^1)$ and
$\widetilde{u}^\prime=u(\cdot\,;u_0^\prime,\widetilde{\eta}^2)$, where
$\widetilde{\eta}^1,\widetilde{\eta}^2$ are the noise maps defined in
\eqref{E:noise-concat}. Consequently, for any $t\geq 0$,
\begin{align}\label{08151052}
(\widetilde{u},\widetilde{u}^\prime)\circ\theta_{nT}(t)=\big({u}(t;\widetilde{u}_{nT},\theta_{nT}\widetilde{\eta}^1(u_0,u_0^\prime)),{u}(t;\widetilde{u}^\prime_{nT},\theta_{nT}\widetilde{\eta}^2(u_0,u_0^\prime))\big).
\end{align}On $\{\varsigma=n\}$, which belongs to
$\widetilde{\mathcal{F}}_{nT}$, we have $\widetilde{u}_{nT}$ and $\widetilde{u}^\prime_{nT}$ belong to $B_{H^2}(d_0)$. Thus the recursive construction \eqref{E:noise-concat} implies that, the conditional
distribution of $(\theta_{nT}\widetilde{\eta}^1(u_0,u_0^\prime),\theta_{nT}\widetilde{\eta}^2(u_0,u_0^\prime))$ given
$\widetilde{\mathcal{F}}_{nT}$ equals that of 
\begin{align*}
(\widetilde{\eta}^3(w,w^\prime),\widetilde{\eta}^4(w,w^\prime))|_{(w,w^\prime)=(\widetilde{u}_{nT},\widetilde{u}^\prime_{nT})},
\end{align*}
where $(\widetilde{\eta}^3,\widetilde{\eta}^4)$ is an independent copy of
the pair of noise maps $(\widetilde{\eta}^1,\widetilde{\eta}^2)$. The justification of this correspondence, and the derivation of
\eqref{08070655} from it, repeat the argument in the proof of
Lemma~\ref{local markov}, and we omit them.
\end{proof}

\begin{proof}[Proof of Main Theorem]  The existence of a stationary measure $\nu\in\mathcal{P}(H^2)$ is
established in Appendix~\ref{S:A4} (Lemma~\ref{05050844}), where it is
also shown that every stationary measure in $\mathcal{P}(H^2)$ has finite
moments of all orders in $H^2$ (Lemma~\ref{05050848}). It therefore
remains to prove the uniqueness
of~$\nu$ and the mixing estimate~\eqref{05050801}; both follow, as we show below, from a
polynomial moment bound on $l_0(\infty)$ for initial data in $H^2$.

\noindent\textit{Step 1: Recurrence and decoupling times}.
 For $d_0$ fixed, we introduce the following stopping times:
\begin{align*}
    \vartheta&:=\min\big\{n\ge0:\|\widetilde{u}_{nT}^\prime\|_{H^2}\vee\|\widetilde{u}_{nT}\|_{H^2}\leq d_0\big\},\\\sigma&:=\min\{n\ge 1:l_0(n)>0\},
\end{align*}
with $\min\emptyset=\infty$. We then define two families of stopping times $\{\vartheta_k\}_{k\ge0}$ and~$\{\sigma_k\}_{k\ge0}$ recursively by
\begin{align*}
    &\sigma_0:=-1,\ \vartheta_0:=\vartheta,\quad \sigma_{k+1}:=\vartheta_k+\sigma\circ\theta_{\vartheta_k T},\\  &\vartheta_{k+1}:=\sigma_{k+1}+1+\vartheta\circ\theta_{(\sigma_{k+1}+1) T}.
\end{align*}
It follows that 
\begin{align*}
    \sigma_{k+1}&=\infty\text{ if and only if $\vartheta_k=\infty$ or $l_0(\infty)=\vartheta_k$},\\ \vartheta_{k+1}&=\infty\text{ if  $\sigma_{k+1}=\infty$.}
\end{align*}
The times $\vartheta_k$ and $\sigma_{k+1}$ are, respectively, the $(k+1)$-th
recurrence time and the $(k+1)$-th decoupling time. The recurrence time is made strictly larger than the last decoupling time to ensure that $l_0(\vartheta_k)=\vartheta_k$ when $\vartheta_k<\infty$. In~particular, when $\vartheta_k<\infty$ and $\sigma_{k+1}=\infty$, we have
$l_0(\infty)=\vartheta_k$.

We shall also need the following fact. Let $k\ge0$ and let $\eta^1,\eta^2$ be
two mutually independent Wiener processes of the form~\eqref{04130915},
independent of~$\widetilde{\mathcal{F}}_{(\sigma_{k}+1)T}$. Then, on the event
$\{\sigma_{k}<\infty\}$, the conditional distribution of~$\vartheta\circ\theta_{(\sigma_{k}+1)T}$ given
$\widetilde{\mathcal{F}}_{(\sigma_{k}+1)T}$ coincides with that of
\begin{align*}
    \min\big\{j\ge0:\|u(jT;\widetilde{u}_{(\sigma_{k}+1)T},\eta^1)\|_{H^2}
    \vee\|u(jT;\widetilde{u}^\prime_{(\sigma_{k}+1)T},\eta^2)\|_{H^2}\leq d_0\big\}.
\end{align*}
 To prove this, it suffices, as in the proof of Lemma~\ref{L1:l0-markov}, to
establish the equality on $\{\sigma_{k}=n\}$ for any $n\geq -1$. On $\{\sigma_{k}=n\}$, the construction \eqref{08151052} implies that for any $M\ge0$, the event $\big\{\vartheta\circ \theta_{(\sigma_{k}+1) T}\leq M\big\}$ equals
\begin{align*}
    &\mathop{\scalebox{1.3}{$\cup$}}_{j=0}^M\big\{\|u(jT;\widetilde{u}_{(n+1)T},{\eta}^{1,n+1})\|_{H^2}\vee\|u(jT;\widetilde{u}^\prime_{(n+1)T},{\eta}^{2,n+1})\|_{H^2}\leq d_0\big\},
\end{align*}
where $\eta^{i,n+1}$, $i=1,2$, are the Wiener processes whose increments on
\mbox{$[lT,(l+1)T]$} are given by $bW^{i(l+n+1)}$, $l\ge0$. By construction,
they are mutually independent, have the same distribution as $bW$ in
\eqref{04130915}, and are independent of $\widetilde{\mathcal{F}}_{(n+1)T}$. The claim then follows by the argument used in the proof of
Lemma~\ref{local markov}. In particular, $\vartheta\circ \theta_{(\sigma_{k}+1)T}$ has the same distribution as $\vartheta_{d_0}$ defined in \eqref{08030032} (with initial conditions $\widetilde{u}_{(\sigma_{k}+1)T},\widetilde{u}^\prime_{(\sigma_{k}+1)T}$). Thus by choosing $T=T^\sharp\vee T_{d_0}$, Lemma \ref{03160229} implies the existence of $c_0>0$ such that for any $k\geq 0$,
\begin{align}\label{08151958}
    \mathbb{E}\big[\mathrm{exp}\{c_0\vartheta\}\circ\theta_{(\sigma_{k}+1)T}|\widetilde{\mathcal{F}}_{(\sigma_{k}+1)T}\big]\hspace{-0.5mm}\leq\hspace{-0.5mm} C\big(1\hspace{-0.5mm}+\hspace{-0.5mm}\widetilde{F}(\widetilde{u}_{(\sigma_{k}+1)T})\hspace{-0.5mm}+\hspace{-0.5mm}\widetilde{F}(\widetilde{u}^\prime_{(\sigma_{k}+1)T})\big),
\end{align}
with $\widetilde{F}$ defined in \eqref{03120847}.

\noindent\textit{Step 2: Polynomial convergence}.
Here we show that for any $q\geq 1$, $f\in L_b(H^1)$, and~$t> 0$,
\begin{align}\label{05050909}
    \mathbb{E}\big[|f(\widetilde{u}_t)-f(\widetilde{u}_t^\prime)|\big]\leq C_{q}\|f\|_{L}(1+\widetilde{F}(u_0)+\widetilde{F}(u^\prime_0)) t^{-q}.
\end{align}
 A crucial ingredient in proving this is to establish a moment estimate for~$l_0(\infty)$. We first show that $$\mathbb{P}\big(l_0(\infty)=\infty\big)=0.$$ To this end, we define a random time
\begin{align*}
    \kappa_0=\min\{k\ge0:\sigma_{k+1}=\infty\},\quad \min\emptyset:=\infty.
\end{align*}
We claim that $\kappa_0$ is $\mathbb{P}$-a.s. finite. Indeed, for any $k\ge1$, we have
\begin{align*}
\nonumber\mathbb{P}\big(\kappa_0\geq k\big)&=\mathbb{P}\big(\sigma_i<\infty,i=1,\ldots,k\big)=\mathbb{P}\big(\sigma_k<\infty,\vartheta_{k-1}<\infty\big)\\&=\mathbb{E}\big[\mathbb{P}\big({\sigma_k<\infty}\,|\,\widetilde{\mathcal{F}}_{\vartheta_{k-1} T}\big);{\vartheta_{k-1}<\infty}\big].
\end{align*}
Notice that \eqref{08151958} implies
\begin{align}\label{08161435}
    \mathbb{P}\big(\sigma_{k-1}<\infty,\vartheta_{k-1}=\infty\big)=0,
\end{align}
and hence
\begin{align*}
    \mathbb{P}(\kappa_0\geq k)=\mathbb{E}\big[\mathbb{P}\big({\sigma_k<\infty}\,|\,\widetilde{\mathcal{F}}_{\vartheta_{k-1} T}\big);{\sigma_{k-1}<\infty}\big].
\end{align*}
Since $l_0(\vartheta_{k-1})=\vartheta_{k-1}$ on
$\{\vartheta_{k-1}<\infty\}$, applying Lemma~\ref{L1:l0-markov} with
$\varsigma=\vartheta_{k-1}$ and Corollary~\ref{07311420}, we deduce
that
\begin{align*}
    \mathbb{P}\big({\sigma_k<\infty}\,|\,\widetilde{\mathcal{F}}_{\vartheta_{{k-1} }T}\big)=\mathbb{P}_{\widetilde{u}_{\vartheta_{k-1}T},\widetilde{u}^\prime_{{\vartheta_{k-1}}T}}\big(l_0(\infty)\neq 0\big)\leq\frac{3}{4}.
\end{align*}
As $\{\sigma_{k-1}<\infty\}=\{\kappa_0\geq k-1\}$, this gives
$$
\mathbb{P}(\kappa_0\geq k)\leq\frac{3}{4}\,
\mathbb{P}(\kappa_0\geq k-1)
$$ for every $k\geq2$, and, together with the bound
$\mathbb{P}(\kappa_0\geq 1)\leq\frac{3}{4}$, iterating yields
\begin{align}\label{03160319}
\mathbb{P}\big(\kappa_0\geq k\big)\leq\left(\frac{3}{4}\right)^{k}.
\end{align}
Consequently,
\begin{align*}
  \mathbb{P}\big(\kappa_0=\infty\big)= \lim_{k\rightarrow\infty}\mathbb{P}\big(\kappa_0\geq k\big)=0.
\end{align*}
By \eqref{08161435}, we have $\vartheta_{\kappa_0}<\infty$ and
$\sigma_{\kappa_0+1}=\infty$ $\mathbb{P}$-a.s., whence
$l_0(\infty)=\vartheta_{\kappa_0}$ $\mathbb{P}$-a.s. Thus for any $q\geq 1$, \eqref{03160319} implies that
\begin{align}\label{03160954}
    \nonumber\mathbb{E}\big[l_0^q(\infty)\big]&=\mathbb{E}\big[\vartheta_{\kappa_0}^{q}\big]=\sum_{k=0}^\infty\mathbb{E}\big[\vartheta_{k}^{q};{\kappa_0=k},{\vartheta_k<\infty}\big]\\&\leq  \sum_{k=0}^\infty\mathbb{E}\big[\vartheta_{k}^{2q};{\vartheta_k<\infty}\big]^{\frac{1}{2}}\mathbb{P}\big(\kappa_0=k\big)^{\frac{1}{2}}\nonumber\\&\leq \sum_{k=0}^\infty\left(\frac{3}{4}\right)^{\frac{k}{2}}\mathbb{E}\big[\vartheta_{k}^{2q};{\vartheta_k<\infty}\big]^{\frac{1}{2}}.
\end{align}
Moreover, by denoting $\varrho=\sigma+1+\vartheta\circ\theta_{(\sigma+1) T}$, we have
\begin{align*}
    \mathbb{E}\big[\vartheta_{k}^{2{q}};&\hspace{1mm}{\vartheta_k<\infty}\big]= \mathbb{E}\big[(\vartheta+\sum_{l=0}^{k-1} \varrho\circ\theta_{\vartheta_l T})^{2q};{\vartheta_k<\infty}\big]\\&\leq\ C_{q}(k+1)^{2q-1}\Big(\mathbb{E}\big[\vartheta^{2q}\big]+\sum_{l=0}^{k-1} \mathbb{E}\big[\big(\varrho\circ\theta_{\vartheta_l T})^{2q};{\vartheta_k<\infty}\big]\Big)\\&\leq \ C_{q}(k\hspace{-0.5mm}+\hspace{-0.5mm}1)^{2q-1}\big(\mathbb{E}\big[\vartheta^{2q}\big]+1\big)\hspace{-0.5mm}+\hspace{-0.5mm}C_{q}(k\hspace{-0.5mm}+\hspace{-0.5mm}1)^{2q-1}\sum_{l=0}^{k-1} \mathbb{E}\big[\big(\sigma\circ\theta_{\vartheta_lT})^{2q};\\&\ \ \quad{\sigma\circ\theta_{\vartheta_lT}<\infty}\big]\hspace{-0.5mm}+\hspace{-0.5mm}C_{{q}}(k\hspace{-0.5mm}+\hspace{-0.5mm}1)^{2q-1}\sum_{l=1}^{k}\mathbb{E}\big[\big(\vartheta\circ\theta_{(\sigma_{l}+1)T}\big)^{2q};{\sigma_{l}<\infty}\big]\\&=:\  \mathrm{I+II+III}.
\end{align*}
For $\mathrm{I}$, since $T\geq T_{d_0}$, \eqref{08151958} implies that 
\begin{align}\label{03220306}
    \mathrm{I}\leq C_{{q}}(k+1)^{2q-1}\big(1+\widetilde{F}(u_0)+\widetilde{F}(u_0^\prime)\big).
\end{align}
To estimate $\mathrm{II}$, notice that there exists a non-negative measurable function $g_\sigma:C([0,\infty);H^2)^2\rightarrow\mathbb{R}$ such that $\sigma \mathbf{1}_{\sigma<\infty}=g_\sigma(\widetilde{u},\widetilde{u}^\prime)$. Thus, taking $\varsigma=\vartheta_l$ and $g=g^{2q}_\sigma$ in
Lemma~\ref{L1:l0-markov}, then applying Lemma~\ref{03151059} and
Lemma~\ref{03160423} with $q$ replaced by~$10q$, we obtain $c_{1}>0$
and $C_{q}>0$ such that, for any $0\leq l\leq k-1$,
\begin{align}\label{03160945}
     \nonumber\mathbb{E}\big[(\sigma&\circ\theta_{\vartheta_l T})^{2q};{\sigma\circ\theta_{\vartheta_lT}<\infty}\big]=\mathbb{E}\big[\mathbb{E}_{\widetilde{{u}}_{\vartheta_l T},\widetilde{u}^\prime_{\vartheta_l T}}\big[\sigma^{2q};{\sigma<\infty}\big];{\vartheta_l<\infty}\big]\\\hspace{-1.5mm}\leq&\nonumber\hspace{1mm}\mathbb{E}\big[\mathbb{P}_{\widetilde{{u}}_{\vartheta_l T},\widetilde{u}^\prime_{\vartheta_l T}}(\sigma=1)+\sum_{m=2}^\infty m^{2q}\mathbb{P}_{\widetilde{{u}}_{\vartheta_l T},\widetilde{u}^\prime_{\vartheta_l T}}(\sigma=m);{\vartheta_l<\infty}\big]\\\leq&\hspace{0.5mm}\frac{1}{2}+C_{{q}}\sum_{m=2}^\infty m^{2q}\big((1+m)^{-5q+1}+e^{-c_{1} m}\big)\leq C_{{q}},
\end{align}
where the first inequality follows since $\mathbb{P}_{w,w^\prime}(\sigma=1)=\mathbb{P}_{w,w^\prime}(l_0(1)\neq 0)$ for $w,w^\prime\in B_{H^2}(d_0)$. Thus we have
\begin{align}\label{03220306-1}
    \mathrm{II}\leq C_q(k+1)^{2q}.
\end{align}
Finally, for $\mathrm{III}$, \eqref{08151958} implies that
   \begin{align*}  \mathbb{E}\big[(\vartheta\circ\theta_{(\sigma_l+1) T})^{2q};{\sigma_l<\infty}\big]=&\hspace{1mm}\mathbb{E}\big[\mathbb{E}\big[(\vartheta\circ\theta_{(\sigma_l+1) T})^{2q}\big|\widetilde{\mathcal{F}}_{(\sigma_l+1) T}\big];{\sigma_l<\infty}\big]\\\leq&\ C_{q}\hspace{0.5mm}\mathbb{E}\big[1+\widetilde{F}(\widetilde{u}_{(\sigma_l+1) T})+\widetilde{F}(\widetilde{u}^\prime_{(\sigma_l+1) T});{\sigma_l<\infty}\big].
   \end{align*}
   Recall that $\sigma_l=\vartheta_{l-1}+\sigma\circ\theta_{\vartheta_{l-1}T}$, which implies that $$\mathbf{1}_{\sigma_l<\infty}=\mathbf{1}_{\sigma\circ\theta_{\vartheta_{l-1}T}<\infty}\mathbf{1}_{\vartheta_{l-1}<\infty}.$$ Since $\mathbf{1}_{\sigma\circ \theta_{\vartheta_{l-1}T}<\infty}=\mathbf{1}_{g_\sigma(\widetilde{u},\widetilde{u}^\prime)>0}\circ\theta_{\vartheta_{l-1}T}$, by taking the measurable function $$g(u^1,u^2)=[\widetilde{F}(u^1_{(g_\sigma(u^1,u^2) +1)T})+\widetilde{F}(u^2_{(g_{\sigma}(u^1,u^2)+1) T})]\mathbf{1}_{g_\sigma(u^1,u^2)>0}$$ and $\varsigma=\vartheta_{l-1}$ in Lemma \ref{L1:l0-markov}, we have
   \begin{align*}
&\mathbb{E}\big[(\vartheta\circ\theta_{(\sigma_l+1) T})^{2q};{\sigma_l<\infty}\big]\\\leq&\ C_{{q}}\hspace{-0.5mm}+\hspace{-0.5mm}C_{{q}}\mathbb{E}\big[\mathbb{E}\big[\big[\big(\widetilde{F}(\widetilde{u}_{(\sigma+1) T})\hspace{-0.5mm}+\hspace{-0.5mm}\widetilde{F}(\widetilde{u}^\prime_{(\sigma+1) T})\big)\mathbf{1}_{\sigma<\infty}\big]\hspace{-0.5mm}\circ\hspace{-0.5mm}\theta_{\vartheta_{l-1}T}|\widetilde{\mathcal{F}}_{\vartheta_{l-1}T}\big];{\vartheta_{l-1}\hspace{-0.5mm}<\hspace{-0.5mm}\infty}\big]\\=&\ C_{{q}}+C_{{q}}\mathbb{E}\big[\mathbb{E}_{\widetilde{u}_{\vartheta_{l-1} T},\widetilde{u}^\prime_{\vartheta_{l-1} T}}\big[{{\widetilde{F}(\widetilde{u}_{(\sigma+1) T})+\widetilde{F}(\widetilde{u}^\prime_{(\sigma+1) T})}};{\sigma<\infty}\big];{\vartheta_{l-1}<\infty}\big].
\end{align*}
 Since \eqref{03160945} and H\"older's inequality imply that for any $1\leq l\leq k$,
\begin{align*}
    &\ \mathbb{E}_{\widetilde{u}_{\vartheta_{l-1} T},\widetilde{u}^\prime_{\vartheta_{l-1} T}}\big[(\sigma+1)^2;\sigma<\infty\big]\leq C_q,
\end{align*}
we apply Lemmas \ref{02280441}(iii), \ref{02280436}(iii), and \ref{02231236}(iii) at stopping time $(\sigma_l+1)T$ to deduce that
\begin{align*}
\mathbb{E}\big[(\vartheta\circ\theta_{(\sigma_l+1) T})^{2q};{\sigma_l\!<\!\infty}\big]\leq &\hspace{1mm}C_{q}\mathbb{E}\big[\mathbb{E}_{\widetilde{u}_{\vartheta_{l-1} T},\widetilde{u}^\prime_{\vartheta_{l-1} T}}\big[(1+\sigma T)^2;{\sigma\!<\!\infty}\big]^{\frac{1}{2}};\\&\hspace{1mm}\qquad\hspace{1mm}{\vartheta_{l-1}\!<\!\infty}\big]+C_{q}\leq C_{q}.
\end{align*}
As a result,
\begin{align}\label{03220306-2}
    \mathrm{III}\leq C_q(k+1)^{2q}.
\end{align}
Combining \eqref{03220306}, \eqref{03220306-1}, and \eqref{03220306-2}, we obtain
\begin{align}\label{03160956}
    \mathbb{E}\big[\vartheta_k^{2q};{\vartheta_k<\infty}\big]\leq C_{q}(k+1)^{2q}\big(1+\widetilde{F}(u_0)+\widetilde{F}(u^\prime_0)\big).
\end{align}
Together, \eqref{03160954} and \eqref{03160956} give \begin{align*}
    \mathbb{E}\big[l_0^{q}(\infty)\big]&\leq C_{q}\sum_{k=0}^\infty \big(\frac{3}{4}\big)^{\frac{k}{2}}(k+1)^{q}\big(1+\widetilde{F}(u_0)+\widetilde{F}(u_0^\prime)\big)^{\frac{1}{2}}\\&\leq C_{q}\big(1+\widetilde{F}(u_0)+\widetilde{F}(u_0^\prime)\big).
\end{align*}
Thus for any $f\in L_b(H^1)$, $t\geq 0$, and $l\ge 1$, Chebyshev's inequality implies that
\begin{align}\label{03161059}
    \nonumber\mathbb{E}\big[|f(\widetilde{u}_t)-f(\widetilde{u}^\prime_t)|;l_0(\infty)\geq l\big]&\leq 2\|f\|_{L^\infty}\mathbb{P}\big(l_0(\infty)\geq l\big)\\&\leq   C_{q}\|f\|_{L}(1+\widetilde{F}(u_0)+\widetilde{F}(u_0^\prime))l^{-q}.
\end{align}
On the other hand, for any $t\geq lT$, 
\begin{align*}
     \mathbb{E}\big[|f(\widetilde{u}_t)-f(\widetilde{u}^\prime_t)|;l_0(\infty)<l\big]&\leq\|f\|_{L}\mathbb{E}\big[\|\widetilde{u}_t-\widetilde{u}^\prime_t\|_{H^1};l_0(\infty)<l\big].
\end{align*}
On the event $\{l_0(\infty)=k\}$, we have $\widetilde{u}^\prime_t=\widetilde{v}_t$
for all $t\geq kT$, and the process~$\widetilde{v}_{kT+\cdot}$ solves
\eqref{auxiliary} with initial condition $\widetilde{u}^\prime_{kT}$, driven by
the noise~$\theta_{kT}\widetilde{\eta}^1$. Combining this with \eqref{06190211}
and \eqref{08130433}, we obtain
\begin{align*}
    \mathbb{E}\big[\|\widetilde{u}_t-\widetilde{u}^\prime_t\|_{H^1}&;
    l_0(\infty)<l\big]
    \leq\sum_{k=0}^{l-1}\mathbb{E}\big[\|\widetilde{u}_t
    -\widetilde{v}_t\|^2_{H^1};l_0(\infty)=k\big]^{\frac{1}{2}}\\
    &\leq\sqrt{2}\sum_{k=0}^{l-1}\mathbb{E}\big[\widetilde{\mathcal{M}}
    \big(t-kT,\,u^{1,k}-v^{1,k}\big);
    \widetilde{u}_{kT},\widetilde{u}_{kT}^\prime\in B_{H^2}(d_0),\\
    &\hspace{2.1cm}\
    \tau_1^{u^{1,k},{v}^{1,k}}>t-kT
    \big]^{\frac{1}{2}},
\end{align*}
where $u^{1,k}:=u(\cdot;\widetilde{u}_{kT},\theta_{kT}\widetilde{\eta}^1)$ and $v^{1,k}:=v(\cdot;\widetilde{u}_{kT},\widetilde{u}^\prime_{kT},\theta_{kT}\widetilde{\eta}^1)$. Notice that~$\tau_1^{{u}^{1,k},{v}^{1,k}}$
 is measurably determined by $\theta_{kT}\widetilde{\eta}^1$, $\widetilde{u}_{kT}$, and $\widetilde{u}^\prime_{kT}$. Since~$\theta_{kT}\widetilde{\eta}^1$ is independent of $\widetilde{\mathcal{F}}_{kT}$ and has the same distribution as $bW$ (see Lemma \ref{08080341}), we deduce from \eqref{07312337} that, for any $t\geq lT$, 
\begin{align}
    \mathbb{E}\big[|f(\widetilde{u}_t)&-f(\widetilde{u}^\prime_t)|;
    l_0(\infty)<l\big]\nonumber\\
    \leq&\ \sqrt{2}\|f\|_{L}\sum_{k=0}^{l-1}\mathbb{E}\big[\mathbb{E}\big[
    \widetilde{\mathcal{M}}\big(t-kT,u(t-kT;w,bW)\nonumber\\
    &\hspace{1.2cm}
    -v(t-kT;w,w^\prime,bW)\big);
    t-kT<\tau^{u,v}_1\big]\big|_{(w,w^\prime)
    =(\widetilde{u}_{kT},\widetilde{u}^\prime_{kT})};\nonumber\\
    &\hspace{1.3cm}
    \widetilde{u}_{kT},\widetilde{u}^\prime_{kT}\in B_{H^2}(d_0)
    \big]^{\frac{1}{2}}\nonumber\\
    \leq&\ C\|f\|_{L}\,l\,e^{-\frac{3\alpha}{8}(t-lT)}.\label{03220438}
\end{align}
Taking $l=\lfloor \frac{t}{2 T}\rfloor\vee 1$ and adding \eqref{03161059} and \eqref{03220438}, we arrive at \eqref{05050909} for~$t\geq T$; for $t\in(0,T)$, the estimate \eqref{05050909} follows from the trivial bound $$
|f(\widetilde{u}_t)-f(\widetilde{u}_t^\prime)|\leq 2\|f\|_{L}\leq 2T^{q}\|f\|_{L}\,t^{-q}.
$$

\noindent\textit{Step 3: Proof of uniqueness and \eqref{05050801}}.
Let $\nu\in\mathcal{P}(H^2)$ be the stationary measure constructed in
Lemma~\ref{05050844}. Since $(\widetilde{u},\widetilde{u}^\prime)$ is a coupling of $(u,u^\prime)$,
\eqref{05050909} gives
$$
\big|\PPPP_tf(u_0)-\PPPP_tf(u_0^\prime)\big|\leq C_q\|f\|_{L}
\big(1+\widetilde{F}(u_0)+\widetilde{F}(u_0^\prime)\big)t^{-q}
$$
for any $u_0,u_0^\prime\in H^2$ and $f\in L_b(H^1)$. Integrating this with
respect to $\nu(\dd u_0^\prime)$ and using the stationarity of~$\nu$, we obtain
\begin{align*}
    \big|\PPPP_tf(u_0)-\int_{H^2}f(v)\nu(\dd v)\big|&=\big|\PPPP_tf(u_0)-\int_{H^2} f(v)\PPPP_t^*\nu(\dd v)\big|\\&\leq \nonumber C_q\|f\|_{L}\big(1+\widetilde{F}(u_0)+\int_{H^2} \widetilde{F}(v)\nu(\dd v)\big)t^{-q}\\&\leq  C_q\|f\|_L\big(1+\|u_0\|_{H^2}^{20}\big)t^{-q},
\end{align*}  where on the last line we use $\widetilde{F}(u)\leq C\big(1+\|u\|_{H^2}^{20}\big)$ and Lemma~\ref{05050848}.
Integrating this inequality with respect to~$\lambda(\dd u_0)$, for arbitrary
$\lambda\in\mathcal{P}(H^2)$, we get 
\eqref{05050801}.

It remains to prove the uniqueness. Let $\widetilde{\nu}\in\mathcal{P}(H^2)$ be
another stationary measure. Applying \eqref{05050801} with
$\lambda=\widetilde{\nu}$ and using $\PPPP_t^*\widetilde{\nu}=\widetilde{\nu}$,
we obtain
\begin{align*}
    \|\widetilde{\nu}-\nu\|^*_{L}
    \leq C_q\Big(1+\int_{H^2}\|u\|_{H^2}^{20}\,\widetilde{\nu}(\dd u)\Big)t^{-q}.
\end{align*}
 Letting
$t\rightarrow\infty$ gives $\widetilde{\nu}=\nu$.
\end{proof}

\appendix

\setcounter{section}{0}
\refstepcounter{section}   

\section*{Appendix}
\phantomsection
\addcontentsline{toc}{section}{Appendix}

\setcounter{subsection}{0}

\subsection{A priori estimates}

We begin with the following a priori estimates for solutions to \eqref{equation-1}.
\begin{lemma}\label{02280441}
    For any $p\geq 1$, $t\geq 0$, $\gamma>0$, and any stopping time $\tau$ such that $\mathbb{E}[\tau^2;\tau<\infty]<\infty$, the solution $u$ to \eqref{equation-1} satisfies
    \begin{align*}
    &\mathrm{(i)} \quad \mathbb{E}\big[\|u_t\|^{2p}+\frac{\alpha p}{2}\int_0^te^{-\gamma(t-s)}\|u_s\|^{2p}\dd s\big]\leq e^{-(\gamma\wedge \frac{\alpha p}{2})t}\|u_0\|^{2p}+C_{p,\gamma},\\
      &\mathrm{(ii)} \quad \mathbb{E}\big[\sup_{s\in[0,t]}\big\{\|u_s\|^{2p}+{\alpha p}\int_0^s\|u_r\|^{2p}\dd r\big\}\big]\leq 2\|u_0\|^{2p}+C_pt,\\
       &\mathrm{(iii)} \quad \mathbb{E}\big[\|u_\tau\|^{2p};\tau<\infty\big]\leq C_{p}(1+\|u_0\|^{2p})\mathbb{E}\big[(1+\tau)^2;\tau<\infty\big]^{\frac{1}{2}}.
    \end{align*}
\end{lemma}
Estimates (i) and (ii) are standard. Estimate (iii) follows by the argument used to prove Lemma~\ref{02231236}(iii) below, which is carried out there in a more involved setting; we therefore omit the details here to avoid repetition.

For any $p\ge 1$, we define
 \begin{align*}
     \mathcal{E}^u_{0,p}(t):=\|u_t\|^{2p}+\alpha p\int_0^t\|u_s\|^{2p}\dd s.
 \end{align*} 
\begin{lemma}\label{03050812-1}
 For any $p\ge1$, there exists $K_{0,p}>0$ such that, for any $q\ge1$ and $\rho\geq 1$, we have
\begin{align*}
 \mathbb{P}\Big\{\sup\limits_{t\geq 0}\big\{ \mathcal{E}_{0,p}^u(t)-K_{0,p}t-\|u_0\|^{2p}\big\}\geq\rho\Big\}\leq  C_{p,q}(1+\|u_0\|^{(2p-1)q})\rho^{-\frac{q}{2}+1}.
\end{align*}
\end{lemma}
\begin{proof}
 In the cases when $q\in [1,2]$, since $(1+\|u_0\|^{(2p-1)q})\rho^{-\frac{q}{2}+1}\geq 1$, we obtain the required inequality by setting $C_{p,q}=1$ for any $p\geq 1$ and $q\in[1,2]$. Thus it suffices to consider the case $q>2$.\par
\noindent\textit{Case 1: $p\geq 2$}. 
Applying It\^o's formula to $\|u_t\|^{2p}$ and using Young's inequality, we obtain that
\begin{align}
\|u_t\|^{2p}
&+2\alpha p\int_0^t\|u_s\|^{2p}\dd s \notag\\
&= \|u_0\|^{2p}
+2p\int_0^t\|u_s\|^{2p-2}\langle u_s,b\dd W_s\rangle \notag+2p\sum_{j=1}^\infty b_j^2\|e_j\|^2\int_0^t\|u_s\|^{2p-2}\dd s \notag\\
&\quad+2p(p-1)\sum_{j=1}^\infty b_j^2
\int_0^t\|u_s\|^{2p-4}(\langle u_s, e_j\rangle^2+\langle u_s, ie_j\rangle^2)\dd s \notag\\
&\leq   \|u_0\|^{2p}
+2p\int_0^t\|u_s\|^{2p-2}\langle u_s,b\dd W_s\rangle
+\alpha p\int_0^t\|u_s\|^{2p}\dd s
+C_pt.\notag
\end{align}
By taking $K_{0,p}:=C_p+1$, we have
\begin{align}\label{03040403}
    \mathcal{E}_{0,p}^u(t)-K_{0,p}t-\|u_0\|^{2p}\leq 2p\int_0^t\|u_s\|^{2p-2}\langle u_s, b\dd W_s\rangle-t.
\end{align}
Thus, by Lemma~\ref{02280441}(ii),
\begin{align}\label{03070952}
    \nonumber\mathbb{P}\Big\{\sup\limits_{t\geq 0}\big\{\mathcal{E}_{0,p}^u(t)&-K_{0,p}t-\|u_0\|^{2p}\big\}\geq\rho\Big\}\\\nonumber\leq &\sum_{m=0}^\infty\mathbb{P}\Big\{2p\sup_{t\in[m,m+1]}\big|\int_0^t\|u_s\|^{2p-2}\langle u_s, b\dd W_s\rangle\big|\geq\rho+ m\Big\}\\\leq \nonumber&\sum_{m=0}^\infty\frac{C_{p,q}}{(\rho+ m)^q}\mathbb{E}\big[\sup_{t\in[0,m+1]}\big|\int_0^{t}\|u_s\|^{2p-2}\langle u_s, b\dd W_s\rangle\big|^q\big]\\\leq \nonumber&\sum_{m=0}^\infty\frac{C_{p,q}}{(\rho+ m)^q}\mathbb{E}\big[\big|\int_0^{m+1}\|u_s\|^{4p-2}\dd s \big|^\frac{q}{2}\big]\\\leq \nonumber&\sum_{m=0}^\infty\frac{C_{p,q}}{(\rho+ m)^q}(m+1)^{\frac{q}{2}-1}\mathbb{E}\big[\int_0^{m+1}\|u_s\|^{2pq-q}\dd s \big]\\\leq&\sum_{m=0}^\infty\frac{C_{p,q}}{(\rho+ m)^q}\big[(m+1)^{\frac{q}{2}}+\|u_0\|^{(2p-1)q}(m+1)^{\frac{q}{2}-1}\big].
\end{align}
To estimate the last series, note that since $\rho\geq 1$ and $q>2$, we have
\begin{align}\label{08010427}
    \nonumber\sum_{m=0}^\infty&\frac{C_{p,q}}{(\rho+ m)^q}\big[(m+1)^{\frac{q}{2}}+\|u_0\|^{(2p-1)q}(m+1)^{\frac{q}{2}-1}\big]\\&\leq \nonumber  \big(\sum_{m=0}^{\lfloor \rho\rfloor}+\sum_{m=\lceil\rho\rceil}^{\infty}\big)\frac{C_{p,q}}{(\rho+ m)^q}\big[(m+1)^{\frac{q}{2}}+\|u_0\|^{(2p-1)q}(m+1)^{\frac{q}{2}-1}\big]\\&\leq \nonumber C_{p,q}{\rho^{-q}}\big\{\rho(\rho+1)^{\frac{q}{2}}+\|u_0\|^{(2p-1)q}\rho({\rho}+1)^{\frac{q}{2}-1}\big\}\\&\quad\nonumber+C_{p,q}\sum_{m=\lceil{\rho}\rceil}^{\infty}\frac{1}{m^q}\big[(m+1)^{\frac{q}{2}}+\|u_0\|^{(2p-1)q}(m+1)^{\frac{q}{2}-1}\big]\\&\leq \nonumber   C_{p,q} {\rho^{-\frac{q}{2}+1}}(1+\|u_0\|^{(2p-1)q})+C_{p,q}\sum_{m=\lceil{\rho}\rceil}^{\infty}\frac{1}{m^{\frac{q}{2}}}(1+\|u_0\|^{(2p-1)q})\\&\leq C_{p,q}(1+\|u_0\|^{(2p-1)q})\rho^{-\frac{q}{2}+1}.
\end{align}
\noindent\textit{Case 2: $p\in[1,2)$}. It suffices to establish
\eqref{03040403} for $p\in[1,2)$; the estimates \eqref{03070952} and
\eqref{08010427} then apply without change. Since $x\mapsto x^{p}$ is not
twice differentiable at the origin for $p<2$, we apply It\^o's formula to
$(\varepsilon+\|u_t\|^2)^{p}$ and pass to the limit as $\varepsilon\to0$.
This argument is standard, and we refer the reader to Lemma~2.2
of~\cite{BDHPS03} for the details.
\end{proof}
We apply the previous lemma to derive bounds for the stopping time
\begin{align}\label{03190353-1}
\tau_{0,p}^u:=\inf\{t\geq0:\mathcal{E}_{0,p}^u(t)\geq \rho+Kt+\mathscr{C}\|u_0\|^{2p}\},  
\end{align}with the usual convention $\inf\!\emptyset=\infty$.

\begin{corollary}\label{03090533} For any $p,q\geq 1$, $\rho\geq 1$, $K\geq K_{0,p}+1$,  $\mathscr{C}\geq 1$, and $l\geq 0$, we have
\begin{align*}
    \mathbb{P}(l\leq \tau_{0,p}^u<\infty)\leq  C_{p,q}\big(\rho+l\big)^{-\frac{q}{2}+1}(1+\|u_0\|^{(2p-1)q}).
\end{align*}
\end{corollary}
\begin{proof}
Notice that on the event $\{\tau_{0,p}^u<\infty\}$, we have
\begin{align*}
    \mathcal{E}_{0,p}^u(\tau_{0,p}^u)=\rho+K\tau_{0,p}^u+\mathscr{C} \|u_0\|^{2p}.
\end{align*}
Since $K\geq K_{0,p}+ 1$, on the set $\{l\leq\tau_{0,p}^u<\infty\}$, we have
\begin{align*}
   \rho+l\leq&\hspace{1mm}\rho+(K-K_{0,p})\tau_{0,p}^u+(\mathscr{C}-1)\|u_0\|^{2p}\\ =&\hspace{1mm}\mathcal{E}_{0,p}^u(\tau_{0,p}^u)-K_{0,p}\tau_{0,p}^u-\|u_0\|^{2p}\\\leq& \hspace{1mm}\sup\limits_{t\geq 0}\big\{\mathcal{E}_{0,p}^u(t)-K_{0,p}t-\|u_0\|^{2p}\big\}.
\end{align*}
Noting that $\rho\geq 1$, we deduce from Lemma \ref{03050812-1} that,
\begin{align*}
    \mathbb{P}(l\leq \tau_{0,p}^u<\infty)\leq&\ \mathbb{P}\Big\{ \sup\limits_{t\geq 0}\big\{\mathcal{E}_{0,p}^u(t)-K_{0,p}t-\|u_0\|^{2p}\big\}\geq \rho+l\Big\}
    \\\leq&\ C_{p,q}\big(\rho+l\big)^{-\frac{q}{2}+1}(1+\|u_0\|^{(2p-1)q}).
\end{align*}
\end{proof}

Recall the functional $\mathcal{H}$ defined in \eqref{03090247-1}. The
following estimates are the analogues of Lemma~\ref{02280441} for
$\mathcal{H}$, and are proved in the same way.
\begin{lemma}\label{02280436}
For any $p\geq 1$, $t\geq0$, $\gamma>0$, and any stopping time $\tau$ such that $\mathbb{E}\big[\tau^2;\tau<\infty\big]<\infty$, we have 
\begin{align*}
        \noindent &{\mathrm{(i)}}\quad
        \mathbb{E}\big[\mathcal{H}^p(u_t)+\frac{p\alpha}{2}\int_0^te^{-\gamma(t-s)}\mathcal{H}^p(u_s)\dd s\big]\leq e^{-(\gamma\wedge\frac{\alpha p}{2})t}\mathcal{H}^p(u_0)+C_p,\\
    &{\mathrm{(ii)}}\quad \mathbb{E}\big[\sup_{s\in[0,t]}\big\{\mathcal{H}^p(u_s)+{\alpha p}\int_0^s\mathcal{H}^p(u_r)\dd r\big\}\big]\leq 2\mathcal{H}^p(u_0)+C_pt,\\&
           {\mathrm{(iii)}}\quad
           \mathbb{E}\big[\mathcal{H}^p(u_\tau);\tau<\infty\big]\leq C_p\big(1+\mathcal{H}^p(u_0)\big)\mathbb{E}\big[(1+\tau)^2;\tau<\infty\big]^{\frac{1}{2}}.
\end{align*}
\end{lemma}
Next, we define 
$$\mathcal{E}_{1,p}^u(t):=\mathcal{H}^{p}(u_t)+{\alpha p}\int_0^t\mathcal{H}^{p}(u_s)\dd s.
$$ 
\begin{lemma}\label{03050812} For any $p\geq 1$ there exists $K_{1,p}>0$ such that, for any $q\geq 1$ and $\rho\geq 1$, we have
\begin{align*}
\mathbb{P}\Big\{\sup_{t\geq0}\{\mathcal{E}_{1,p}^u(t) & -K_{1,p}t-\mathcal{H}^{p}(u_0)\}\geq\rho\Big\}\\\leq&\  C_{p,q}\big(1+\mathcal{H}^{q(p-\frac{1}{4})}(u_0)\big)\rho^{-\frac{q}{2}+1}.
\end{align*}
\end{lemma}
\begin{proof}For the same reason as mentioned at the beginning of the proof of Lemma \ref{03050812-1}, it suffices to consider the case $q>2$. The case $p\in[1,2)$ is similar to Case 2 of Lemma \ref{03050812-1}. Thus we only consider the case $p\geq 2$. Applying It\^o's formula to $\mathcal{H}^p(u_t)$, we obtain 
    \begin{align*}
        \mathcal{H}^{p}(u_t)&=\mathcal{H}^{p}(u_0)-p\alpha\int_0^t\mathcal{H}^{p-1}(u_s)\big({\|\partial_xu_s\|^2}+{\|u_s\|_{L^4}^4}\big)\dd s+  p\sum_{j=1}^\infty b_j^2\|\partial_x e_j\|^2\\&\quad\times\int_0^t\mathcal{H}^{p-1}(u_s)\dd s+2p\sum_{j=1}^\infty b_j^2\int_0^t\mathcal{H}^{p-1}(u_s)\hspace{-1mm}\int_\mathbb{R}|u_s|^{2}e_j^2\dd x\dd s\\&\quad+  \frac{p(p-1)}{2}\sum_{j=1}^\infty b_j^2\int_0^t\mathcal{H}^{p-2}(u_s)(\langle \partial_x u_s,\partial_x e_j\rangle^2+\langle \partial_x u_s,i\partial_x e_j\rangle^2)\dd s\\&\quad+  \frac{p(p-1)}{2}\sum_{j=1}^\infty b_j^2\int_0^t\mathcal{H}^{p-2}(u_s)(\langle |u_s|^2u_s,e_j\rangle^2+\langle |u_s|^2u_s,ie_j\rangle^2) \dd s\\&\quad +p(p-1)\sum_{j=1}^\infty b_j^2\int_0^t\mathcal{H}^{p-2}(u_s)\big(\langle \partial_xu_s,\partial_x e_j\rangle\langle |u_s|^2u_s,e_j\rangle\\&\quad +\langle \partial_xu_s,i\partial_x e_j\rangle\langle |u_s|^2u_s,ie_j\rangle\big)\dd s+  M^{1,p}_t,
    \end{align*}
    where
\begin{align}\label{06190550}
    \nonumber M^{1,p}_t:=&\hspace{0.5mm}p\sum_{j=1}^\infty b_j\mathrm{Re}\int_0^t\hspace{-0.5mm}\mathcal{H}^{p-1}(u_s)\hspace{-1mm}\int_{\mathbb{R}}\partial_x \overline{u}_s(x)\partial_x e_j(x) \dd x (\dd \beta^1_j(s)+i\dd \beta^2_j(s))\\&+p\sum_{j=1}^\infty b_j\mathrm{Re}\int_0^t\hspace{-0.5mm}\mathcal{H}^{p-1}(u_s)\hspace{-1mm}\int_{\mathbb{R}}|u_s|^2\overline{u}_s e_j \dd x(\dd \beta^1_j(s)+i\dd \beta^2_j(s)).
\end{align}
We then have
\begin{align*}
   \nonumber\mathcal{H}^p(u_t)&-\mathcal{H}^p(u_0)+2p\alpha\int_0^t\mathcal{H}^{p}(u_s)\dd s\\\leq& \nonumber\ M^{1,p}_t+\frac{p\alpha}{2}\int_0^t\mathcal{H}^p(u_s)\dd s+C_{p}(\sum_{j=1}^\infty b_j^2\|\partial_xe_j\|^{2})^pt\\&\nonumber+2p\sum_{j=1}^\infty b_j^2\int_0^t\mathcal{H}^{p-1}(u_s)\|u_s\|_{L^4}^2\|e_j\|_{L^4}^2\dd s\\&\nonumber+2p(p-1)\sum_{j=1}^\infty b_j^2\int_0^t\mathcal{H}^{p-2}(u_s)(\|\partial_xu_s\|^2\|\partial_xe_j\|^2+\|u_s\|_{L^4}^6\|e_j\|_{L^4}^2)\dd s \\\leq&\nonumber\ M^{1,p}_t+\frac{p\alpha}{2}\int_0^t\mathcal{H}^p(u_s)\dd s+C_{p}(\sum_{j=1}^\infty b_j^2\|\partial_xe_j\|^{2})^pt\\&\nonumber+{pC}\sum_{j=1}^\infty b_j^2\|e_j\|_{L^4}^2\int_0^t\mathcal{H}^{p-\frac{1}{2}}(u_s)\dd s\\&\nonumber+{p(p-1)}C\sum_{j=1}^\infty b_j^2\int_0^t\mathcal{H}^{p-1}(u_s)\|\partial_xe_j\|^2\dd s\\&\nonumber+{p(p-1)C}\sum_{j=1}^\infty b_j^2\|e_j\|_{L^4}^2\int_0^t\mathcal{H}^{p-\frac{1}{2}}(u_s)\dd s\\\leq&\ M^{1,p}_t\hspace{-0.5mm}+\hspace{-0.5mm}{p\alpha}\int_0^t\mathcal{H}^p(u_s)\dd s+C_pt.
\end{align*}
Setting ${K}_{1,p}:=C_p+1$, this inequality becomes
\begin{align*}
    &\mathcal{H}^p(u_t)-\mathcal{H}^p(u_0)+ {p\alpha}\int_0^t\mathcal{H}^p(u_s)\dd s-K_{1,p}t\leq\hspace{0.5mm} M^{1,p}_t-t.
\end{align*}
Thus Chebyshev's inequality and Burkholder--Davis--Gundy inequality imply that
\begin{align*}
   \mathbb{P}\Big(\sup_{t\geq0}&\big\{\mathcal{H}^p(u_t)-\mathcal{H}^p(u_0)+{p\alpha}\int_0^t\mathcal{H}^p(u_s)\dd s-K_{1,p} t\big\}\geq \rho  \Big)\\\leq&\ \mathbb{P}\Big(\sup_{t\geq 0}\big\{M_t^{1,p}- t\big\}\geq\rho\Big)\\\leq& \sum_{m=0}^\infty\mathbb{P}\Big(\sup_{m\leq t\leq m+1}\big\{M^{1,p}_t- t\big\}\geq\rho\Big)\\\leq&\sum_{m=0}^\infty\mathbb{P}\Big(\sup_{t\leq m+1}M_t^{1,p}\geq \rho+ m\Big)\leq C_q\sum_{m=0}^\infty\frac{\mathbb{E}\big[\langle M^{1,p}\rangle_{m+1}^{\frac{q}{2}}\big]}{(\rho+ m)^{q}}.
\end{align*}
From the definition \eqref{06190550} of $M^{1,p}$, we obtain
\begin{align}\label{03130248}
    \nonumber &\hspace{1mm}\mathbb{E}\big[\langle M^{1,p}\rangle_{m+1}^{\frac{q}{2}}\big]\\\leq&\ \nonumber C_{p,q}\mathbb{E}\big[\big(\sum_{j=1}^\infty b_j^2\int_0^{m+1}\mathcal{H}^{2(p-1)}(u_s)(\|\partial_x{u}_s\|^2\|\partial_xe_j\|^2+\|u_s\|_{L^4}^6\|e_j\|_{L^4}^2)\dd s\big)^{\frac{q}{2}}\big]\\\leq&\nonumber \ C_{p,q}\big(\sum_{j=1}^\infty b_j^2\|e_j\|_{H^1}^2\big)^{\frac{q}{2}}\mathbb{E}\big[\big(\int_0^{m+1}\mathcal{H}^{2(p-1)}(u_s)(\|\partial_x{u}_s\|^2+\|u_s\|_{L^4}^6)\dd s\big)^{\frac{q}{2}}\big]\\\leq&\nonumber \ C_{p,q}\mathbb{E}\big[\big(\int_0^{m+1}\mathcal{H}^{2(p-1)}(u_s)(\|\partial_x{u}_s\|^2+\|u_s\|_{L^4}^6)\dd s\big)^{\frac{q}{2}}\big]\\\leq&\nonumber \ C_{p,q}(m+1)^{\frac{q}{2}-1}\mathbb{E}\big[\int_0^{m+1}\mathcal{H}^{q(p-1)}(u_s)(\|\partial_x{u}_s\|^q+\|u_s\|_{L^4}^{3q})\dd s\big]\\\leq&\nonumber \ C_{p,q}(m+1)^{\frac{q}{2}-1}\mathbb{E}\big[\int_0^{m+1}\mathcal{H}^{q(p-\frac{1}{2})}(u_s)(1+\mathcal{H}^{\frac{q}{4}}(u_s))\dd s\big]\\\leq&\nonumber \ C_{p,q}(m+1)^{\frac{q}{2}-1}\Big(m+1+\mathbb{E}\big[\int_0^{m+1}\mathcal{H}^{q(p-\frac{1}{4})}(u_s)\dd s\big]\Big)\\\leq& \ C_{p,q}(m+1)^{\frac{q}{2}-1}\big(m+1+\mathcal{H}^{q(p-\frac{1}{4})}(u_0)\big).
\end{align}
Thus 
\begin{align*}
    \mathbb{P}\Big(\sup_{t\geq0}\big\{\mathcal{H}^p(u_t)-\mathcal{H}^p(u_0)+&\hspace{1mm}{p\alpha}\int_0^t\mathcal{H}^p(u_s)\dd s-K_{1,p} t\big\}\geq \rho  \Big)\\\leq&\ C_{p,q}\sum_{m=0}^\infty \frac{(m+1)^{\frac{q}{2}-1}}{(\rho+ m)^q}\big(m+1+\mathcal{H}^{q(p-\frac{1}{4})}(u_0)\big).
\end{align*}
Summing the series as in \eqref{08010427} completes the proof.
\end{proof}
Next we define the stopping time
\begin{align}\label{03190353-2}
\tau_{1,p}^u:=\inf \{t\geq0:\mathcal{E}_{1,p}^u(t)\geq \rho+Kt+\mathscr{C}\mathcal{H}^p(u_0)\}.
\end{align}
\begin{corollary}\label{03090531}
    For any $p,q\geq 1$, $\rho\geq 1$, $K\geq K_{1,p}+1$, $\mathscr{C}\geq 1$, and $l\geq 0$,  we have
    $$\mathbb{P}\big( l\leq\tau_{1,p}^u<\infty\big)\leq C_{p,q}\big[{1+\mathcal{H}^{q(p-\frac{1}{4})}(u_0)}\big]{(\rho+l)^{-\frac{q}{2}+1}}.$$
\end{corollary}
\begin{proof}
    Since $K\geq K_{1,p}+1$, on the event $\{l\leq\tau_{1,p}^u<\infty\}$, we have
    \begin{align*}
        \rho+l\leq& \hspace{1mm}\rho+(K-K_{1,p})\tau_{1,p}^u+(\mathscr{C}-1)\mathcal{H}^p(u_0)\\=&\hspace{1mm}\mathcal{E}_{1,p}^u(\tau_{1,p}^u)-K_{1,p}\tau_{1,p}^u-\mathcal{H}^p(u_0)\\\leq&\hspace{1mm}\sup_{t\geq 0}\big\{\mathcal{E}_{1,p}^u(t)-K_{1,p}t-\mathcal{H}^p(u_0)\big\}.
    \end{align*}
    Notice that $\rho\geq 1$, then we have by Lemma \ref{03050812},
    \begin{align*}
        \mathbb{P}\big(l\leq\tau_{1,p}^u<\infty\big)\leq&\ \mathbb{P}\big(\sup_{t\geq 0}\{\mathcal{E}_{1,p}^u(t)-K_{1,p}t-\mathcal{H}^p(u_0)\}\geq \rho+l\big)\\\leq&\ C_{p,q}\big({1+\mathcal{H}^{q(p-\frac{1}{4})}(u_0)}\big){(\rho+l)^{-\frac{q}{2}+1}}.
    \end{align*}
\end{proof}
To estimate the $H^2$-norm of solutions, we consider the  functional $\mathcal{G}$ defined in \eqref{03090247-2}.
\begin{lemma}\label{02231236}For any $p\geq1$, $t\geq 0$, and any stopping time $\tau$ such that $\mathbb{E}\big[\tau^2;\tau<\infty\big]<\infty$, we have
\begin{align*}
   &{\rm{(i)}}\ \mathbb{E}\big[\mathcal{G}^p(u_t)\big]\leq e^{-{\alpha}p t}\mathcal{G}^{p}(u_0)+C_p e^{-\alpha pt}(\|u_0\|^{10p}+{\mathcal{H}}^{5p}(u_0))+C_p,\\&
   {\rm{(ii)}}\ \mathbb{E}\big[\!\!\sup_{s\in[0,t]}\!\!\big\{\mathcal{G}^p(u_s)\hspace{-0.5mm}+\hspace{-0.5mm}\alpha p\int_0^s\mathcal{G}^p(u_r)\dd r\big\}\big]\hspace{-0.5mm}\leq\hspace{-0.5mm} 2\mathcal{G}^p(u_0)\hspace{-0.5mm}+\hspace{-0.5mm}C_p(\|u_0\|^{10p}\hspace{-0.5mm}+\hspace{-0.5mm}\mathcal{H}^{5p}(u_0)\hspace{-0.5mm}+\hspace{-0.5mm}t),\\&
   {\rm{(iii)}}\ \mathbb{E}\big[\mathcal{G}^p(u_\tau);\tau<\infty\big]\leq C_p(1+\|u_0\|^{10p}+\mathcal{H}^{5p}(u_0)+\mathcal{G}^p(u_0))\\&\hspace{7cm}\times\big(\mathbb{E}\big[(1+\tau)^{2};\tau<\infty\big]\big)^{\frac{1}{2}}.
\end{align*}
\end{lemma}
\begin{proof}
The proof is divided into four steps.  

\noindent{\it Step 1: estimate for $\mathcal{G}(u_t)$}. We first apply It\^o's formula to $\mathcal{G}(u_t)$:
\begin{align}\label{03040446}
 \nonumber    \mathcal{G}(u_t)-\mathcal{G}(u_0)&=\int_0^t\langle D\mathcal{G}(u_s),-\alpha u_s+i\partial_x^2u_s-i|u_s|^2u_s\rangle \dd s\\&\quad+\frac{1}{2}\sum_{j=1}^\infty b_j^2\int_0^t\big[D^2\mathcal{G}(u_s)(e_j,e_j)+D^2\mathcal{G}(u_s)(ie_j,ie_j)\big]\dd s\nonumber
\\&\quad+\int_0^t\langle D\mathcal{G}(u_s),b\dd W_s\rangle,
\end{align}
where 
\begin{align}\label{03090253}
   \nonumber D\mathcal{G}(u)(v) &=  \langle \partial_x^2u,\partial_x^2v\rangle+\mathrm{Re}\int_{\mathbb{R}} (|u|^2\partial_x\overline{u}\partial_xv+\overline{u}v|\partial_xu|^2)\dd x\\&\quad +\mathrm{Re}\int_{\mathbb{R}} (u^2\partial_x\overline{u}\partial_x\overline{v}+{u}v(\partial_x\overline{u})^2)\dd x,
\end{align}
and
\begin{align}\label{03310222}
    \nonumber D^2\mathcal{G}(u)(v,w)&= \langle \partial_x^2w,\partial_x^2v\rangle+\mathrm{Re}\int_{\mathbb{R}} (|u|^2\partial_x\overline{w}\partial_xv+\overline{w}v|\partial_xu|^2)\dd x\\\nonumber&\quad +\mathrm{Re}\int_{\mathbb{R}}(u\overline{w}+w\overline{u})\partial_x\overline{u}\partial_xv\dd x\nonumber\\&\quad +\mathrm{Re}\int_{\mathbb{R}} \overline{u}v(\partial_x\overline{u} \partial_xw+\partial_x{u}\partial_x\overline{w})\dd x\nonumber\\\nonumber&\quad+\mathrm{Re}\int_{\mathbb{R}} (u^2\partial_x\overline{w}\partial_x\overline{v}+{w}v(\partial_x\overline{u})^2)\dd x\nonumber\\&\quad +2\mathrm{Re}\int_{\mathbb{R}} u{w}\partial_x\overline{u}\partial_x\overline{v}
    \dd x+2\mathrm{Re}\int_{\mathbb{R}} {u}v\partial_x\overline{u} \partial_x\overline{w}\dd x.
\end{align}
Let us compute each term in \eqref{03040446}. For the first term, note that
\begin{align}\label{03090252}
    \nonumber\langle D\mathcal{G}(u_s),-\alpha u_s +i\partial_x^2u_s-i|u_s|^2u_s&\rangle= \nonumber \langle \partial_x^2u_s,\partial_x^2(-\alpha u_s+i\partial_x^2 u_s-i|u_s|^2u_s)\rangle\\\  \nonumber&\hspace{-0.5mm}+\langle \partial_xu_s|u_s|^2,\partial_x(-\alpha u_s\hspace{-0.5mm}+\hspace{-0.5mm}i\partial_x^2 u_s\hspace{-0.5mm}-\hspace{-0.5mm}i|u_s|^2u_s) \rangle\\\nonumber&+\langle u_s|\partial_xu_s|^2,(-\alpha u_s+i\partial_x^2 u_s-i|u_s|^2u_s)\rangle\\\nonumber&+\langle u_s^2\partial_x \overline{u}_s,\partial_x(-\alpha u_s+i\partial_x^2 u_s-i|u_s|^2u_s)\rangle\\\nonumber&+\langle \overline{u}_s(\partial_xu_s)^2,(-\alpha u_s+i\partial_x^2 u_s-i|u_s|^2u_s)\rangle \\&= \mathrm{I}+\mathrm{II}+\mathrm{III}+\mathrm{IV}+\mathrm{V}.
\end{align}
We compute the five terms separately. For $\mathrm{I}$, integrating by parts,
\begin{align}\label{03300347-1}
    \nonumber  \mathrm{I}&=-\alpha\|\partial_x^2u_s\|^2-\langle \partial_x^2u_s,i\partial_x^2(|u_s|^2u_s)\rangle\\&\nonumber = \hspace{-0.5mm}-\hspace{-0.5mm}\alpha\|\partial_x^2u_s\|^2\hspace{-0.5mm}-\hspace{-0.5mm}\langle \partial_x^2u_s,i\big(2(\partial_xu_s)^2\overline{u}_s\hspace{-0.5mm}+\hspace{-0.5mm}2|u_s|^2\partial_x^2u_s\hspace{-0.5mm}+\hspace{-0.5mm}4u_s|\partial_xu_s|^2\hspace{-0.5mm}+\hspace{-0.5mm}u_s^2\partial_x^2\overline{u}_s\big)\rangle\\&= \hspace{-0.5mm}-\alpha\|\partial_x^2u_s\|^2-\langle \partial_x^2u_s,i\big(2(\partial_xu_s)^2\overline{u}_s+4u_s|\partial_xu_s|^2+u_s^2\partial_x^2\overline{u}_s\big)\rangle.
\end{align}
For term $\mathrm{II}$, we have
\begin{align}\label{03300347-2}
 \mathrm{II}=&\nonumber -\alpha\||\partial_xu_s||u_s|\|^2-\langle \partial_x(\partial_xu_s|u_s|^2),i\partial_x^2u_s\rangle-\langle \partial_xu_s|u_s|^2,i\partial_x(|u_s|^2u_s)\rangle\\=&\nonumber -\alpha\||\partial_xu_s||u_s|\|^2-\langle \partial_x^2u_s|u_s|^2+|\partial_xu_s|^2{u}_s+(\partial_xu_s)^2\overline{u}_s,i\partial_x^2u_s\rangle\\&\nonumber -\langle \partial_xu_s|u_s|^2,i(2u_s\partial_xu_s\overline{u}_s+u_s^2\partial_x\overline{u}_s)\rangle\\=&\nonumber -\alpha\||\partial_xu_s||u_s|\|^2-\langle |\partial_xu_s|^2{u}_s+(\partial_xu_s)^2\overline{u}_s,i\partial_x^2u_s\rangle\\&-\langle \partial_xu_s|u_s|^2,iu_s^2\partial_x\overline{u}_s\rangle.
 \end{align}
 Similarly, for $\mathrm{IV}$,  
 \begin{align}\label{03300347-3}
        \mathrm{IV}=&\nonumber -\alpha\langle u_s^2,(\partial_xu_s)^2\rangle-\langle \partial_x(u_s^2\partial_x \overline{u}_s),i\partial_x^2u_s\rangle\\&\nonumber -\langle u_s^2\partial_x\overline{u}_s,i\partial_x(|u_s|^2u_s)\rangle\\=&\nonumber -\alpha\langle u_s^2,(\partial_xu_s)^2\rangle-\langle u_s^2\partial^2_x \overline{u}_s,i\partial_x^2u_s\rangle-2\langle u_s|\partial_x u_s|^2,i\partial_x^2u_s\rangle\\&\nonumber -2\langle u_s^2\partial_x\overline{u}_s,i |u_s|^2\partial_x u_s\rangle-\langle u_s^2\partial_x\overline{u}_s,i u_s^2\partial_x \overline{u}_s\rangle\\=&\nonumber -\alpha\langle u_s^2,(\partial_xu_s)^2\rangle-\langle u_s^2\partial^2_x \overline{u}_s,i\partial_x^2u_s\rangle-2\langle u_s|\partial_x u_s|^2,i\partial_x^2u_s\rangle\\&-2\langle u_s^2\partial_x\overline{u}_s,i |u_s|^2\partial_x u_s\rangle.
 \end{align}
Arguing as for $\mathrm{II}$ and $\mathrm{IV}$, we obtain
\begin{align}\label{03300347-4}
     \mathrm{III}&=-\alpha\||\partial_xu_s||u_s|\|^2+\langle u_s|\partial_xu_s|^2,i\partial_x^2u_s\rangle,\\
    \mathrm{V}&=-\alpha\langle(\partial_xu_s)^2,u_s^2\rangle+\langle \overline{u}_s(\partial_x u_s)^2,i\partial_x^2u_s\rangle-\langle \overline{u}_s(\partial_xu_s)^2,i|u_s|^2u_s\rangle.\label{03300347-5}
\end{align}
Combining \eqref{03300347-1}--\eqref{03300347-5} with \eqref{03090252}, we find
\begin{align}\label{03050118}
     \mathrm{I+II+III+IV+V}&=\nonumber-\alpha\|\partial_x^2u_s\|^2-2\alpha\||\partial_xu_s||u_s|\|^2-2\alpha\langle(\partial_xu_s)^2,(u_s)^2\rangle\\&\quad \nonumber-2\langle \partial_x^2u_s,i\big((\partial_xu_s)^2\overline{u}_s+u_s|\partial_xu_s|^2\big)\rangle\\&=\nonumber-2\alpha\mathcal{G}(u_s)-\frac{\alpha}{2}\|\partial_x| u_s|^2\|^2\\&\nonumber\quad-\langle \partial_x^2u_s,i\big(2(\partial_xu_s)^2\overline{u}_s+2u_s|\partial_xu_s|^2\big)\rangle\nonumber\\&\leq\nonumber-2\alpha\mathcal{G}(u_s)+C\|\partial_x^2u_s\|\||\partial_xu_s|^2u_s\|\\&\leq\nonumber-2\alpha\mathcal{G}(u_s)+C\|\partial_x^2u_s\|\|\partial_xu_s\|_{L^4}^2\|u_s\|_{L^\infty}\\&\leq\nonumber-2\alpha\mathcal{G}(u_s)+C\|\partial_x^2u_s\|^{\frac{3}{2}}\|\partial_xu_s\|^{2}\|u_s\|^{\frac{1}{2}}\\&\leq-\frac{3\alpha}{2} \mathcal{G}(u_s)+C\|\partial_xu_s\|^8\|u_s\|^2.
\end{align}
By \eqref{03090253}, the stochastic integral in \eqref{03040446} equals
\begin{align}
    \nonumber\int_0^t\langle& D\mathcal{G}(u_s),b\dd W_s\rangle= \mathrm{Re}\sum_{j=1}^\infty b_j\int_0^t\int_\mathbb{R}\partial_x^2\overline{u}_s\partial_x^2e_j\dd x(\dd\beta^1_j(s)+i\dd \beta^2_j(s))\\&\nonumber+\mathrm{Re}\sum_{j=1}^\infty b_j\int_0^t\int_{\mathbb{R}} (|u_s|^2\partial_x\overline{u}_s\partial_xe_j+\overline{u}_se_j|\partial_xu_s|^2)\dd x(\dd \beta^1_j(s)+i\dd \beta^2_j(s))\\&\nonumber+\mathrm{Re}\sum_{j=1}^\infty b_j\int_0^t\int_{\mathbb{R}} (\overline{u}_s^2\partial_x{u}_s\partial_xe_j+{u}_se_j(\partial_x\overline{u}_s)^2)\dd x(\dd \beta^1_j(s)+i\dd \beta^2_j(s)).
\end{align}
Using \eqref{03310222}, the It\^o correction term in \eqref{03040446}
satisfies
\begin{align}\label{03300310}
      \nonumber\frac{1}{2}\sum_{j=1}^\infty &b_j^2\int_0^tD^2\mathcal{G}(u_s)(e_j,e_j)\dd s+\frac{1}{2}\sum_{j=1}^\infty b_j^2\int_0^tD^2\mathcal{G}(u_s)(ie_j,ie_j)\dd s\\=&\ \nonumber\frac{1}{2}\sum_{j=1}^\infty b_j^2\int_0^t\big[2\|\partial_x^2e_j\|^2+2\int_{\mathbb{R}} \big(|u_s|^2(\partial_x{e}_j)^2+{e}_j^2|\partial_xu_s|^2\big)\dd x\\&\nonumber+12\mathrm{Re}\int_\mathbb{R}u_s\partial_x\overline{u}_s{e}_j\partial_xe_j\dd x\big]\dd s\\\leq&\ \nonumber C\sum_{j=1}^\infty b_j^2\int_0^t\big( \|\partial_x^2e_j\|^2+\|u_s\|^2\|\partial_x e_j\|^2_{L^\infty}+\|\partial_xu_s\|^2\|e_j\|_{H^1}^2\\&\nonumber+\|u_s\|_{L^\infty}\|\partial_xu_s\|\|e_j\|_{L^\infty}\|\partial_x e_j\|\big)\dd s\\\leq &\ C\big(t+\int_0^t\|u_s\|^2\dd s+\int_0^t\|\partial_xu_s\|^2\dd s\big).
\end{align}
Combining  \eqref{03050118} and \eqref{03300310} with \eqref{03040446}, we obtain
\begin{align}
    \nonumber\mathcal{G}(u_t)&-\mathcal{G}(u_0)+\frac{3\alpha}{2}\int_0^t\mathcal{G}(u_s) \dd s\\&\leq  C\int_0^t\big(\|\partial_xu_s\|^8\|u_s\|^2+\|\partial_xu_s\|^2+\|u_s\|^2+1\big)\dd s\nonumber\\&\quad+\sum_{j=1}^\infty b_j\int_0^t\langle\partial_x^2u_s,\partial_x^2e_j\rangle\dd \beta_j^1(s)+\sum_{j=1}^\infty b_j\int_0^t\langle\partial_x^2u_s,i\partial_x^2e_j\rangle\dd \beta_j^2(s)\nonumber\\&\quad\nonumber+\sum_{j=1}^\infty b_j\mathrm{Re}\int_0^t\hspace{-0.2mm}\int_{\mathbb{R}} (|u_s|^2\partial_x\overline{u}_s\partial_xe_j+\overline{u}_se_j|\partial_xu_s|^2)\dd x(\dd \beta^1_j(s)+i\dd \beta^2_j(s))\\&\quad\nonumber+\sum_{j=1}^\infty b_j\mathrm{Re}\int_0^t\hspace{-0.2mm}\int_{\mathbb{R}} (\overline{u}_s^2\partial_x{u}_s\partial_xe_j+{u}_se_j(\partial_x\overline{u}_s)^2)\dd x(\dd \beta^1_j(s)+i\dd \beta^2_j(s)).
\end{align}
\noindent{\it Step 2: Proof of \rm{(i)}}.
 We only give a proof when $p\geq 2$; the case $p\in[1,2)$ is similar to Case 2 of Lemma \ref{03050812-1}.\par Applying It\^o's formula to $\mathcal{G}^p(u_t)$, we obtain
 \begin{align}\label{02210734-1}
     \nonumber \mathcal{G}^p(u_t)&-\mathcal{G}^p(u_0)+\frac{3\alpha p}{2}\int_0^t\mathcal{G}^p(u_s) \dd s\\&\leq Cp\int_0^t\mathcal{G}^{p-1}(u_s)\big(1+\|u_s\|^{10}+\mathcal{H}^{5}(u_s)\big)\dd s+M_t^{2,p}+A_t^{2,p},
\end{align}
where $M^{2,p}$ is the martingale part of $\mathcal{G}^p(u_t)$, given by\begin{align}\label{03300458-1}
    &\nonumber M^{2,p}_t=p\sum_{j=1}^\infty b_j\mathrm{Re}\int_0^t\mathcal{G}^{p-1}(u_s)\\&\hspace{2cm}\nonumber\times\Big(\int_\mathbb{R}\partial_x^2u_s\partial_x^2e_j\dd x+\int_{\mathbb{R}} (|u_s|^2\partial_x\overline{u}_s\partial_xe_j\hspace{-0.5mm}+\hspace{-0.5mm}\overline{u}_se_j|\partial_xu_s|^2)\dd x\\&\hspace{2cm}+\int_{\mathbb{R}} (\overline{u}_s^2\partial_x{u}_s\partial_xe_j\hspace{-0.5mm}+\hspace{-0.5mm}{u}_se_j(\partial_x\overline{u}_s)^2)\dd x\Big)(\dd \beta^1_j(s)\hspace{-0.5mm}+\hspace{-0.5mm}i\dd \beta^2_j(s)).
\end{align}
The It\^o correction of $\mathcal{G}^{p}(u_t)$ is given by
   \begin{align}\label{03300458}
   \nonumber \quad A^{2,p}_t=&\hspace{1mm}\frac{p(p-1)}{2}\int_0^t\mathcal{G}^{p-2}(u_s)\sum_{j=1}^\infty b_j^2\Big|\int_\mathbb{R}\partial_x^2u_s\partial_x^2e_j\dd x+\int_{\mathbb{R}} (|u_s|^2\partial_x\overline{u}_s\partial_xe_j\hspace{-0.5mm}\\&+\overline{u}_se_j|\partial_xu_s|^2)\dd x+\int_{\mathbb{R}} (\overline{u}_s^2\partial_x{u}_s\partial_xe_j\hspace{-0.5mm}+\hspace{-0.5mm}{u}_se_j(\partial_x\overline{u}_s)^2)\dd x\Big|^2\dd s\nonumber\\\nonumber\leq&\hspace{1mm} C_p\int_0^t \mathcal{G}^{p-2}(u_s)\sum_{j=1}^\infty b_j^2\big(\|\partial_xu_s\|^2\|\partial_x^3e_j\|^2\\&\quad +\|u_s\|_{L^4}^4\|\partial_x u_s\|^2\|\partial_xe_j\|_{L^\infty}^2+\|\partial_xu_s\|^4\|u_s\|_{L^\infty}^2\|e_j\|_{L^\infty}^2\big)\dd s\nonumber\\\leq&\nonumber \hspace{1mm} C_p\sum_{j=1}^\infty b_j^2\|e_j\|_{H^3}^2\int_0^t \mathcal{G}^{p-2}(u_s)\big(\|\partial_xu_s\|^2+\|u_s\|^4_{L^4}\|\partial_x u_s\|^2\\&\quad+\|\partial_xu_s\|^5\|u_s\| \big)\dd s.
    \end{align}
Thus combining \eqref{02210734-1} and \eqref{03300458}, we obtain that for any $t\geq 0$,
\begin{align}\label{02220446-1}
    \nonumber \mathcal{G}^p(&u_t)+\frac{3\alpha p}{2}\int_0^t\mathcal{G}^p(u_s)\dd s-\mathcal{G}^p(u_0)\\\nonumber\leq &\ \frac{\alpha p}{4}\int_0^t\mathcal{G}^p(u_s)\dd s+C_p\int_0^t\big(1+\|u_s\|^{10p}+\mathcal{H}^{5p}(u_s)\big)\dd s+M^{2,p}_t\\&+C_{p}\int_0^t\mathcal{G}^{p-2}(u_s)\big(\mathcal{H}^2(u_s)+\mathcal{H}^\frac{5}{2}(u_s)\|u_s\|\big)\dd s\nonumber\\\leq&\ \frac{\alpha p}{2}\int_0^t\mathcal{G}^p(u_s)\dd s+C_p\int_0^t\big(1+\|u_s\|^{10p}+\mathcal{H}^{5p}(u_s)\big)\dd s+M^{2,p}_t.
\end{align}
Applying It\^o's formula to $e^{\alpha pt}\mathcal{G}^p(u_t)$ and taking expectations on both sides, we obtain that
\begin{align*}
    e^{\alpha pt}\mathbb{E}\big[\mathcal{G}^p(u_t)\big]-\mathcal{G}^p(u_0)\leq C_p\int_0^te^{\alpha ps}\mathbb{E}[1+\|u_s\|^{10p}+\mathcal{H}^{5p}(u_s)]\dd s.
\end{align*}
By taking $\gamma=\alpha p$ in Lemmas \ref{02280441}(i) and \ref{02280436}(i), we obtain that
\begin{align*}
    \mathbb{E}\big[\mathcal{G}^p(u_t)\big]-e^{-\alpha pt}\mathcal{G}^p(u_0)\leq &\hspace{1mm}C_p\int_0^te^{-\alpha p(t-s)}\dd s+C_p\mathbb{E}[\int_0^te^{\alpha p(s-t)}\mathcal{H}^{5p}(u_s)\dd s]\\&+C_p\mathbb{E}[\int_0^te^{\alpha p(s-t)}\|u_s\|^{10p}\dd s]\\\leq&\hspace{1mm} C_p+e^{-\alpha p t}C_p(\|u_0\|^{10p}+\mathcal{H}^{5p}(u_0)),
\end{align*}
implying (i).   

\noindent{\it Step 3: Proof of $\mathrm{(ii)}$}. From \eqref{02220446-1} we deduce that
\begin{align*}
    \nonumber\sup_{s\in[0,t]}\big\{\mathcal{G}^p(u_s)&+\alpha p\int_0^s\mathcal{G}^p(u_r)\dd r\big\}-\mathcal{G}^p(u_0)\\\leq&\ \nonumber\nonumber C_p\int_0^t\big(1+\|u_s\|^{10p}+\mathcal{H}^{5p}(u_s)\big)\dd s+\sup_{s\in[0,t]}M^{2,p}_s.
\end{align*}
Taking expectation on both sides, we obtain
\begin{align}\label{03300509}
 \nonumber \mathbb{E}\big[\sup_{s\in[0,t]}&\big\{\mathcal{G}^p(u_s)+\alpha p\int_0^s\mathcal{G}^p(u_r)\dd r\big\}\big]-\mathcal{G}^p(u_0) \\\leq&\ C_p\mathbb{E}\big[\int_0^t\big(1+\|u_s\|^{10p}+\mathcal{H}^{5p}(u_s)\big)\dd s\big]+\mathbb{E}\big[\sup_{s\in[0,t]}M^{2,p}_s\big].
\end{align}
Moreover, the Burkholder--Davis--Gundy inequality and \eqref{03300458-1}
imply
\begin{align}\label{03300512}
    \nonumber \ \mathbb{E}&\big[\mathop{\sup}_{s\in[0,t]} M^{2,p}_s \big]\leq C_p(\sum_{j=1}^\infty b_j^2\|\partial_x^2e_j\|^2)^{\frac{1}{2}}\mathbb{E}\big[(\int_0^t\mathcal{G}^{2p-1}(u_s)\dd s)^{\frac{1}{2}}\big]\\&\nonumber \quad+C_p\big(\sum_{j=1}^\infty b_j^2\|\partial_{x}e_j\|_{L^\infty}^2\big)^{\frac{1}{2}}\mathbb{E}\big[\big(\int_0^t\mathcal{G}^{2p-2}(u_s)\|u_s\|_{L^4}^4\|\partial_xu_s\|^2\dd s\big)^{\frac{1}{2}}\big]\\&\nonumber \quad+C_p\big(\sum_{j=1}^\infty b_j^2\|e_j\|_{L^\infty}^2\big)^{\frac{1}{2}}\mathbb{E}\big[\big(\int_0^t\mathcal{G}^{2p-2}(u_s)\|\partial_xu_s\|^5\|u_s\|\dd s\big)^{\frac{1}{2}}\big]\\&\nonumber \leq\frac{1}{2}\mathbb{E}\big[\sup_{s\in[0,t]}\mathcal{G}^{p}(u_s)\big]+\frac{p\alpha}{2}\mathbb{E}\big[\int_0^t\mathcal{G}^p(u_s)\dd s\big]+C_p\mathbb{E}\big[\int_0^t\|\partial_xu_s\|^p\|u_s\|_{L^4}^{2p}\dd s\big]\\&\quad+C_p\mathbb{E}\big[\int_0^t\|\partial_xu_s\|^{\frac{5p}{2}}\|u_s\|^\frac{p}{2}\dd s\big]+C_pt.
\end{align}
 Lemmas \ref{02280441}(ii) and \ref{02280436}(ii) and inequalities \eqref{03300509} and \eqref{03300512} imply that
\begin{align*}
    \mathbb{E}\big[\sup_{s\in[0,t]}\big\{\mathcal{G}^p(u_s)&+{\alpha p}\int_0^s\mathcal{G}^p(u_r)\dd r\big\}\big] \\\leq&\ 2\mathcal{G}^p(u_0)+C_p\mathbb{E}\big[\int_0^t\big(1+\|u_s\|^{10p}+\mathcal{H}^{5p}(u_s)\big)\dd s\big]\\\leq& \ 2\mathcal{G}^p(u_0)+C_p\|u_0\|^{10 p}+C_{p}\mathcal{H}^{5p}(u_0)+C_pt,
\end{align*}
which proves (ii).  

\noindent{\it Step 4: Proof of \rm{(iii)}}.
 To prove {\rm{(iii)}}, we deduce from \eqref{02210734-1} and \eqref{03300458} that
\begin{align}\label{04290441}
    \mathcal{G}^p(u_t)+\alpha p\int_0^t\mathcal{G}^p(u_s)\dd s&\leq \mathcal{G}^p(u_0)+ M^{2,p}_t\nonumber\\&\quad+C_p\int_0^t\big(1+\|u_s\|^{10p}+\mathcal{H}^{5p}(u_s)\big)\dd s.
\end{align}
Applying It\^o's formula to the process $(t+1)^{-2}\mathcal{G}^p(u_t)$, we
find\begin{align*}
    &(t+1)^{-2}\mathcal{G}^p(u_t)-\mathcal{G}^p(u_0)+\alpha p\int_0^t(s+1)^{-2}\mathcal{G}^p(u_s)\dd s\\&\quad\leq  \int_0^t(s+1)^{-2}\dd M_s^{2,p}+ C_p\int_0^t(s+1)^{-2}\big(1+\|u_s\|^{10p}+\mathcal{H}^{5p}(u_s)\big)\dd s\\&\qquad \hspace{2mm}-2\int_0^t(s+1)^{-3}\mathcal{G}^p(u_s)\dd s.
\end{align*}
Replacing $t$ by $\tau\wedge n$, taking expectation, and by Lemmas~\ref{02280441}(i) and~\ref{02280436}(i), we obtain that
\begin{align}\label{03301017}
    \nonumber\mathbb{E}\big[\big((\tau\wedge n)+1\big)^{-2}\mathcal{G}^p(u_{\tau\wedge n})\big]&-\mathcal{G}^p(u_0)+\alpha p\mathbb{E}\big[\int_0^{\tau\wedge n}(s+1)^{-2}\mathcal{G}^p(u_s)\dd s\big]\\\leq&\ C_p\mathbb{E}\big[\int_0^{\tau\wedge n}(s+1)^{-2}(1+\|u_s\|^{10p}+\mathcal{H}^{5p}(u_s))\dd s\big]\nonumber\\\leq&\ C_p\int_0^\infty (s+1)^{-2}\dd s\hspace{0.5mm}(1+\|u_0\|^{10p}+\mathcal{H}^{5p}(u_0))\nonumber\\\leq&\ C_p\big(1+\|u_0\|^{10p}+\mathcal{H}^{5p}(u_0)\big).
\end{align}
Notice that on the event $\{\tau=\infty\}$, we have
\begin{align*}
    \mathbb{E}\big[\big((\tau\wedge n)+1\big)^{-2}
    &\mathcal{G}^p(u_{\tau\wedge n});\tau=\infty\big]
    =\mathbb{E}\big[(n+1)^{-2}\mathcal{G}^p(u_{n});\tau=\infty\big]\\
    &\leq\mathbb{E}\big[(n+1)^{-2}\mathcal{G}^p(u_{n})\big]\\
    &\leq C(n+1)^{-2}\big(1+\mathcal{G}^p(u_0)+\|u_0\|^{10p}
    +\mathcal{H}^{5p}(u_0)\big)\to 0
\end{align*}
as $n\rightarrow\infty$. Therefore, there exists a subsequence $\{n_k\}_{k\geq 1}$ such that $n_k\rightarrow\infty$ as $k\rightarrow\infty$ and
\begin{align*}
    \big((\tau\wedge n_k)+1\big)^{-2}\mathcal{G}^p(u_{\tau\wedge n_k})
    \mathbf{1}_{\tau=\infty}\to0
    \qquad\mathbb{P}\text{-a.s. as }k\rightarrow\infty.
\end{align*}
Taking $n=n_k$ in \eqref{03301017} and applying Fatou's lemma as
$k\to\infty$, we get
\begin{align*}
\nonumber\mathbb{E}\big[(\tau+1)^{-2}\mathcal{G}^p(u_{\tau});\tau<\infty\big]\leq  C_p\big(1+\|u_0\|^{10p}+\mathcal{H}^{5p}(u_0)+\mathcal{G}^p(u_0)\big).\nonumber
\end{align*}
H\"older's inequality then gives
\begin{align*}
     \mathbb{E}\big[\mathcal{G}^{\frac{p}{2}}(u_\tau);\tau<\infty\big]&\leq C_p\big(1+\|u_0\|^{10p}+\mathcal{H}^{5p}(u_0)+\mathcal{G}^p(u_0)\big)^{\frac{1}{2}}\\&\quad\times\mathbb{E}\big[(1+\tau)^2;\tau<\infty\big]^{\frac{1}{2}}.
\end{align*}
Replacing $p$ by $2p$, we obtain (iii).\end{proof}
As above, we set
\begin{align*}
    \mathcal{E}_{2,p}^u(t):=\mathcal{G}^p(u_t)+{\alpha p}\int_0^t\mathcal{G}^p(u_s)\dd s.
\end{align*}
\begin{lemma}\label{03080704}For any $p\geq 1$, there exist $K_{2,p}>0$ and $\mathcal{C}_{p}^\sharp >0$ such that, for any $q\geq 1$ and $\rho\geq 3\mathcal{C}_p^\sharp\vee 1$, 
    \begin{align*}
\mathbb{P}\Big\{\sup_{t\geq0}&\big\{\mathcal{E}_{2,p}^u(t)-K_{2,p}t-\mathcal{G}^{p}(u_0)-\mathcal{C}_{p}^\sharp \|u_0\|^{10p}-\mathcal{C}_{p}^\sharp \mathcal{H}^{5p}(u_0)\big\}\geq\rho\Big\}\\\leq &\ \hspace{1mm}C_{p,q}\big[1+\|u_0\|^{10q(p+\frac{1}{2})}+\mathcal{H}^{5q(p+\frac{1}{2})}(u_0)+\mathcal{G}^{q(p+\frac{1}{2})}(u_0)\big]\rho^{-\frac{q}{2}+1}.
\end{align*}
\end{lemma}
\begin{proof}
 As in the proof of Lemma~\ref{03050812-1}, it suffices to consider the case $q>2$. By \eqref{04290441}, there exists $\mathcal{C}_{p}^\sharp>0$ such that
\begin{align*}
    \mathcal{E}_{2,p}^u(t)-\mathcal{G}^p(u_0)\nonumber&\leq 5\alpha p\mathcal{C}_{p}^\sharp \int_0^t\mathcal{H}^{5p}(u_s)\dd s\\&\quad+5\alpha p\mathcal{C}_{p}^\sharp \int_0^t\|u_s\|^{10p}\dd s+\mathcal{C}_{p}^\sharp t+M^{2,p}_t,
\end{align*}
where $M^{2,p}$ is given in \eqref{03300458-1}. Setting
$K_{2,p}:=\mathcal{C}_{p}^\sharp(K_{0,5p}+K_{1,5p}+1)+1/3$, we obtain
\begin{align*}
    &\Big\{\sup_{t\geq0}\big\{\mathcal{E}_{2,p}^u(t)-K_{2,p}t-\mathcal{G}^p(u_0)-\mathcal{C}_{p}^\sharp \|u_0\|^{10p}-\mathcal{C}_{p}^\sharp \mathcal{H}^{5p}(u_0)\big\}\geq\rho\Big\}\\&\subseteq\Big\{\sup_{t\geq 0}\big\{5\alpha p\mathcal{C}_{p}^\sharp \int_0^t\|u_s\|^{10p}\dd s\hspace{-0.5mm}-\hspace{-0.5mm}\mathcal{C}_{p}^\sharp {K_{0,5p}}t\hspace{-0.3mm}-\hspace{-0.3mm}\mathcal{C}_{p}^\sharp \|u_0\|^{10p}\big\}\geq\frac{\rho}{3}\Big\}\\&\qquad {\scalebox{1.3}{$\cup$}}\Big\{\sup_{t\geq 0}\{5\alpha p\mathcal{C}_{p}^\sharp \int_0^t\mathcal{H}^{5p}(u_s)\dd s-\mathcal{C}_{p}^\sharp {K_{1,5p}}t-\mathcal{C}_{p}^\sharp \mathcal{H}^{5p}(u_0)\}\geq\frac{\rho}{3}\Big\}\\&\qquad {\scalebox{1.3}{$\cup$}}\Big\{\sup_{t\geq 0}\big\{M^{2,p}_t-\frac{t}{3}\big\}\geq\frac{\rho}{3}\Big\}.
\end{align*}
Since $\rho\geq 3\mathcal{C}_p^\sharp $, by Lemmas \ref{03050812} and \ref{03050812-1},
\begin{align}\label{03310300}
    \nonumber\mathbb{P}\Big\{\sup_{t\geq0}\big\{\mathcal{E}_{2,p}^u(t)-K_{2,p}&t-\mathcal{G}^p(u_0)-\mathcal{C}_{p}^\sharp \|u_0\|^{10p}-\mathcal{C}_{p}^\sharp \mathcal{H}^{5p}(u_0)\big\}\geq\rho\Big\}\\&\leq C_{p,q}\big(1+\|u_0\|^{q(10p-1)}+\mathcal{H}^{q(5p-\frac{1}{4})}(u_0)\big)\rho^{-\frac{q}{2}+1}\nonumber\\&\quad+\mathbb{P}\Big\{\sup_{t\geq 0}\big\{M^{2,p}_t-\frac{t}{3}\big\}\geq\frac{\rho}{3}\Big\}.
\end{align}
The second term on the right-hand side of this inequality can be bounded by
\begin{align*}
    \nonumber  \mathbb{P}\Big\{\sup_{t\geq 0}\big\{M^{2,p}_t-\frac{t }{3}\big\}\geq\frac{\rho}{3}\Big\}&\leq \sum_{m=0}^\infty\mathbb{P}\Big\{\sup_{t\in [m,m+1]}\big\{M^{2,p}_t-\frac{t }{3}\big\}\geq\frac{\rho}{3}\Big\}\\& \leq\sum_{m=0}^\infty\mathbb{P}\Big\{\sup_{t\in [0,m+1]}M^{2,p}_t\geq\frac{\rho}{3}+\frac{m }{3}\Big\}\nonumber\\&\leq\sum_{m=0}^\infty{(\frac{\rho}{3}+\frac{m }{3})^{-q}}{\mathbb{E}\big[\sup_{t\in [0,m+1]}|M^{2,p}_t|^q\big]}.
\end{align*}
The Burkholder--Davis--Gundy inequality now gives
\begin{align*}
&\nonumber \ \mathbb{P}\Big\{\sup_{t\geq 0}\big\{M^{2,p}_t-\frac{t }{3}\big\}\geq\frac{\rho}{3}\Big\}\leq C_q\sum_{m=0}^\infty{(\rho+m)^{-q}}{\mathbb{E}\big[\sup_{0\leq t\leq m+1}|M_t^{2,p}|^q\big]}\\\leq
    &\nonumber \ C_q\sum_{m=0}^\infty {({\rho}+m)^{-q}}\Big\{\big(\sum_{j=1}^\infty b_j^2\|\partial_x^2e_j\|^2\big)^{\frac{q}{2}}\mathbb{E}\big[\big(\int_0^{m+1}\hspace{-2mm}\mathcal{G}^{2p-1}(u_s)\dd s\big)^{\frac{q}{2}}\big]\\&\nonumber +\big(\sum_{j=1}^\infty b_j^2\|\partial_{x}e_j\|_{L^\infty}^2\big)^{\frac{q}{2}}\mathbb{E}\big[\big(\int_0^{m+1}\hspace{-2mm}\mathcal{G}^{2p-2}(u_s)\|\partial_xu_s\|^2\|u_s\|^4_{L^4}\dd s\big)^{\frac{q}{2}}\big]\\&\nonumber +\big(\sum_{j=1}^\infty b_j^2\|e_j\|_{L^\infty}^2\big)^{\frac{q}{2}}\mathbb{E}\big[\big(\int_0^{m+1}\hspace{-2mm}\mathcal{G}^{2p-2}(u_s)\|\partial_xu_s\|^5\|u_s\|\dd s\big)^{\frac{q}{2}}\big]\Big\}\\\leq&\nonumber  \ C_{p,q}\sum_{m=0}^\infty{({\rho}+m)^{-q}}(m+1)^{\frac{q}{2}-1}\mathbb{E}\Big[\int_0^{m+1}\hspace{-2mm}\big(\mathcal{G}^{\frac{q(2p-1)}{2}}(u_s)\\&\qquad +\mathcal{G}^{{q(p-1)}}(u_s)\mathcal{H}^{q}(u_s)+\mathcal{G}^{{q(p-1)}}(u_s)\mathcal{H}^{\frac{5q}{4}}(u_s)\|u_s\|^{\frac{q}{2}}\big)\dd s\Big].
    \end{align*}
Applying Young's inequality and then
Lemmas~\ref{02280441}(ii), \ref{02280436}(ii), and~\ref{02231236}(ii), we
obtain    \begin{align*}
    &\mathbb{P}\Big\{\sup_{t\geq 0}\big\{M^{2,p}_t-\frac{t }{3}\big\}\geq\frac{\rho}{3}\Big\}\\\leq& \ C_{p,q}\sum_{m=0}^\infty{({\rho}+m)^{-q}}(m+1)^{\frac{q}{2}-1}\mathbb{E}\Big[\int_0^{m+1}\big(\mathcal{G}^{\frac{q(2p+1)}{2}}(u_s)\\&\qquad +\mathcal{H}^{5q(p+\frac{1}{2})}(u_s)+\|u_s\|^{10q(p+\frac{1}{2})}+1\big)\dd s\Big]\\
   \leq&\ C_{p,q}\sum_{m=0}^\infty{({\rho}+m)^{-q}}(m+1)^{\frac{q}{2}-1}\big(\|u_0\|^{10q(p+\frac{1}{2})}+\mathcal{H}^{5q(p+\frac{1}{2})}(u_0)+\mathcal{G}^{q(p+\frac{1}{2})}(u_0)\\&\qquad +m+1\big)\\\leq&\ C_{p,q} \big[1+\|u_0\|^{10q(p+\frac{1}{2})}+\mathcal{H}^{5q(p+\frac{1}{2})}(u_0)+\mathcal{G}^{q(p+\frac{1}{2})}(u_0)\big]\rho^{-\frac{q}{2}+1},
\end{align*}
where the last inequality is obtained by summing the series as in
\eqref{08010427}. Together with \eqref{03310300}, this gives the required inequality.
\end{proof}
Finally, we set
\begin{align}\label{03190353-3}
    \tau_{2,p}^u:=\inf\{t\geq 0: \mathcal{E}_{2,p}^u(t)\geq \rho+Kt+\mathscr{C}(\mathcal{G}^p(u_0)+\|u_0\|^{10p}+\mathcal{H}^{5p}(u_0))\}.
\end{align}
\begin{corollary}\label{03090542}
    For any $p,q\geq1$, $\rho\geq 3\mathcal{C}_p^\sharp\vee 1$, $K\geq K_{2,p}+1$, $\mathscr{C}\geq \mathcal{C}_{p}^\sharp \vee 1$, and $l\geq 0$, we have
    \begin{align*}
        \mathbb{P}(l\leq \tau_{2,p}^u<\infty)\leq\hspace{1mm}C_{p,q} \big[1&+\|u_0\|^{10q(p+\frac{1}{2})}+\mathcal{H}^{5q(p+\frac{1}{2})}(u_0)\\&+\mathcal{G}^{q(p+\frac{1}{2})}(u_0)\big]{\big(\rho+l\big)^{-\frac{q}{2}+1}}.
    \end{align*}
\end{corollary}
\begin{proof}
    Since $K\geq K_{2,p}+1$, we deduce that on the event $\{l\leq \tau_{2,p}^u<\infty\}$, 
    \begin{align*}
    \rho+l\leq&\hspace{1mm} \rho+(K-K_{2,p})\tau_{2,p}^u+(\mathscr{C}-1)\mathcal{G}^p(u_0)\\&+(\mathscr{C}-\mathcal{C}_p^\sharp)\big(\|u_0\|^{10p}+\mathcal{H}^{5p}(u_0)\big)\\
        =&\hspace{1mm}\mathcal{E}_{2,p}^u(\tau_{2,p}^u)-K_{2,p}\tau_{2,p}^u-\mathcal{G}^p(u_0)-\mathcal{C}_{p}^\sharp \|u_0\|^{10p}-\mathcal{C}_{p}^\sharp \mathcal{H}^{5p}(u_0)\\\leq&\ \sup_{t\geq 0}\big\{\mathcal{E}_{2,p}^u(t)-K_{2,p}t-\mathcal{G}^p(u_0)-\mathcal{C}_{p}^\sharp \|u_0\|^{10p}-\mathcal{C}_{p}^\sharp \mathcal{H}^{5p}(u_0)\big\}.
    \end{align*}
    By Lemma \ref{03080704}, we have
    \begin{align*}
        \mathbb{P}\big(l\leq \tau_{2,p}^u&<\infty\big)\\\leq& \hspace{1mm}\mathbb{P}\Big(\sup_{t\geq 0}\big\{\mathcal{E}_{2,p}^u(t)-K_{2,p}t-\mathcal{G}^p(u_0)-\mathcal{C}_{p}^\sharp\|u_0\|^{10p}\\&\hspace{3cm}-\mathcal{C}_{p}^\sharp \mathcal{H}^{5p}(u_0)\big\}\geq \rho+l\Big)\\\leq&\hspace{0.5mm} C_{p,q} \big[1+\|u_0\|^{10q(p+\frac{1}{2})}+\mathcal{H}^{5q(p+\frac{1}{2})}(u_0)+\mathcal{G}^{q(p+\frac{1}{2})}(u_0)\big]{\big(\rho+l\big)^{-\frac{q}{2}+1}}\hspace{-1mm}.
    \end{align*}
\end{proof}

\subsection{Weighted moment estimate}

We consider the space-time weight function $\psi:\mathbb{R}_+\times \mathbb{R}\rightarrow\mathbb{R}$ given in \eqref{E:psi-def}
and estimate the weighted moments $\|\psi(t)u_t\|^{2p}$. For any $t\geq0$, we denote
    \begin{align*}
        \mathcal{E}_{\psi}^{u}(t):=\|\psi(t)u_t\|^{2}+{\alpha }\int_0^t\|\psi(s)u_s\|^{2}\dd s.
    \end{align*}
\begin{lemma}\label{03080957}
    There exist $K_\psi>0$ and $\mathcal{C}_\psi>0$ such that for any $p\geq 1$, $q\geq 1$, $t\geq 0$, and $\rho\geq 3\mathcal{C}_\psi\vee 1$,
    \begin{align*}
        &{\rm{(i)}\ } \mathbb{E}\big[\|\psi(t)u_t\|^{2p}+{\alpha p}\int_0^t\|\psi(s)u_s\|^{2p}\dd s\big]\leq C_pt+C_p\big(\|u_0\|^{2p}+\mathcal{H}^p(u_0)\big),\\&
        {\rm{(ii)}\ } \mathbb{P}\Big\{\sup_{t\geq 0}\big\{\mathcal{E}^{u}_{\psi}(t)-K_\psi t-\mathcal{C}_\psi\big(\|u_0\|^{2}+\mathcal{H}(u_0)\big)\big\}\geq \rho\Big\}\\&\ \qquad\qquad\qquad\qquad\qquad\qquad\qquad\qquad\hspace{2mm}\leq C_{q}\big(1+\|u_0\|^{q}+\mathcal{H}^{\frac{3q}{4}}(u_0)\big)\rho^{-\frac{q}{2}+1}.
    \end{align*}
\end{lemma}
      \begin{proof}  
        The proof is divided into three steps.\par
        \noindent{\it Step 1: pathwise estimate for $\|\psi(t)u_t\|^2$}.
        By It\^o's formula, we have
        \begin{align}\label{03310406}
            \nonumber\|\psi(t)u_t\|^2=&\hspace{1mm} 2\int_0^t\langle \psi(s)u_s,\partial_s\psi(s) u_s\rangle \dd s+ 2\int_0^t\langle \psi(s)u_s,\psi(s)(i\partial_x^2 u_s -\alpha u_s\\&\nonumber -i|u_s|^2u_s)\rangle \dd s+2\sum_{j=1}^\infty b_j\int_0^t\langle \psi(s)u_s,\psi(s)e_j\rangle\dd \beta^1_j(s)\\&\nonumber +2\sum_{j=1}^\infty b_j\int_0^t\langle \psi(s)u_s,i\psi(s)e_j\rangle\dd \beta^2_j(s)+2\sum_{j=1}^\infty b_j^2\int_0^t \|\psi(s)e_j\|^2 \dd s\\\leq&\nonumber \hspace{1mm}2\int_0^t\|\psi(s)u_s\|\|\partial_s\psi(s)u_s\| \dd s- 2\int_0^t\langle \partial_x (\psi(s)^2 u_s),i\partial_x u_s\rangle \dd s\\&\nonumber-2\alpha\int_0^t\|\psi(s)u_s\|^2\dd s +2\sum_{j=1}^\infty b_j\int_0^t\langle \psi(s)u_s,\psi(s)e_j\rangle\dd \beta^1_j(s)\\&\nonumber+2\sum_{j=1}^\infty b_j\int_0^t\langle \psi(s)u_s,i\psi(s)e_j\rangle\dd \beta^2_j(s)+2\sum_{j=1}^\infty b_j^2\int_0^t \|\psi(s)e_j\|^2 \dd s\nonumber\\=:&\hspace{0.5mm}\mathrm{I+II}-2\alpha\int_0^t\|\psi(s)u_s\|^2\dd s+2\sum_{j=1}^\infty b_j\int_0^t\langle \psi(s)u_s,\psi(s)e_j\rangle\dd \beta^1_j(s)\nonumber\\&+2\sum_{j=1}^\infty b_j\int_0^t\langle \psi(s)u_s,i\psi(s)e_j\rangle\dd \beta_j^2(s)+\mathrm{III}.
        \end{align}
        Note that
\begin{align}
    |\partial_x\psi(t,x)|\vee |\partial_t\psi(t,x)|\leq C, \qquad (t,x)\in\mathbb{R}_+\times\mathbb{R}.\label{03310443}
\end{align}The terms $\mathrm{I}, \mathrm{II}$ and $\mathrm{III}$ are then estimated as follows:
        \begin{align}\label{03310412}
           \mathrm{I}&\leq \frac{\alpha}{4}\int_0^t\|\psi(s)u_s\|^2\dd s+ C\int_0^t\|u_s\|^2\dd s,\\
            \nonumber \mathrm{II}&=  2\mathrm{Re}\int_0^t\hspace{-1mm}\int_{\mathbb{R}} i\partial_x (\psi(s)^2)u_s\partial_x \overline{u}_s \dd x\dd s+ 2\mathrm{Re}\int_0^t\hspace{-1mm}\int_{\mathbb{R}}i|\partial_x u_s|^2\psi(s)^2\dd x\dd s\\ &=\nonumber  4\mathrm{Re}\int_0^t\hspace{-1mm}\int_{\mathbb{R}} i\psi(s)\partial_x \psi(s)u_s\partial_x \overline{u}_s \dd x\dd s\\&\nonumber\leq C\int_0^t\|\psi(s)u_s\|\|\partial_xu_s\|\dd s\\&\leq  \frac{\alpha}{4}\int_0^t\|\psi(s)u_s\|^2\dd s+C\int_0^t\|\partial_x u_s\|^2\dd s,\label{03310412-1}\\
            \mathrm{III}&\leq 2t\sum_{j=1}^\infty b_j^2\|\varphi e_j\|^2,\label{03310412-2}
        \end{align}
        where the last inequality follows from the fact that
        \begin{align}\label{05200044}
        \psi(t,x)\leq \varphi(x),    \qquad (t,x)\in\mathbb{R}_+\times \mathbb{R}.
        \end{align}
    Combining \eqref{03310412}--\eqref{03310412-2} with \eqref{03310406}, we obtain
    \begin{align*}
        \|\psi(t)u_t\|^2+\frac{3\alpha}{2}\int_0^t\|\psi(s)u_s\|^2\dd s&\leq C\int_0^t\|u_s\|^2\dd s+C\int_0^t\|\partial_x u_s\|^2\dd s\\&\quad  +2t\sum_{j=1}^\infty b_j^2\|\varphi e_j\|^2\\&\quad+2\sum_{j=1}^\infty b_j\int_0^t\langle \psi(s)u_s,\psi(s)e_j\rangle\dd \beta^1_j(s)\\&\quad +2\sum_{j=1}^\infty b_j\int_0^t\langle \psi(s)u_s,i\psi(s)e_j\rangle\dd \beta^2_j(s).
    \end{align*}
 \noindent{\it Step 2: proof of $\mathrm{(i)}$}.    Again, we only consider the case $p\geq2$. We apply It\^o's formula to $\|\psi(t)u_t\|^{2p}$, obtaining that
    \begin{align}\label{02230145}
        &\nonumber\quad \|\psi(t)u_t\|^{2p}+\frac{3\alpha p}{2}\int_0^t\|\psi(s)u_s\|^{2p}\dd s\leq Cp\int_0^t\|\psi(s)u_s\|^{2p-2}\|u_s\|^{2}\dd s\\&\ \nonumber+Cp\int_0^t\|\psi(s)u_s\|^{2p-2}\|\partial_x u_s\|^2\dd s+2p\sum_{j=1}^\infty b_j^2\|\varphi e_j\|^2\int_0^t\|\psi(s)u_s\|^{2p-2}\dd s\\&\ \nonumber+2p\sum_{j=1}^\infty b_j\int_0^t\|\psi(s)u_s\|^{2p-2}\big(\langle \psi(s)u_s,\psi(s)e_j\rangle\dd \beta^1_j(s)\\&\nonumber\ +\langle \psi(s)u_s,i\psi(s)e_j\rangle\dd \beta^2_j(s)\big)+4p(p-1)\sum_{j=1}^\infty b_j^2\|\varphi e_j\|^2\int_0^t\|\psi(s)u_s\|^{2p-2}\dd s\\&\nonumber\leq\frac{\alpha p}{2}\int_0^t\|\psi(s)u_s\|^{2p}\dd s+C_p\int_0^t(\|u_s\|^{2p}+\|\partial_xu_s\|^{2p})\dd s+C_pt\big(\sum_{j=1}^\infty b_j^2\|\varphi e_j\|^2\big)^{p}\\&\ \nonumber+2p\sum_{j=1}^\infty b_j\int_0^t\|\psi(s)u_s\|^{2p-2}\big(\langle \psi(s)u_s,\psi(s)e_j\rangle\dd \beta^1_j(s)\\&\ \qquad\qquad\qquad\qquad\qquad\qquad+\langle \psi(s)u_s,i\psi(s)e_j\rangle\dd \beta^2_j(s)\big).
    \end{align}
    Then applying Lemmas~\ref{02280441}(ii) and~\ref{02280436}(ii), we have
    \begin{align*}
    \mathbb{E}\big[\|\psi(t)u_t\|^{2p} +{\alpha p}\int_0^t\|\psi(s)u_s\|^{2p}\dd s\big]&\leq C_p\mathbb{E}\big[\int_0^t \|u_s\|^{2p}\dd s\big]\\&\quad+C_p\mathbb{E}\big[\int_0^t\|\partial_xu_s\|^{2p}\dd s\big]+C_pt\\&\leq  C_p\big(\|u_0\|^{2p}+\mathcal{H}^{p}(u_0)\big)+C_pt.
    \end{align*} 
        This proves (i).
        
         \noindent{\it Step 3: proof of $\mathrm{(ii)}$}. It suffices to consider the case $q>2$. Inequality \eqref{02230145} implies that, for some $\mathcal{C}_\psi>0$,
      \begin{align}\label{02230225}
    \nonumber\|\psi(t)u_t\|^{2} +{\alpha }\int_0^t\|\psi(s)u_s\|^{2}\dd s\nonumber&\leq  \alpha \mathcal{C}_\psi\int_0^t \|u_s\|^{2}\dd s+{\alpha} \mathcal{C}_\psi\int_0^t\mathcal{H}(u_s)\dd s\\&\quad+\mathcal{C}_\psi t+M^{\psi}_t,
    \end{align}
    where
    \begin{align*}
M^{\psi}_t:=2\sum_{j=1}^\infty b_j\int_0^t\big(\langle \psi(s)u_s,\psi(s)e_j\rangle\dd \beta^1_j(s)+\langle \psi(s)u_s,i\psi(s)e_j\rangle\dd \beta^2_j(s)\big).\nonumber
    \end{align*}
    We denote $K_\psi:=\mathcal{C}_\psi(K_{0,1}+K_{1,1}+1)+1/3$. Then \eqref{02230225} implies that 
    \begin{align*}
       & \Big\{\sup_{t\geq 0}\big\{\mathcal{E}_{\psi}^{u}(t)-K_\psi t-\mathcal{C}_\psi\big(\|u_0\|^{2}+\mathcal{H}(u_0)\big)\big\}\geq\rho\Big\}\\\subseteq& \ \Big\{\sup_{t\geq 0}\big\{M^{\psi}_t-\frac{t}{3}\big\}\geq\frac{\rho}{3}\Big\}\\&{\scalebox{1.3}{$\cup$}}\Big\{\sup_{t\geq 0}\{\alpha \int_0^t \|u_s\|^{2}\dd s-{K_{0,1}}t-\|u_0\|^{2}\}\geq\frac{\rho}{3\mathcal{C}_\psi}\Big\}\\&{\scalebox{1.3}{$\cup$}}\Big\{\sup_{t\geq 0}\big\{{\alpha}\int_0^t\mathcal{H}(u_s)\dd s-K_{1,1}t-\mathcal{H}(u_0)\big\}\geq\frac{\rho}{3\mathcal{C}_\psi}\Big\}.
    \end{align*}
    Since $\rho\geq 3\mathcal{C}_\psi $, by Lemmas \ref{03050812-1} and \ref{03050812},
\begin{align}\label{03081006}
    \nonumber \mathbb{P}\Big\{&\sup_{t\geq 0}\big\{\mathcal{E}^{u}_{\psi}(t)-K_\psi t-\mathcal{C}_\psi\big(\|u_0\|^{2}+\mathcal{H}(u_0)\big)\big\}\geq\rho\Big\}\\\leq&\ \nonumber\mathbb{P}\Big\{\sup_{t\geq 0}\big\{M_t^{\psi}-\frac{t }{3}\big\}\geq\frac{\rho}{3}\Big\}\\&\nonumber+\mathbb{P}\Big\{\sup_{t\geq 0}\big\{\alpha \int_0^t \|u_s\|^{2}\dd s-{K_{0,1}}t-\|u_0\|^{2}\big\}\geq\frac{\rho}{3\mathcal{C}_\psi}\Big\}\\&\nonumber+\mathbb{P}\Big\{\sup_{t\geq 0}\big\{{\alpha}\int_0^t\mathcal{H}(u_s)\dd s-{K_{1,1}}t-\mathcal{H}(u_0)\big\}\geq\frac{\rho}{3\mathcal{C}_\psi}\Big\}\\\leq&\nonumber  \sum_{m=0}^\infty\mathbb{P}\Big\{\sup_{t\in[m,m+1]}\big\{M_t^{\psi}-\frac{t}{3}\big\}\geq\frac{\rho}{3}\Big\}\\&\hspace{1mm} +C_{q}\big(1+\|u_0\|^{q}+\mathcal{H}^{\frac{3q}{4}}(u_0)\big)\rho^{-\frac{q}{2}+1}.
    \end{align}
To estimate the series in this inequality, note that for any $q>2$,
    \begin{align*}
    \nonumber \mathbb{E}\big[\langle M_t^\psi\rangle_{m+1}^{\frac{q}{2}}\big]\leq& \ C_q\mathbb{E}\big[\big(\sum_{j=1}^\infty b_j^2\int_0^{m+1}\hspace{-1mm}\|\psi(s)u_s\|^2\|\psi(s)e_j\|^2\dd s\big)^{\frac{q}{2}}\big]\\\leq&\  C_q\big(\sum_{j=1}^\infty b_j^2\|\varphi e_j\|^2\big)^{\frac{q}{2}}\mathbb{E}\big[\big(\int_0^{m+1}\|\psi(s)u_s\|^2\dd s\big)^{\frac{q}{2}}\big].
 \end{align*}
Thus by applying Chebyshev's, H\"older's, and the Burkholder--Davis--Gundy inequalities, together with
\eqref{05200044} and Lemma~\ref{03080957}(i), we have
\begin{align}\label{05191052}
   &\nonumber\sum_{m=0}^\infty\mathbb{P}\Big\{\sup_{t\in[m,m+1]}\big\{M_t^{\psi}-\frac{t}{3}\big\}\geq\frac{\rho}{3}\Big\}\leq\nonumber\sum_{m=0}^\infty\mathbb{P}\Big\{\sup_{t\in[m,m+1]}M_t^{\psi}\geq\frac{\rho}{3}+\frac{m}{3}\Big\} \\\nonumber\leq&\ C_{q}\sum_{m=0}^\infty\frac{1}{(\rho+m)^q}(m+1)^{\frac{q}{2}-1}\mathbb{E}\big[\int_0^{m+1}\|\psi(s)u_s\|^{q}\dd s\big]\\\nonumber\leq&\ C_{q}\big(\sum_{j=1}^\infty b_j^2\|\varphi e_j\|^2\big)^{\frac{q}{2}}\sum_{m=0}^\infty\frac{1}{(\rho+m)^q}(m+1)^{\frac{q-2}{2}}\big[m+1+\|u_0\|^{q}+\mathcal{H}^\frac{q}{2}(u_0)\big]\nonumber\\\leq&\  C_{q}\big(1+\|u_0\|^{q}+\mathcal{H}^{\frac{q}{2}}(u_0)\big)\rho^{-\frac{q}{2}+1},
\end{align}
where the last step follows by summing the series as in \eqref{08010427}, using the fact that $\rho\geq 1$. Combining  \eqref{03081006} and \eqref{05191052}, we complete the proof.
\end{proof}
As before, we define
\begin{align}\label{03190353-4}
    \tau_{\psi}^{u}:=\inf\big\{t\geq 0:\mathcal{E}_{\psi}^{u}(t)\geq \rho+Kt+\mathscr{C}\big(\|u_0\|^{2}+\mathcal{H}(u_0)\big)\big\}.
\end{align} 
\begin{corollary}\label{03090723}
    For any $q\geq 1$, $\rho\geq 3\mathcal{C}_\psi\vee 1$, $K\geq K_\psi+1$, $\mathscr{C}\geq \mathcal{C}_\psi$, and  $l\geq 0$, we have
    \begin{align*}
        \mathbb{P}\big(l\leq \tau_{\psi}^{u}<\infty\big)\leq C_{q}\big(1+\|u_0\|^{q}+\mathcal{H}^{\frac{3q}{4}}(u_0)\big){(\rho+l)^{-\frac{q}{2}+1}}.
    \end{align*}
\end{corollary}
\begin{proof}
The fact $K\geq K_\psi+1$ implies that on the event $\{l\leq \tau_{\psi}^{u}<\infty\}$,
\begin{align*}
    \rho+l&\leq \rho+(K-K_\psi)\tau_{\psi}^u+\big(\mathscr{C}-\mathcal{C}_\psi\big)\big(\|u_0\|^2+\mathcal{H}(u_0)\big)\\&=\mathcal{E}^{u}_{\psi}(\tau_{\psi}^{u})-K_\psi\tau_{\psi}^{u}-\mathcal{C}_\psi(\|u_0\|^{2}+\mathcal{H}(u_0))\\&\leq  \sup_{t\geq 0}\big\{\mathcal{E}_{\psi}^u (t)-K_\psi t-\mathcal{C}_\psi(\|u_0\|^{2}+\mathcal{H}(u_0))\big\}.
\end{align*}
 Then Lemma~\ref{03080957} implies
\begin{align*}
     \mathbb{P}\big(l\leq \tau_{\psi}^{u}<\infty\big) & \leq   \mathbb{P}\big(\sup_{t\geq 0}\big\{\mathcal{E}^{u}_{\psi}(t)\!-K_\psi t-\mathcal{C}_\psi (\|u_0\|^{2}\!+\!\mathcal{H}(u_0))\big\}\geq \rho+l\big)\\&\leq  C_{q}\big(1+\|u_0\|^{q}+\mathcal{H}^{\frac{3q}{4}}(u_0)\big){(\rho+l)^{-\frac{q}{2}+1}}.
\end{align*}
\end{proof}
The following stopping time is used throughout
Sections~\ref{S:2}--\ref{S:4}:
\begin{align}\label{04080050}
\tau_1^u:=\min\{\tau^u_{0,6},\tau^u_{0,15},\tau_{1,2}^u,\tau_{1,15}^u,
\tau_{2,2}^u,\tau_{2,3}^u,\tau_{\psi}^u\},
\end{align}
where $\tau_{0,p}^{u}$, $\tau_{1,p}^{u}$, $\tau_{2,p}^u$ and
$\tau_{\psi}^{u}$ are defined in \eqref{03190353-1}, \eqref{03190353-2}, \eqref{03190353-3}, and \eqref{03190353-4},
with $p=6,15$; $p=2,15$; and $p=2,3$, respectively. The constants are fixed~by
\begin{gather}
K=\max\{K_{0,6},K_{0,15},K_{1,2},K_{1,15},K_{2,2},K_{2,3},K_\psi\}+1,
\label{04271005-1}\\
\mathscr{C}=\mathcal{C}_2^\sharp\vee\mathcal{C}_3^\sharp
\vee\mathcal{C}_\psi\vee 1,\label{04271005-2}
\end{gather}
while the parameter $\rho$ is required to satisfy
\begin{align}\label{04271005-3}
\rho\ge\rho_*:=1\vee 3\mathcal{C}_2^\sharp\vee 3\mathcal{C}_3^\sharp
\vee 3\mathcal{C}_\psi,
\end{align}
so that the hypotheses of Corollaries~\ref{03090533}, \ref{03090531},
\ref{03090542} and~\ref{03090723} are satisfied. Here $K_{0,6},\ldots,
K_\psi$, $\mathcal{C}_2^\sharp$, $\mathcal{C}_3^\sharp$ and
$\mathcal{C}_\psi$ are the constants of Lemmas~\ref{03050812-1},
\ref{03050812}, \ref{03080704} and~\ref{03080957}.
\begin{lemma}\label{03120055}
   For any $l\geq 0$, $\rho\ge\rho_*$ and $q\geq 1$, we have
    \begin{align*}
      \mathbb{P}\big(l\leq \tau_1^u<\infty\big)\leq C_q\big(1+\|u_0\|^{36q}+\mathcal{H}^{18q}(u_0)+\mathcal{G}^{\frac{7q}{2}}(u_0)\big)\big(\rho+l\big)^{-\frac{q}{2}+1}.
    \end{align*}
\end{lemma}
\begin{proof}
    It follows from \eqref{04080050} that
    \begin{align*}
    \{l\leq \tau_1^u<\infty\}\subseteq &\hspace{4mm}\{l\leq \tau_{0,6}^u<\infty\}\displaystyle \scalebox{1.3}{$\cup$}\{l\leq \tau_{0,15}^u<\infty\}\\&\hspace{1mm}\displaystyle \scalebox{1.3}{$\cup$} \{l\leq \tau_{1,2}^u<\infty\}\scalebox{1.3}{$\cup$} \{l\leq \tau_{1,15}^u<\infty\}\\&\hspace{1mm}\displaystyle \scalebox{1.3}{$\cup$}\{l\leq \tau_{2,2}^u<\infty\}\displaystyle \scalebox{1.3}{$\cup$}\{l\leq \tau_{2,3}^u<\infty\}\scalebox{1.3}{$\cup$} \{l\leq \tau_\psi^{u}<\infty\}.
    \end{align*}
      Corollaries~\ref{03090533}, \ref{03090531},~\ref{03090542}, and \ref{03090723} imply that
    \begin{align*}
        \mathbb{P} \big(l\leq\hspace{1mm} \tau_1^u<\infty\big)&\leq  \hspace{1mm} \mathbb{P}\big(l\leq \tau_{0,6}^u<\infty\big)+\mathbb{P}\big(l\leq \tau_{0,15}^u<\infty\big)\\&\quad+ \mathbb{P}\big(l\leq \tau_{1,2}^u<\infty\big)+ \mathbb{P}\big(l\leq \tau_{1,15}^u<\infty\big)\\&\quad+\mathbb{P}\big(l\leq \tau_{2,2}^u<\infty\big)+\mathbb{P}\big(l\leq \tau_{2,3}^u<\infty\big)+\mathbb{P}\big(l\leq \tau_\psi^u<\infty\big)\\&\leq C_q\big(1+\|u_0\|^{36q}+\mathcal{H}^{18q}(u_0)+\mathcal{G}^{\frac{7q}{2}}(u_0)\big)(\rho+l)^{-\frac{q}{2}+1}.
    \end{align*}
\end{proof}

\subsection{Exponential recurrence in $H^2$}

The main result of this subsection is the exponential recurrence property in $H^2$, established in Lemma~\ref{03160229}. The following lemma provides the irreducibility needed for the recurrence argument.
\begin{lemma}\label{03041000} For any $R,r>0$, there exist $\varepsilon_{R,r}>0$ and $T(R,r)>0$ such that, for any $t\geq T(R,r)$ and $u_0\in B_{H^2}(R)$, we have
\begin{align*}
    P_{t}\big(u_0,B_{H^2}(r)\big)\geq \varepsilon_{R,r}.
\end{align*}
\end{lemma}
To prove this lemma, we first control the shifted process $\mathcal U_t:=u_t-bW_t$ on a
small-noise event. To this end, let us consider the functional
\begin{align}\label{04091818}
\widetilde{\mathcal{H}}(u):=\mathcal{H}(u)+\frac{1}{2}\|u\|^2,
\end{align}and the event 
$$
    \Omega_{\delta,T}:=\{\sup_{t \in [0,T]}\|bW_t\|_{H^4}\leq\delta\}
$$for any $\delta, T>0$.

\begin{lemma}\label{03050241}
For any $R,r>0$, there exist $T_1:=T_1(R,r)>0$ and $\delta_1:=\delta_1(R,r)\in (0,1)$ such that, for any $T> T_1$, $\delta\in(0,\delta_1]$, and $u_0\in H^1$ with $\widetilde{\mathcal{H}}(u_0)\leq R$, 
we have
\begin{align*}
\sup_{t\in[0,T]}\widetilde{\mathcal{H}}(\mathcal{U}_t)\leq 2R\quad\text{ and }\quad\sup_{t\in[T_1,T]}\widetilde{\mathcal{H}}(\mathcal{U}_t)\leq r
\end{align*}
on the event $\Omega_{\delta,T}$.
\end{lemma}
\begin{proof}
Note that $\mathcal{U}$ satisfies   
    \begin{align}\label{03310343-0}
        \partial_t\mathcal{U}+\alpha \mathcal{U}-i\partial_x^2\mathcal{U}+i|\mathcal{U}|^2\mathcal{U}=P(bW,\mathcal{U}),\quad \mathcal{U}_0=u_0,
    \end{align}where
    \begin{align}\label{03310343}
        \nonumber P(bW,\mathcal{U})_t&=i\partial_x^2 (bW_t)-\alpha bW_t-i[2|bW_t|^2\mathcal{U}_t+2bW_t|\mathcal{U}_t|^2+\overline{bW}_t\hspace{0.5mm}\mathcal{U}_t^2\\&\quad+|bW_t|^2bW_t+(bW_t)^2\overline{\mathcal{U}}_t].
    \end{align}
Multiplying \eqref{03310343-0} by $\overline{\mathcal{U}}-\partial_x^2\overline{\mathcal{U}}+|\mathcal{U}|^2\overline{\mathcal{U}}$, taking the real part, and integrating by parts, we deduce
    \begin{align}\label{03090231}
        \nonumber&\widetilde{\mathcal{H}}(\mathcal{U}_t)\hspace{0.5mm}-\hspace{0.5mm} \widetilde{\mathcal{H}}(u_0)+2\alpha\int_0^t\widetilde{\mathcal{H}}(\mathcal{U}_s)\dd s+\frac{\alpha}{2}\int_0^t\|\mathcal{U}_s\|_{L^4}^4\dd s\\\nonumber=&\int_0^t\langle -\partial_x^2 \mathcal{U}_s+|\mathcal{U}_s|^2\mathcal{U}_s+\mathcal{U}_s,P(bW,\mathcal{U})_s\rangle\dd s\\\nonumber=&\int_0^t\langle -\partial_x^2 \mathcal{U}_s+|\mathcal{U}_s|^2\mathcal{U}_s+\mathcal{U}_s,i\partial_x^2(bW_s)-\alpha bW_s-i|bW_s|^2bW_s\rangle\dd s\\&\ \nonumber-\int_0^t\langle -\partial_x^2 \mathcal{U}_s+|\mathcal{U}_s|^2\mathcal{U}_s+\mathcal{U}_s,2i|bW_s|^2\mathcal{U}_s+i(bW_s)^2\overline{\mathcal{U}}_s\rangle\dd s\\&\ \nonumber -\int_0^t\langle -\partial_x^2 \mathcal{U}_s+|\mathcal{U}_s|^2\mathcal{U}_s+\mathcal{U}_s,2ibW_s|\mathcal{U}_s|^2+i\overline{bW}_s\mathcal{U}_s^2\rangle \dd s\\\leq&\ \nonumber C\int_0^t\hspace{-2mm}\big(\hspace{-0.25mm}(\|bW_s\|_{H^3}\hspace{-0.5mm}+\hspace{-0.5mm}\|bW_s\|_{H^1}^3)\| \mathcal{U}_s\|_{H^1}\hspace{-1mm}+\hspace{-1mm}\|\mathcal{U}_s\|_{L^4}^2\|\mathcal{U}_s\|(\|bW_s\|_{H^3}\hspace{-1mm}+\hspace{-1mm}\|bW_s\|_{H^1}^3)\big)\dd s\\\nonumber&\ +C\int_0^t(\|bW_s\|_{H^2}^2\|\mathcal{U}_s\|^2_{H^1}+\|bW_s\|_{H^1}^2\|\mathcal{U}_s\|_{L^4}^4)\dd s\\\nonumber&\ +C\int_0^t(\|bW_s\|_{H^2}\|\mathcal{U}_s\|_{H^1}^3+\|bW_s\|_{H^1}\|\mathcal{U}_s\|_{L^4}^4\|\mathcal{U}_s\|_{H^1})\dd s\\\leq&\ C\int_0^t(1+\|bW_s\|_{H^4}^2)\|bW_s\|_{H^4}\big(1+\widetilde{\mathcal{H}}^{\frac{3}{2}}(\mathcal{U}_s)\big)\dd s.
    \end{align}
For any $T>0$, we define the stopping time
 \begin{align*}
     \tau_R:=\inf\{t\in[0,T]:\widetilde{\mathcal{H}}(\mathcal{U}_t)\geq 3R\}\wedge T,
 \end{align*}
 and fix $\delta\in(0,1)$. Then
  \eqref{03090231} implies
 \begin{align*}
 \widetilde{\mathcal{H}}(\mathcal{U}_t)\leq e^{-2\alpha t}R+\frac{C}{2\alpha}(1+R^\frac{3}{2})\delta
 \end{align*} for $t \leq \tau_R$ on $\Omega_{\delta,T}$.
 Choosing $\widetilde{\delta}_1:=\widetilde{\delta}_1(R)\in (0,1)$ so small that  
 \begin{align*}
     \frac{C}{2\alpha}(1+R^\frac{3}{2})\widetilde{\delta}_1\leq R,
 \end{align*}
we obtain that for any $\delta\leq \widetilde{\delta}_1$,
 \begin{align*}
     \widetilde{\mathcal{H}}(\mathcal{U}_t)\leq 2R \quad \text{ \ on $\Omega_{\delta,T}$ for   $t\leq \tau_R$}.
 \end{align*}
 This implies $\tau_R=T$ and 
 \begin{align}\label{06190422}
     \sup_{t\in[0,T]}\widetilde{\mathcal{H}}(\mathcal{U}_t)\leq 2R\quad \text{ \  on $\Omega_{\delta,T}$ for   $\delta\leq\widetilde{\delta}_1$}.
 \end{align}
  Moreover, given $r>0$, choosing $\delta_1:=\delta_1(R,r)\leq \widetilde{\delta}_1(R)$ so small  that
 \begin{align*}
     \frac{C}{2\alpha}(1+R^\frac{3}{2})\delta_1\leq \frac{r}{2}
 \end{align*}and setting $T_1({R,r}):=\frac{1}{2\alpha}\ln\big(\frac{2R}{r}\big)\vee 1$, we deduce that, for any $T > T_1$ and~$\delta\leq\delta_1$,
 \begin{align*}
     \widetilde{\mathcal{H}}(\mathcal{U}_t)\leq r \quad\text{ on $\Omega_{\delta,T}$ for   $t\in[T_1,T]$}.
 \end{align*}
  Combining this with \eqref{06190422}, we obtain the required result.
\end{proof}
\begin{proof}[Proof of Lemma~\ref{03041000}] 
The proof is divided into three steps. In Step~1 we derive a pathwise
estimate for $\mathcal{G}(\mathcal{U}_t)$ along the shifted process
$\mathcal{U}=u-bW$. Combined with Lemma~\ref{03050241}, it is used in
Step~2 to prove the following claim: for any $R,r>0$, there are
$T_2:=T_2(R,r)>0$ and $\varepsilon_1:=\varepsilon_1(R,r)>0$ such~that
\begin{align}\label{03050424-2}
    P_{T_2}\big(u_0, B_{H^2}(r)\big)\geq \varepsilon_1
    \qquad\text{for } u_0\in B_{H^2}(R).
\end{align}
Finally, in the last step, we complete the proof of the lemma.\par 
\noindent{\it Step 1: A pathwise estimate of $\mathcal{G}(\mathcal{U}_t)$}.
     We have (cf. \eqref{03050118})
\begin{align}\label{03310314}
     \nonumber\mathcal{G}(\mathcal{U}_t)-\mathcal{G}(u_0)+\frac{3\alpha}{2}\int_0^t\mathcal{G}(\mathcal{U}_s)\dd s & \leq  C\int_0^t\|\partial_x\mathcal{U}_s\|^8\|\mathcal{U}_s\|^2\dd s\\&\quad+\int_0^t\langle D\mathcal{G}(\mathcal{U}_s),P(bW,\mathcal{U})_s\rangle \dd s.
\end{align}
By \eqref{03090253}, the second term on the right-hand side of \eqref{03310314}
can be written~as
\begin{align}\label{03310315}
 \nonumber \int_0^t\langle &D\mathcal{G}(\mathcal{U}_s),P(bW,\mathcal{U})_s\rangle \dd s=\int_0^t\langle \partial_x^2\mathcal{U}_s,\partial_x^2 P(bW,\mathcal{U})_s\rangle \dd s\\&\quad+\nonumber  \int_0^t\langle|\mathcal{U}_s|^2\partial_x P(bW,\mathcal{U})_s,\partial_x\mathcal{U}_s\rangle\dd s+\int_0^t\langle \mathcal{U}_s|\partial_x\mathcal{U}_s|^2, P(bW,\mathcal{U})_s\rangle\dd s\\&\quad+\nonumber  \int_0^t\langle \mathcal{U}_s^2,\partial_x\mathcal{U}_s\partial_x(P(bW,\mathcal{U})_s)\rangle\dd s+\int_0^t\langle (\partial_x\mathcal{U}_s)^2,\mathcal{U}_sP(bW,\mathcal{U})_s\rangle\dd s\\&=: \mathrm{I+II+III+IV+V}.
\end{align}
  For $\mathrm{I}$, \eqref{03310343} and H\"older's inequality yield
\begin{align*}
     \mathrm{I}&=\int_0^t\langle\partial_x^2\mathcal{U}_s,i\partial_x^4(bW_s)- \alpha\partial_x^2(bW_s)-i\partial_x^2(|bW_s|^2bW_s)\rangle\dd s\\&\quad-  \int_0^t\langle \partial_x^2 \mathcal{U}_s,\partial_x^2(2i|bW_s|^2\mathcal{U}_s+i(bW_s)^2\overline{\mathcal{U}}_s)\rangle\dd s\\&\quad-\int_0^t\langle \partial_x^2 \mathcal{U}_s,\partial_x^2(2i|\mathcal{U}_s|^2bW_s +i\mathcal{U}_s^2\overline{bW}_s)\rangle\dd s\\&\leq  C\hspace{-0.5mm}\int_0^t\hspace{-0.5mm}\big(\|\partial_x^4(bW_s)\|\hspace{-0.5mm}+\hspace{-0.5mm}\|\partial_x^2 (bW_s)\|+\|bW_s\|^2_{L^\infty}\|\partial_x^2 (bW_s)\|\hspace{-0.5mm}\\&\qquad\qquad+\hspace{-0.5mm}\|bW_s\|_{L^\infty}\|\partial_x(bW_s)\|_{L^4}^2\big)\|\partial_x^2\mathcal{U}_s\|\dd s\hspace{-0.5mm}\\&\quad +\hspace{-0.5mm}C\int_0^t\hspace{-0.5mm}\big(\|bW_s\|_{L^\infty}^2\|\partial_x^2\mathcal{U}_s\|\hspace{-0.2mm}+\hspace{-0.2mm}\|\partial_x(bW_s)\|_{L^\infty}\|bW_s\|_{L^\infty}\|\partial_x\mathcal{U}_s\|\\&\qquad\qquad+\|\partial_x(bW_s)\|_{L^\infty}^2\|\mathcal{U}_s\|+\|\partial^2_x(bW_s)\|_{L^\infty}\|bW_s\|_{L^\infty}\|\mathcal{U}_s\|\big)\|\partial_x^2\mathcal{U}_s\|\dd s\hspace{-0.5mm}\\&\quad+ C\int_0^t\big(\|\partial^2_x(bW_s)\|_{L^\infty}\|\mathcal{U}_s\|_{L^4}^2+\|bW_s\|_{L^\infty}\|\mathcal{U}_s\|_{L^\infty}\|\partial_x^2\mathcal{U}_s\|\\&\qquad\qquad +\|bW_s\|_{L^\infty}\|\partial_x\mathcal{U}_s\|^2_{L^4}+\|\partial_x(bW_s)\|_{L^\infty}\|\mathcal{U}_s\|_{L^\infty}\|\partial_x\mathcal{U}_s\|\big)\|\partial_x^2\mathcal{U}_s\|\dd s.
\end{align*}
Using the Sobolev embedding $H^1\hookrightarrow L^\infty$ and the interpolation inequality
\begin{align*}
    \|\partial_xu\|_{L^4}\leq C\|\partial_xu\|^{\frac{3}{4}}\|\partial_x^2u\|^{\frac{1}{4}},\qquad u\in H^2,
\end{align*}
we obtain
\begin{align*}
     \mathrm{I}&\leq C\int_0^t(\|bW_s\|_{H^4}+\|bW_s\|^3_{H^2})\|\partial_x^2\mathcal{U}_s\|\dd s+C\int_0^t\|bW_s\|_{H^1}^2\|\partial_x^2 \mathcal{U}_s\|^2\dd s\\&\quad +C\int_0^t\|bW_s\|_{H^2}^2\|\partial_x\mathcal{U}_s\|\|\partial_x^2\mathcal{U}_s\|\dd s+C\int_0^t\| bW_s\|_{H^3}^2\|\mathcal{U}_s\|\|\partial_x^2\mathcal{U}_s\|\dd s\\&\quad +C\int_0^t\|bW_s\|_{H^3}\|\mathcal{U}_s\|^2_{L^4}\|\partial_x^2\mathcal{U}_s\|\dd s+C\int_0^t\|bW_s\|_{H^1}\|\mathcal{U}_s\|_{H^1}\|\partial_x^2\mathcal{U}_s\|^2\dd s\\&\quad +C\int_0^t\|bW_s\|_{H^1}\|\partial_x\mathcal{U}_s\|^{\frac{3}{2}}\|\partial_x^2\mathcal{U}_s\|^{\frac{3}{2}}\dd s+C\int_0^t\|bW_s\|_{H^2}\|\mathcal{U}_s\|_{H^1}^{2}\|\partial_x^2\mathcal{U}_s\|\dd s.
\end{align*}
Arguing similarly for $\mathrm{II}$, we derive
\begin{align*}
    \hspace{1mm}\mathrm{II}&=\int_0^t\langle |\mathcal{U}_s|^2\partial_x\mathcal{U}_s,i\partial_x^3(bW_s)\rangle\dd s-\int_0^t\langle |\mathcal{U}_s|^2\partial_x\mathcal{U}_s,\ \alpha\partial_x(bW_s)\\&\qquad\hspace{-1mm}+i\partial_x(|bW_s|^2bW_s)\rangle\dd s-\int_0^t\langle |\mathcal{U}_s|^2\partial_x\mathcal{U}_s,\partial_x(2i|bW_s|^2\mathcal{U}_s+i(bW_s)^2\overline{\mathcal{U}}_s)\rangle\dd s\\&\qquad\hspace{-1mm}-\int_0^t\langle |\mathcal{U}_s|^2\partial_x\mathcal{U}_s,\partial_x(2ibW_s|\mathcal{U}_s|^2+  i\overline{bW}_s\mathcal{U}_s^2)\rangle\dd s\\&\leq C\int_0^t(\|bW_s\|_{H^4}+\|bW_s\|_{H^2}^3)\|\mathcal{U}_s\|_{L^4}^2\|\partial_x\mathcal{U}_s\|\dd s\hspace{-0.2mm}+\hspace{-0.2mm}C\int_0^t\|bW_s\|_{H^2}^2\|\mathcal{U}_s\|_{H^1}^4\dd s\\&\qquad +C\int_0^t\|bW_s\|_{H^2}\|\mathcal{U}_s\|_{H^1}^5\dd s.
\end{align*}
Proceeding as above, we find
\begin{align*}
\hspace{0.5mm}\mathrm{III}&=\int_0^t\langle |\partial_x\mathcal{U}_s|^2\mathcal{U}_s,i\partial_x^2(bW_s)\rangle\dd s-\int_0^t\langle |\partial_x\mathcal{U}_s|^2\mathcal{U}_s,\ \alpha bW_s+i|bW_s|^2bW_s\rangle\dd s\\&\quad -\int_0^t\langle |\partial_x\mathcal{U}_s|^2\mathcal{U}_s,2i|bW_s|^2\mathcal{U}_s+i(bW_s)^2\overline{\mathcal{U}}_s\rangle\dd s\\&\quad -\int_0^t\langle |\partial_x\mathcal{U}_s|^2\mathcal{U}_s,2ibW_s |\mathcal{U}_s|^2+i\overline{bW}_s\mathcal{U}_s^2\rangle\dd s\\&\leq\hspace{0.5mm}C\int_0^t(\|bW_s\|_{H^3}+\|bW_s\|_{H^1}^3)\|\mathcal{U}_s\|_{H^1}^3\dd s+C\int_0^t\|bW_s\|_{H^1}^2\|\mathcal{U}_s\|_{H^1}^4\dd s\\&\quad +C\int_0^t\|bW_s\|_{H^1}\|\mathcal{U}_s\|_{H^1}^5\dd s.
\end{align*}
The estimates for $\mathrm{IV}$ and $\mathrm{V}$ are the same as those for $\mathrm{II}$ and $\mathrm{III}$. Adding $\mathrm{I}$--$\mathrm{V}$ and applying Young's inequality, we get
\begin{align*}
    &\nonumber \mathrm{I+II+III+IV+V}\leq \frac{\alpha}{2}\int_0^t\mathcal{G}(\mathcal{U}_s)\dd s+\hspace{0.5mm}C\int_0^t(\|bW_s\|_{H^4}^2+\|bW_s\|_{H^2}^6)\dd s\\\nonumber&+C\int_0^t(\|bW_s\|^2_{H^1}+\|bW_s\|_{H^1}\|\mathcal{U}_s\|_{H^1})\|\partial_x^2\mathcal{U}_s\|^2\dd s\\&\nonumber +C\int_0^t(\|bW_s\|_{H^2}^4+\|bW_s\|_{H^3}^4)\|\mathcal{U}_s\|_{H^1}^2\dd s+C\int_0^t\|bW_s\|_{H^3}^2\|\mathcal{U}_s\|_{L^4}^4\dd s\\&\nonumber +C\int_0^t\|bW_s\|_{H^1}^4\|\mathcal{U}_s\|_{H^1}^6\dd s+C\int_0^t\|bW_s\|_{H^2}^2\|\mathcal{U}_s\|_{H^1}^4\dd s\\&+C\int_0^t(\|bW_s\|_{H^4}+\|bW_s\|_{H^2}^3)\|\mathcal{U}_s\|_{L^4}^2\|\partial_x\mathcal{U}_s\|\dd s+C\int_0^t\|bW_s\|_{H^2}\|\mathcal{U}_s\|_{H^1}^5\dd s\\&+C\int_0^t(\|bW_s\|_{H^3}+\|bW_s\|_{H^1}^3)\|\mathcal{U}_s\|_{H^1}^3\dd s.
\end{align*}
Combining this with \eqref{03310314}, \eqref{03310315}, and using \eqref{04091818}, we obtain 
\begin{align}\label{03310507}
    \nonumber \mathcal{G}(&\mathcal{U}_t)-\mathcal{G}(u_0)+\alpha\int_0^t\mathcal{G}(\mathcal{U}_s)\dd s\\\leq&\nonumber \  C\int_0^t\widetilde{\mathcal{H}}^5(\mathcal{U}_s)\dd s+ C\int_0^t(\|bW_s\|_{H^4}^2+\|bW_s\|_{H^2}^6)\dd s\\&+\nonumber C\int_0^t\big(\|bW_s\|_{H^1}^2\hspace{-1mm}+\|bW_s\|_{H^1}\widetilde{\mathcal{H}}^{\frac{1}{2}}(\mathcal{U}_s)\big)\mathcal{G}(\mathcal{U}_s)\dd s\hspace{-0.5mm}\\&+C\int_0^t\hspace{-1mm}\|bW_s\|_{H^4}(\|bW_s\|_{H^4}^3+1)\widetilde{\mathcal{H}}(\mathcal{U}_s)\nonumber \dd s\\&+C\int_0^t\|bW_s\|_{H^3}(\|bW_s\|_{H^3}^2+1)\widetilde{\mathcal{H}}^{\frac{3}{2}}(\mathcal{U}_s)\dd s\nonumber+ C\int_0^t\|bW_s\|_{H^2}^2\widetilde{\mathcal{H}}^2(\mathcal{U}_s)\dd s\\&+\hspace{-0.5mm}C\int_0^t\|bW_s\|_{H^2}\widetilde{\mathcal{H}}^\frac{5}{2}(\mathcal{U}_s)\dd s+C\int_0^t\|bW_s\|_{H^1}^4\widetilde{\mathcal{H}}^3(\mathcal{U}_s)\dd s.
\end{align}
\noindent{\it Step 2: Proof of \eqref{03050424-2}}.
Let $R_1,r_1>0$, and let $T_1:=T_1(R_1,r_1)$ and $\delta_1:=\delta_1(R_1,r_1)\in (0,1)$ be
as in Lemma~\ref{03050241}. That lemma shows that, for any $T>T_1$, any $u_0$
with $\widetilde{\mathcal{H}}(u_0)\leq R_1$, and any $\delta\leq\delta_1$,
\begin{align}\label{03050439}
    \sup_{t\in[0,T]}\widetilde{\mathcal{H}}(\mathcal{U}_t)\leq 2R_1
    \quad\text{ and }\quad
    \sup_{t\in[T_1,T]}\widetilde{\mathcal{H}}(\mathcal{U}_t)\leq r_1
    \text{\quad on }\Omega_{\delta,T}.
\end{align}
From \eqref{03310507} and \eqref{03050439}, and using Gronwall's lemma together with the fact that $\delta_1\leq 1$, we deduce that for any $T> T_1$ and $\delta\leq \delta_1$,
\begin{align*}
     &\nonumber\mathcal{G}(\mathcal{U}_T)-e^{-\frac{\alpha}{2} T}\mathcal{G}(u_0)\\&\leq\  \nonumber C\int_0^{T_1} e^{-\frac{\alpha}{2}(T-s)}(2R_1)^5\dd s+C\int_{T_1}^{T} e^{-\frac{\alpha}{2}(T-s)}(r_1)^5\dd s\\&\nonumber\quad+ C\int_0^T\hspace{-1mm}e^{-\frac{\alpha}{2}(T-s)}\delta^2\dd s+ \int_0^Te^{-\frac{\alpha}{2}(T-s)}\big(C\delta^2+C(2R_1)^{\frac{1}{2}}\delta-\frac{\alpha}{2}\big)\mathcal{G}(\mathcal{U}_s)\dd s\\&\quad+  C\int_0^T\hspace{-1mm}e^{-\frac{\alpha}{2}(T-s)}\delta\big(2R_1+(2R_1)^{3}\big)\dd s\text{\qquad  on $\Omega_{\delta,T}$}.
\end{align*}
We choose $\widetilde{\delta}_2:=\widetilde{\delta}_2(R_1,r_1)\in\big(0,\delta_1(R_1,r_1)\big]$ so small that
\begin{align*}
    C(\widetilde{\delta}_2)^2+C(2R_1)^{\frac{1}{2}}\widetilde{\delta}_2\leq \frac{\alpha}{2}.
\end{align*}
Then the last integrand is non-positive, so that, for $\delta\in (0,\widetilde{\delta}_2]$ and $T>T_1$, the following holds on $\Omega_{\delta,T}$:
\begin{align}\label{03050414}
    \nonumber\hspace{-3mm}\mathcal{G}(\mathcal{U}_{T})\leq&\ \nonumber e^{-\frac{\alpha}{2}T}\mathcal{G}(u_0)+\frac{C\delta}{\alpha}\big(\delta+2R_1+(2R_1)^{3}\big)+C\int_0^{T_1}e^{-\frac{\alpha}{2}(T-T_1)}(2R_1)^5 \dd s\\&+C\int^{T}_{T_1}e^{-\frac{\alpha}{2}(T-s)}(r_1)^5 \dd s\nonumber \\\leq&\ e^{-\frac{\alpha}{2}{T}}\mathcal{G}(u_0)+C\delta\big(1\hspace{-0.5mm}+\hspace{-0.5mm}(R_1)^{3})\hspace{-0.5mm}+\hspace{-0.5mm}C(r_1)^5\hspace{-0.5mm}+\hspace{-0.5mm}C(R_1)^5 T_1e^{-\frac{\alpha}{2}(T-T_1)}.
    \end{align}
We now choose $R_1=R_1(R)$ so large that
\begin{align}\label{06190757}
\widetilde{\mathcal{H}}(v)\vee\mathcal{G}(v)\leq R_1,
\qquad  v\in B_{H^2}(R).
\end{align}
To estimate \eqref{03050414}, we first choose $r_1\leq \frac{r^2}{16}$ so small that  
\begin{align*}
    C(r_1)^5\leq \frac{r^2}{64}.
\end{align*}
Then we choose $\delta_2:=\delta_2(R,r)\leq \widetilde{\delta}_2(R_1,r_1)\wedge\frac{r}{2}$ so small that 
\begin{align*}
     C\delta_2\big(1+(R_1)^{3}\big)\leq \frac{r^2}{64}.
\end{align*}
Finally, choosing $T=T_2:=T_2(R,r)$ so large that
\begin{align*}
    e^{-\frac{\alpha}{2}{T_2}}R_1\leq \frac{r^2}{64}\quad \text{ and } \quad C (R_1)^5 T_1e^{-\frac{\alpha}{2}(T_2-T_1)}\leq \frac{r^2}{64}, 
\end{align*}
we obtain from \eqref{03050414} that if $\|u_0\|_{H^2}\leq R$, then 
\begin{align}\label{06190545-1}
    \mathcal{G}(\mathcal{U}_{T_2})\leq \frac{r^2}{16} \quad\text{on } \Omega_{\delta_2,T_2}.
\end{align}
On the other hand, since $r_1\leq\frac{r^2}{16}$,  \eqref{03050439} implies
\begin{align}\label{06190545-2}
    \widetilde{\mathcal{H}}(\mathcal{U}_{T_{2}})\leq \frac{r^2}{16} \quad\text{on } \Omega_{\delta_2,T_2}.
\end{align}
Combining \eqref{06190545-1}, \eqref{06190545-2} with \eqref{E:HGnorms}, we obtain that $\mathcal{U}_{T_2}\in B_{H^2}(\frac{r}{2})$. Since $\delta_2\leq \frac{r}{2}$, we have  
\begin{align*}
    \|u_{T_2}\|_{H^2}\leq \|\mathcal{U}_{T_2}\|_{H^2}+\|bW_{T_2}\|_{H^2}\leq  r \quad\text{on } \Omega_{\delta_2,T_2}.
\end{align*}
Since $bW$ is a Gaussian process, we have $\varepsilon_1({R,r}):=\mathbb{P}(\Omega_{\delta_2,T_2})>0$, which proves \eqref{03050424-2}.\par 
\noindent{\it Step 3: Conclusion}. Let $R_2=R_2(R)>0$ be a constant to be chosen, and let $T_3:=T_2(R_2,r)$ and $\delta_3:=\delta_2(R_2,r)$ be the corresponding parameters from Step~2, so that $\mathbb{P}(\Omega_{\delta_3,T_3})=\varepsilon_1(R_2,r)$. For $t\geq T_3$, we obtain
\begin{align*}
    \mathbb{P}_{u_0}\big( u_t\in B_{H^2}(r)\big)&\geq\mathbb{E}\big[\mathbf{1}_{u_{t-T_3}\in B_{H^2}(R_2)}{P}_{T_3}(u_{t-T_3}, B_{H^2}(r))\big]\\&\geq \mathbb{P}(\Omega_{\delta_3,T_3})\mathbb{P}_{u_0}\big(u_{t-T_3}\in B_{H^2}(R_2)\big)\\&\geq \mathbb{P}(\Omega_{\delta_3,T_3})\mathbb{P}_{u_0}\Big(\mathcal{G}(u_{t-T_3})+\widetilde{\mathcal{H}}(u_{t-T_3})\leq \frac{R_2^2}{2}\Big).
\end{align*}
By Lemmas \ref{02280441}(i), \ref{02280436}(i), and \ref{02231236}(i), we have
\begin{align*}
\nonumber \mathbb{P}_{u_0}\Big(\mathcal{G}&(u_{t-T_3})+\widetilde{\mathcal{H}}(u_{t-T_3})> \frac{R_2^2}{2}\Big)\\\leq&\nonumber \ \mathbb{P}_{u_0}\big(\mathcal{G}(u_{t-T_3})>\frac{R_2^2}{4}\big)+\mathbb{P}_{u_0}\big(\widetilde{\mathcal{H}}(u_{t-T_3})>\frac{R_2^2}{4}\big)\\\leq&\nonumber \ \frac{4}{R_2^2}\mathbb{E}_{u_0}\big[\mathcal{G}(u_{t-T_3})\big]+\frac{4}{R_2^2}\mathbb{E}_{u_0}\big[\widetilde{\mathcal{H}}(u_{t-T_3})\big]\\\leq&\nonumber \ \frac{4}{R_2^2}\big(e^{-\alpha (t-T_3)}\mathcal{G}(u_0)+Ce^{-\alpha (t-T_3)}\widetilde{\mathcal{H}}^5(u_0)+C\big)\\\nonumber &+\frac{4}{R_2^2}\big(e^{-{\frac{\alpha}{2} }(t-T_3)}\widetilde{\mathcal{H}}(u_0)+C\big)\\\leq& \ \frac{C}{R_2^2}\big(R_1^5+1\big)
\end{align*}
for $R_1(R)$ defined in \eqref{06190757}. By choosing $R_2=R_2(R)>0$ so large that
\begin{align*}
    \frac{C}{R_2^2}\big(R_1^5+1\big)\leq \frac{1}{2},
\end{align*}
  we obtain
\begin{align*}
    \mathbb{P}_{u_0}\Big(\mathcal{G}(u_{t-T_3})+\widetilde{\mathcal{H}}(u_{t-T_3})\leq \frac{R_2^2}{2}\Big)\geq\frac{1}{2},
\end{align*}
which in combination with \eqref{03050424-2} implies that
\begin{align*}
    \mathbb{P}_{u_0}\big( u_t\in B_{H^2}(r)\big)\geq \frac{1}{2}\varepsilon_1({R_2,r}).
\end{align*}
Thus, choosing $T(R,r):=T_3$ and $\varepsilon_{R,r}:=\frac{1}{2}\varepsilon_1({R_2(R),r})$, we complete the proof. 
\end{proof}
Next, we establish exponential recurrence in $H^2$. To this end, we introduce
the functional $\widetilde{F}:H^2\rightarrow \mathbb{R}$ defined by
\begin{align}\label{03120847}
    \widetilde{F}(u):=\|u\|^{10}+\mathcal{H}^5(u)+\mathcal{G}(u).
\end{align}
Let $u$ and $u^\prime$ be the solutions to \eqref{equation-1} with initial
conditions $u_0$ and $u_0^\prime$, driven by independent noises $bW$ and
$bW^\prime$, respectively. For any $d>0$ and time step $T>0$ (to be specified
in the next lemma), we define the stopping time
\begin{align}\label{08030032}
    \vartheta_d:=\min\big\{n\in \mathbb{N}:\|u_{nT}\|_{H^2}\vee\|u^\prime_{nT}\|_{H^2}\leq d\big\}.
\end{align}
\begin{lemma}\label{03160229} For any $d>0$, there exist $c_d, T_{d}, C_{d}>0$ such that, for any
$T\geq T_{d}$ and $u_0,u_0^\prime\in H^2$,
\begin{align}\label{03310635}
    \mathbb{E}\big[\exp(c_d\vartheta_d)\big]\leq C_{d}\big(1+\widetilde{F}(u_0)+\widetilde{F}(u_0^\prime)\big).
\end{align}
\end{lemma}
\begin{proof}
We deduce \eqref{03310635} from Proposition~5.1 in \cite{Mar14} (see also Proposition~3.3 in \cite{Shi08}). By that proposition, it suffices to find constants $C_*, R_*, T_d>0$ and $a\in(0,1)$ such that
\begin{itemize}
       \item[(i)] $\mathbb{E}\big[1+\widetilde{F}(u_{T_d})\big]\leq a\big(1+\widetilde{F}(u_0)\big)$ for any $\|u_0\|_{H^2}\geq R_*$,
    \item[(ii)] $\mathbb{E}\big[\widetilde{F}(u_t)\big]\leq C_*$ for any $\|u_0\|_{H^2}\leq R_*$ and $t\geq 0$,
    \item[(iii)] $\inf\limits_{\|u_0\|_{H^2}\vee \|u_0^\prime\|_{H^2}\leq R_*}\mathbb{P}\big(\|u_{T_d}\|_{H^2}\vee\|u_{T_d}^\prime\|_{H^2}\leq d\big)>0$.
\end{itemize}
By Lemmas~\ref{02280441}(i), \ref{02280436}(i), and \ref{02231236}(i), there
exists $\mathcal{C}_5>1$ such that 
\begin{align}\label{03310616}
    \mathbb{E}_{u_0}\big[1+\widetilde{F}(u_t)\big]\leq e^{-\alpha t}\widetilde{F}(u_0)+\mathcal{C}_5e^{-\alpha t}(\|u_0\|^{10}+\mathcal{H}^5(u_0))+\mathcal{C}_5
\end{align}  for any $u_0\in H^2$ and $t\ge0$.

We first verify (i). By \eqref{E:HGnorms}, we have $\|u\|_{H^2}^2\leq C(1+\widetilde{F}(u))$ for any $u\in H^2$, so that we can choose $R_*>0$ so large that
\begin{align}\label{03310616-w}
    1+\widetilde{F}(u_0)\geq 4\mathcal{C}_5
    \qquad\text{whenever } \|u_0\|_{H^2}\geq R_*,
\end{align}
and then $T_*>0$ so large that $\mathcal{C}_5e^{-\alpha T_*}\leq 1/8$. 
Combining \eqref{03310616}, \eqref{03310616-w}, and the choice of $T_*$ yields,
for $t\geq T_*$ and $\|u_0\|_{H^2}\geq R_*$,
\begin{align*}
    \mathbb{E}_{u_0}\big[1+\widetilde{F}(u_t)\big]\leq \frac{1}{2}\big(1+\widetilde{F}(u_0)\big), 
\end{align*} which is (i) with $a=1/2$.
On the other hand, for $\|u_0\|_{H^2}\le R_*$,
estimate \eqref{03310616} gives a constant $C_*=C_*(R_*)>0$ with
\begin{align*}
    \mathbb{E}_{u_0}\big[\widetilde{F}(u_t)\big]\leq C_*, \qquad t\ge0,
\end{align*}
which is (ii). Finally, to verify (iii), Lemma~\ref{03041000} provides
$\varepsilon_{R_*,d}>0$ and $T(R_*,d)>0$ such that
$$
P_t\big(w,B_{H^2}(d)\big)\ge\varepsilon_{R_*,d},
\qquad  t\geq T(R_*,d),\ w\in B_{H^2}(R_*).
$$
Since $u$ and $u^\prime$ are driven by independent noises, it follows that 
\begin{align*}
\mathbb{P}\big(\|u_{t}\|_{H^2}\vee\|u^\prime_{t}\|_{H^2}\leq d\big)\geq(\varepsilon_{R_*,d})^2>0
\end{align*}
for any $t\geq T(R_*,d)$ and $u_0,u_0^\prime\in B_{H^2}(R_*)$. Conditions
(i)--(iii) therefore hold with $T_d:=T_*\vee T(R_*,d)$, which completes the proof.
\end{proof}

\subsection{Existence and regularity of stationary measures}\label{S:A4}

In this subsection, we construct a stationary measure concentrated
on~$H^2$ and show that every stationary measure in $\mathcal{P}(H^2)$ has
finite $H^2$-moments of all orders.
\begin{lemma}\label{05050844}
Equation~\eqref{equation-1} admits a stationary measure
$\nu\in\mathcal{P}(H^2)$.\end{lemma}
\begin{proof}
By Theorem 3.4 in \cite{EKZ17}, there is a subsequence $\{n_k\}_{k\in\mathbb{N}}\subseteq\mathbb{N}$ with $n_k\to+\infty$ as
$k\rightarrow\infty$ and a stationary measure $\nu\in \mathcal{P}(H^1)$ such
that
\begin{align}\label{04091400}
    Q_{n_k}(\cdot):=\frac{1}{n_k}\int_0^{n_k}P_t(0,\cdot)\dd t\rightarrow \nu
\quad\text{in } \mathcal{P}(H^1).
\end{align}
  Let $u$ be the solution to \eqref{equation-1} issued from $u_0=0$. Then, by Lemmas~\ref{02280441}(i), \ref{02280436}(i), and~\ref{02231236}(i),
there exists $C>0$ such that, for any $k\geq1$,
\begin{align*}
    \int_{H^1}\|u\|_{H^2}^2 Q_{n_k}(\dd u)\leq
    \frac{2}{n_k}\int_0^{n_k}\mathbb{E}\big[\|u_t\|^2+\mathcal{H}(u_t)
    +\mathcal{G}(u_t)\big]\dd t\leq C.
\end{align*}
Since $\|\cdot\|_{H^2}:H^1\rightarrow [0,\infty]$ is lower semicontinuous, we deduce from \eqref{04091400} that
\begin{align*}
    \int_{H^1}\|u\|_{H^2}^2 \nu(\dd u)\leq \liminf_{k\rightarrow\infty}\int_{H^1}\| u\|_{H^2}^2 Q_{n_k}(\dd u)\leq C,
\end{align*}
which implies that $\nu(H^2)=1$.
\end{proof}
 \begin{lemma}\label{05050848}
    Let $\mu\in\mathcal{P}(H^2)$ be a stationary measure for
    \eqref{equation-1}. Then, for any $p\geq 1$,
    \begin{align*}
        \int_{H^2}\|u\|_{H^2}^{2p}\mu(\dd u)<\infty.
    \end{align*}
\end{lemma}
\begin{proof}
By Proposition~4.1 in~\cite{BFZ23}, which holds without the largeness assumption on $\alpha$, every stationary measure satisfies, for
all $m\geq 1$,
\begin{align}\label{05050735}
    \int_{H^2}\big(\|u\|^{2m}+\mathcal{H}^m(u)\big)\mu(\dd u)\leq C_m.
\end{align}
In view of the inequality 
$$
\|u\|_{H^2}^{2p}\le C_p(\|u\|^{2p}+\mathcal H^p(u)+\mathcal G^p(u))
$$
and \eqref{05050735}, it suffices to show
$\int_{H^2}\mathcal{G}^p(u)\,\mu(\dd u)<\infty$.

For any $R>0$, we consider the truncated functional
$\mathcal{G}_R:H^2\rightarrow \mathbb{R}$ defined by $\mathcal{G}_R(u):=\mathcal{G}(u)\wedge R$. Since $\mu\in\mathcal{P}(H^2)$ is stationary and $u_0\in H^2$ implies
$u_t\in H^2$ almost surely, we have
\begin{align}\label{E:stat-Borel}
    \int_{H^2}\mathbb{E}_u\big[g(u_t)\big]\mu(\dd u)=\int_{H^2}g(u)\mu(\dd u)
\end{align}
for every bounded Borel function $g:H^2\to\mathbb{R}$ and every $t\ge0$; we
shall use this with $g=\mathcal{G}_R^p$.

 By Lemma~\ref{02231236}(i), for any $\varrho>0$, we have 
\begin{align}\label{08240007}
   \int_0^T\int_{B_{H^2}(\varrho)}&\mathbb{E}_{u}
   \big[\mathcal{G}^{p}_R(u_t)\big]\mu(\dd u)\dd t
   \leq \int_0^T\int_{B_{H^2}(\varrho)}
   \mathbb{E}_{u}\big[\mathcal{G}^p(u_t)\big]\mu(\dd u)\dd t\nonumber\\
   &\leq C_p\int_{B_{H^2}(\varrho)}\big(\|u\|^{10p}+\mathcal{H}^{5p}(u)
   +\mathcal{G}^p(u)\big)\mu(\dd u)+C_pT\nonumber\\
   &\leq C_p(\varrho^{20p}+1)+C_pT.
\end{align}
   On the other hand, since $\mathcal G_R\le R$,
   \begin{align}\label{05050651-1} \int_0^T \int_{\big(B_{H^2}(\varrho)\big)^c}\mathbb{E}_{u}
   \big[\mathcal{G}^{p}_R(u_t)\big]\mu(\dd u)\dd t\leq T R^p\mu\big(\big(B_{H^2}(\varrho)\big)^c\big).
   \end{align}
   Combining \eqref{E:stat-Borel}--\eqref{05050651-1}, we~obtain 
   \begin{align*}
    \int_{H^2}\mathcal{G}_R^p(u)\mu(\dd u)
    &=\frac{1}{T}\int_0^T\!\!\int_{H^2}\mathbb{E}_u
      \big[\mathcal{G}_R^p(u_t)\big]\mu(\dd u)\dd t\\
    &\leq \frac{C_p(\varrho^{20p}+1)}{T}+C_p
      +R^p\mu\big(\big(B_{H^2}(\varrho)\big)^c\big).
\end{align*}
For fixed $R$, letting $T\to\infty$ and then $\varrho\to\infty$ gives
$\int_{H^2}\mathcal{G}_R^p\,\dd\mu\leq C_p$ with $C_p$ independent of $R$;
monotone convergence as $R\to\infty$ then yields
 $$
       \int_{H^2}\mathcal{G}^p(u)\mu(\dd u)\leq C_p.
 $$
 Combining this with \eqref{05050735}, we obtain the required estimate.
\end{proof}

\subsection{Well-posedness of the auxiliary process}\label{WPA}

In this subsection, we establish the following lemma.
\begin{lemma}\label{WPA2}
    For any $N \ge1$ and $u_0,u_0^\prime\in H^1$, problem \eqref{auxiliary} is globally well-posed in~$H^1$. Moreover, for any $t>0$ and $p\geq 2$,
    \begin{align*}
        \mathbb{E}\big[\sup_{s\in[0,t]}\widetilde{\mathcal{H}}^p(v_s)\big]<\infty.
    \end{align*}
\end{lemma}
\begin{proof}
The local well-posedness of \eqref{auxiliary} in $H^1$ follows from a 
standard fixed-point argument (cf.~Section~3.1 of \cite{DD03}). To 
establish global well-posedness, it suffices to derive an a priori bound 
on the $H^1$-norm of $v$, which, together with a standard stopping time argument (cf.~Proposition~3.4 of~\cite{DD03}), rules out finite-time blow-up. Since the stopping time argument is standard, we restrict ourselves to deriving the required a priori estimate.

To this end, we apply It\^o's formula to $\frac{1}{2}\|v_t\|^2$ and use that
$\mathsf{P}_N$ is an orthogonal projection commuting with the multiplication
by~$i$:
\begin{align*}
    \frac{1}{2}\|v_t&\|^2-\frac{1}{2}\|u_0^\prime\|^2
    +\alpha \int_0^t\|v_s\|^2\dd s\\
    =&\ \int_0^t\langle i\mathsf{P}_N(|v_s|^2v_s),v_s\rangle\dd s
    -\int_0^t\langle i\mathsf{P}_N\partial_x^2v_s, v_s\rangle \dd s
    -\int_0^t\langle i\mathsf{P}_N(|u_s|^2u_s),v_s\rangle\dd s\\
    &+\int_0^t\langle i\partial_x^2u_s,\mathsf{P}_N v_s\rangle\dd s
    +\int_0^t\langle v_s, b\dd W_s\rangle+\sum_{j=1}^\infty b_j^2\|e_j\|^2t\\\leq&\int_0^t\|v_s\|_{L^4}^3\|\mathsf{P}_Nv_s\|_{L^4}\dd s+\int_0^t\|\partial_x^2\mathsf{P}_N v_s\|\| v_s\|\dd s+\int_0^t\|u_s\|_{L^4}^3\|\mathsf{P}_Nv_s\|_{L^4}\dd s\\&\ +\int_0^t\|u_s\|\|\partial_x^2\mathsf{P}_Nv_s\|\dd s+\int_0^t\langle v_s,b\dd W_s\rangle+Ct.
\end{align*}
Since $e_j\in H^4$, using integration by parts, we deduce that the operators $\partial_x^k\mathsf{P}_N$ and $\mathsf{P}_N\partial_x^k$, $k=0,1,2,$ are well defined on $L^q$ for $q\in[1,2]$ and satisfy
\begin{align}
&\|\partial_x^k\mathsf{P}_N\xi\|_{L^p}\leq C_{N,p,k}\|\mathsf{P}_N\xi\|,
\qquad p\in[2,\infty],\label{07080243}\\
&\|\mathsf{P}_N\partial_x^k\xi\|_{L^p}\leq C_{N,p,q,k}\|\xi\|_{L^q},
\qquad p\in[2,\infty],\ q\in[1,2].\label{07080243-1}
\end{align}
Inequality \eqref{07080243}, together with Young's inequality, implies that
\begin{align}\label{03061034}
     \frac{1}{2}\|v_t\|^2&-\frac{1}{2}\|u_0^\prime\|^2
     +\alpha \int_0^t\|v_s\|^2\dd s\nonumber\\
     &\leq \frac{\alpha}{4}\int_0^t\|v_s\|^4_{L^4}\dd s
     +\frac{\alpha}{4}\int_0^t\|v_s\|^2\dd s
     +C_N\int_0^t\big(\|\mathsf{P}_Nv_s\|^2+\|\mathsf{P}_Nv_s\|^4\big)\dd s\nonumber\\
     &\quad+C_N\int_0^t\big(\|u_s\|^2+\|u_s\|_{L^4}^4\big)\dd s
      +\int_0^t\langle v_s, b\dd W_s\rangle+Ct.
\end{align}
To estimate $\mathcal{H}(v)$, we first apply It\^o's formula to $\frac{1}{2}\|\partial_xv_t\|^2$:
\begin{align*}
    \nonumber \frac{1}{2}\|\partial_xv_t\|^2&-\frac{1}{2}\|\partial_xu_0^\prime\|^2+\alpha\int_0^t\|\partial_x v_s\|^2\dd s\\=&\nonumber\int_0^t\langle -i\mathsf{Q}_N(\partial_x^2v_s),\partial_x^2v_s\rangle\dd s+\int_0^t\langle i\mathsf{Q}_N(|v_s|^2v_s),\partial_x^2v_s\rangle\dd s\\&\ \nonumber+\int_0^t\langle i\mathsf{P}_N(|u_s|^2u_s),\partial_x^2v_s\rangle\dd s-\int_0^t\langle i\mathsf{P}_N\partial_x^2u_s,\partial_x^2v_s\rangle\dd s\\&\ \nonumber-\int_0^t\langle \partial_x^2v_s, b\dd W_s\rangle+{t}\sum_{j=1}^\infty b_j^2\|\partial_xe_j\|^2\\\leq&\nonumber\int_0^t\langle i\mathsf{Q}_N(|v_s|^2v_s),\partial_x^2v_s\rangle\dd s -\int_0^t\langle \partial_x^2v_s, b\dd W_s\rangle+t\sum_{j=1}^\infty b_j^2\|\partial_xe_j\|^2\nonumber\\&\ +C\int_0^t\|\mathsf{P}_N\partial^2_xv_s\|_{L^4}\|u_s\|_{L^4}^3\dd s+C\int_0^t\|\mathsf{P}_N\partial_x^2u_s\|\|\mathsf{P}_N\partial_x^2v_s\|\dd s,
\end{align*}
where
\begin{align*}
    \int_0^t\langle \partial_x^2 v_s,b\dd W_s\rangle
    :=-\mathrm{Re}\sum_{j=1}^\infty b_j\int_0^t\int_{\mathbb{R}}
    \partial_x\overline{v}_s\,\partial_x e_j\dd x
    \big(\dd \beta_j^1(s)+i\dd \beta_j^2(s)\big).
\end{align*}By \eqref{07080243-1} and Young's inequality, 
\begin{align} \label{03061034-1}   
    \nonumber\frac{1}{2}\|\partial_xv_t\|^2&-\frac{1}{2}\|\partial_xu_0^\prime\|^2+\alpha\int_0^t\|\partial_x v_s\|^2\dd s\\\nonumber&\leq \int_0^t\langle i\mathsf{Q}_N(|v_s|^2v_s),\partial_x^2v_s\rangle\dd s+C_N\int_0^t(\|u_s\|_{L^4}^{6}+\|u_s\|^2)\dd s\\&\quad+\frac{\alpha}{4}\int_0^t\|v_s\|^2\dd s -\int_0^t\langle \partial_x^2v_s,b\dd W_s\rangle+Ct.
\end{align}
Similarly, applying It\^o's formula to $\frac{1}{4}\|v_t\|_{L^4}^4$ and
using H\"older's inequality, we obtain  
\begin{align*}
    \nonumber\frac{1}{4}&\|v_t\|_{L^4}^4-\frac{1}{4}\|u_0^\prime\|_{L^4}^4+\alpha\int_0^t\|v_s\|_{L^4}^4\dd s\\=&\int_0^t\langle |v_s|^2v_s,i\mathsf{Q}_N\partial_{x}^2 {v}_s-i\mathsf{Q}_N(|v_s|^2v_s)\rangle\dd s\nonumber+2\sum_{j=1}^\infty b_j^2\int_0^t\|v_s e_j\|^2\dd s\\\nonumber&-\int_0^t\langle |v_s|^2v_s,i\mathsf{P}_N(|u_s|^2u_s)-i\mathsf{P}_N(\partial_x^2u_s)\rangle\dd s+\int_0^t\langle |v_s|^2v_s,b\dd W_s\rangle\\\leq&\nonumber\int_0^t\langle |v_s|^2v_s,i\mathsf{Q}_N\partial_x^2v_s\rangle\dd s+C_N\int_0^t\|v_s\|_{L^4}^3\|\mathsf{P}_N(|u_s|^2u_s)\|_{L^4}\dd s\\&\nonumber+C_N\int_0^t\|v_s\|_{L^4}^3\|\mathsf{P}_N\partial^2_xu_s\|_{L^4}\dd s+\int_0^t\langle |v_s|^2v_s,b\dd W_s\rangle+C\int_0^t\|v_s\|_{L^4}^2\dd s.
\end{align*}
Therefore,  invoking again~\eqref{07080243-1} and applying Young's inequality,
\begin{align}\label{03070134}
    \frac{1}{4}\|v_t\|_{L^4}^4-&\frac{1}{4}\|u_0^\prime\|_{L^4}^4+\alpha\int_0^t\|v_s\|_{L^4}^4\dd s\nonumber\\\leq&\nonumber\int_0^t\langle |v_s|^2v_s,i\mathsf{Q}_N\partial_x^2v_s\rangle\dd s+C_N\int_0^t\|v_s\|_{L^4}^3(\|u_s\|_{L^4}^3+\|u_s\|)\dd s\\&+\int_0^t\langle |v_s|^2v_s,b\dd W_s\rangle+ C\int_0^t\|v_s\|_{L^4}^2\dd s\nonumber\\\leq&-\int_0^t\langle i\mathsf{Q}_N(|v_s|^2v_s),\partial_x^2v_s\rangle\dd s\hspace{-0.5mm}+\hspace{-0.5mm}\frac{\alpha}{4}\int_0^t\|v_s\|_{L^4}^4\dd s\hspace{-0.5mm}+\hspace{-0.5mm}\int_0^t\langle |v_s|^2v_s,b\dd W_s\rangle\nonumber\\& +C_N\int_0^t(\|u_s\|_{L^{4}}^{12}+\|u_s\|^4)\dd s+C_Nt.
\end{align}
Recalling $\widetilde{\mathcal{H}}$ from~\eqref{04091818},
and using \eqref{03061034}--\eqref{03070134}, we get
\begin{align}\label{07102337}
    \nonumber\widetilde{\mathcal{H}}(v_t)-\widetilde{\mathcal{H}}(u_0^\prime)&+  \alpha\int_0^t\widetilde{\mathcal{H}}(v_s)\dd s\leq  C_N\int_0^t\big(\|u_s\|_{L^4}^{12}+\|u_s\|^4\big)\dd s\\&+C_N\int_0^t(\|\mathsf{P}_Nv_s\|^2+\|\mathsf{P}_Nv_s\|^4)\dd s+Ct +M^*_t,
\end{align}
where the martingale term is given by 
\begin{align*}
    M_t^*:=\int_0^t\langle |v_s|^2v_s+v_s-\partial_x^2v_s,\,b\dd W_s\rangle.
\end{align*}The quadratic variation of $M^*$ is  estimated by (cf.~\eqref{03130248}) 
\begin{align*}
    \langle M^*\rangle_t\leq&\hspace{1mm} C\sum_{j=1}^\infty b_j^2\int_0^{t}\hspace{-1mm}\big(\|\partial_x v_s\|^2\|\partial_x e_j\|^2+\|v_s\|^2\|e_j\|^2+\|v_s\|_{L^4}^6\|e_j\|_{L^4}^2\big)\dd s\\\leq&\hspace{1mm} C\int_0^t\big(\|v_s\|_{H^1}^2+\|v_s\|_{L^4}^6\big)\dd s.
\end{align*}
Then, for $p \ge 2$, taking the supremum over $s\in[0,t]$ in
\eqref{07102337}, raising to the $p$-th power, taking expectations, and
applying \eqref{03190251}, the Burkholder--Davis--Gundy inequality, and
Lemmas~\ref{02280441}(i) and \ref{02280436}(i), we~obtain
\begin{align*}
    \mathbb{E}\big[\sup_{s\in[0,t]}&\widetilde{\mathcal{H}}^p(v_s)\big]\leq C_p\widetilde{\mathcal{H}}^p(u_0^\prime)+C_{N,p}\mathbb{E}\big[\big(\int_0^t(\|\mathsf{P}_Nv_s\|^2+\|\mathsf{P}_Nv_s\|_{L^4}^4)\dd s\big)^p\big]\\&\quad+C_{N,p}t^p+C\mathbb{E}\big[\big(\int_0^t\|v_s\|_{L^4}^6\dd s\big)^{\frac{p}{2}}\big]+C\mathbb{E}\big[\big(\int_0^t\|v_s\|_{H^1}^2\dd s\big)^{\frac{p}{2}}\big]\\&\quad +C_{N,p}\mathbb{E}\big[\big(\int_0^t\big(\widetilde{\mathcal{H}}^3(u_s)+1\big)\dd s\big)^p\big]\\&\leq C_p\widetilde{\mathcal{H}}^p(u_0^\prime)+C_{N,p,t}\widetilde{\mathcal{H}}^{3p}(u_0)+C_{N,p}t^p\\&\quad +C_{N,p}t^{p-1}\mathbb{E}\big[\int_0^t\big(\|\mathsf{P}_Nu_s\|^{2p}+\|\mathsf{P}_Nu_s\|_{L^4}^{4p}\big)\dd s\big]\\&\quad +C_{N,p}t^{p-1}\int_0^t \big[\|\mathsf{P}_N(u_0-u_0^\prime)\|^{2p}+\|\mathsf{P}_N(u_0-u_0^\prime)\|_{L^4}^{4p}\big]e^{-2\alpha ps}\dd s\\&\quad +C_pt^{\frac{p}{2}}\mathbb{E}\big[\big(\sup_{s\in[0,t]}\widetilde{\mathcal{H}}^{\frac{p}{4}}(v_s)+1\big)\sup_{s\in[0,t]}\widetilde{\mathcal{H}}^{\frac{p}{2}}(v_s)\big].
\end{align*}
Since $L^4$ and $L^2$ norms are equivalent on the finite-dimensional range
of~$\mathsf{P}_N$, by applying Lemmas~\ref{02280441}(ii) and \ref{02280436}(ii), and using Young's inequality, we obtain that 
\begin{align*}
    \mathbb{E}\big[\sup_{s\in[0,t]}\widetilde{\mathcal{H}}^p(v_s)\big]&\leq C_p\widetilde{\mathcal{H}}^p(u_0^\prime)+C_{N,p,t}+C_{N,p,t}\widetilde{\mathcal{H}}^{3p}(u_0)\\&\quad+C_{N,p,t}(1+\|\mathsf{P}_N(u_0-u_0^\prime)\|^{4p}).
\end{align*}
This estimate implies that \eqref{auxiliary} is globally well-posed. 
\end{proof}

We close this subsection with the $H^2$-regularity of the process $v$.
\begin{lemma}\label{07101918}
If $u_0, u_0^\prime\in H^2$, then almost surely $v\in C(\mathbb{R}_+;H^2)$
and   $\|\psi(t)v_t\|$ is bounded on bounded time
intervals. In particular, the stopping times $\tau^v_{0,p}$,
$\tau^v_{1,p}$, $\tau^v_{2,p}$, and $\tau^v_{\psi}$ are well defined.
\end{lemma}

\begin{proof}[Sketch of the proof]
Since $e_j\in H^4$, the bound \eqref{07080243} also holds with
$\|\cdot\|_{L^p}$ replaced by $\|\cdot\|_{H^2}$, and the feedback term in
\eqref{auxiliary} therefore satisfies, almost surely, 
\begin{align*}
    \big\|\mathsf{P}_N\big(|v_t|^2v_t-|u_t|^2u_t&
    -\partial_x^2(v_t-u_t)\big)\big\|_{H^2}\\&\leq C_N\big\|\mathsf{P}_N\big(|v_t|^2v_t-|u_t|^2u_t
    -\partial_x^2(v_t-u_t)\big)\big\|
    \\&\leq C_N\big(\|u_t\|_{H^1}^3+\|v_t\|_{H^1}^3
    +\|u_t\|+\|v_t\|\big),
\end{align*}
where the last inequality follows from \eqref{07080243-1}. \par By Lemmas~\ref{WPA2}, \ref{02280441}(ii), and~\ref{02280436}(ii), equation
\eqref{auxiliary} is thus \eqref{equation-1} perturbed by a force belonging
to $L^p\big(\Omega,L^\infty([0,T];H^2)\big)$ for any $p\geq1$ and $T>0$, so
that the $H^2$-regularity of the initial condition propagates by standard
arguments (cf.~\cite{DD03}); in particular $v\in C(\mathbb{R}_+;H^2)$. Finally, since $\psi(t,x)\leq t$ for each $x\in\mathbb{R}$ and Lipschitz in $t$
uniformly in $x$ by \eqref{03310443}, the process $t\mapsto\psi(t)v_t$ is
continuous in~$L^2$. 
\end{proof}

\subsection{A parameterized maximal coupling}\label{S:maxcoup}

The coupling construction in Section~\ref{S:4} requires a pair of driving noises, depending measurably on the initial data, such that the corresponding solutions form a maximal coupling. Such a construction is provided by the following lemma. In the special case $f_1=f_2$, it appears as Exercise~1.2.30 in~\cite{KS12}. The construction given on page~296 of that book extends directly to the present setting of two different maps, and we therefore omit the proof.

\begin{lemma}\label{L:max-coupling}
Let $X,$ $Y,$ and $Z$ be Polish spaces, let $\mu\in\mathcal{P}(X)$, and let
$f_1,f_2:Z\times X\to Y$ be measurable. Set
$\nu_i(z,\hspace{0.5mm}\cdot\hspace{0.5mm}):=f_i(z,\hspace{0.5mm}\cdot\hspace{0.5mm})_*\mu,$ $i=1,2$. Then there exist a probability space
$(\Omega,\mathcal{F},\mathbb{P})$ and measurable maps
$\xi_1,\xi_2:Z\times\Omega\to X$ such that, for every $z\in Z$,
\begin{itemize}
\item[(i)] $\mathcal{D}\big(\xi_i(z,\hspace{0.5mm}\cdot\hspace{0.5mm})\big)=\mu$ for $i=1,2$;
\item[(ii)] the pair
\[
\big(f_1(z,\xi_1(z,\hspace{0.5mm}\cdot\hspace{0.5mm})),\ f_2(z,\xi_2(z,\hspace{0.5mm}\cdot\hspace{0.5mm}))\big)
\]
is a maximal coupling of $\big(\nu_1(z,\hspace{0.5mm}\cdot\hspace{0.5mm}),\nu_2(z,\hspace{0.5mm}\cdot\hspace{0.5mm})\big)$, in the sense
of Definition~1.2.21 in~\cite{KS12}.
\end{itemize}
\end{lemma}

  \subsection*{AI use disclosure}
The authors used Anthropic's Claude models to check this paper for typos and
potential errors in the proofs. All mathematical arguments and proofs were
developed and written by the authors, who take full responsibility for the
content of the paper. 
 
  \subsection*{Acknowledgments}
This work is partially supported by NSFC Grant No.~12571156 and the NYU Shanghai
Boost Fund. The authors thank Meng Zhao for valuable discussions.

\addtocontents{toc}{\protect\setcounter{tocdepth}{1}} 
\addcontentsline{toc}{section}{References}
\bibliographystyle{alpha-author}
\bibliography{References}

@article{DO05,
     TITLE = {Ergodicity for a weakly damped stochastic non-linear
              {S}chr\"odinger equation},
  AUTHOR = {Debussche, A.  and Odasso, C.},
  JOURNAL = {J. Evol. Equ.},
  Year = {2005},
  VOLUME = {5},
  NUMBER = {3},
     PAGES = {317--356},
}

@article{DD03,
  TITLE = {The stochastic nonlinear {Schr\"odinger} equation in {$H^1$}},
  AUTHOR = {de Bouard, A. and Debussche, A.},
  JOURNAL = {Stochastic Anal. Appl.},
  Year = {2003},
  VOLUME = {21},
  NUMBER = {1},
  PAGES = {97--126}
}

@article{NZ24,
  TITLE = {Exponential mixing for the white-forced complex {Ginzburg--Landau} equation in the whole space},
  AUTHOR = {Nersesyan, V.  and  Zhao, M.},
  JOURNAL = {SIAM J. Math. Anal.},
  Year = {2024},
  VOLUME = {56},
  NUMBER = {3},
  PAGES = {3646--3678}
}

@article{NZ24b,
  TITLE = {Polynomial mixing for the white-forced {Navier--Stokes} system in the whole space},
  AUTHOR = {Nersesyan, V. and Zhao, M.},
  JOURNAL = {arXiv:2410.15727},
  Year = {2024},
}

@article{EKZ17,
  TITLE = {Existence of invariant measures for the stochastic damped {Schr\"odinger} equation},
  AUTHOR = {Ekren, I. and Kukavica, I. and  Ziane, M.},
  JOURNAL = {Stoch. Partial Differ. Equ. Anal. Comput.},
  Year = {2017},
  VOLUME = {5},
  NUMBER = {3},
  PAGES = {343--367}
}

@book{KS14,
  author    = {Karatzas, I. and Shreve, S. E.},
  title     = {Brownian Motion and Stochastic Calculus},
  series    = {Grad. Texts in Math.},
  volume    = {113},
  edition   = {2},
  publisher = {Springer-Verlag},
  address   = {New York},
  year      = {1991},
}

@article{CXZZ25,
  TITLE = {Exponential mixing for the randomly forced {NLS} equation},
  AUTHOR = {Chen, Y. and Xiang, S. and Zhang, Z. and Zhao, J.},
  JOURNAL = {arXiv:2506.10318},
  Year = {2025}
}

@article{CXZZ26,
  TITLE = {Exponential mixing for nonlinear {Schr\"odinger} equations perturbed by bounded degenerate noise},
  AUTHOR = {Chen, Y. and Xiang, S. and Zhang, Z.},
  JOURNAL = {arXiv:2604.05911},
  Year = {2026}
}

@article{CX25,
  TITLE = {{Donsker--Varadhan} large deviation principle for locally damped and randomly forced {NLS} equations},
  AUTHOR = {Chen, Y. and Xiang, S.},
  JOURNAL = {To appear in Annales Henri Poincaré},
  Year = {2026}
}

@article{G24,
  TITLE = {Polynomial mixing for the white-forced wave equation on the whole line},
  AUTHOR = {Gao, P.},
  JOURNAL = {arXiv:2412.13230},
  Year = {2024}
}

@article{Ha02,
  TITLE = {Exponential mixing properties of stochastic {PDEs} through asymptotic coupling},
  AUTHOR = {Hairer, M.},
  JOURNAL = {Probab. Theory Related Fields},
  Year = {2002},
    VOLUME = {124},
  NUMBER = {3},
       PAGES = {345--380},

}

@article{BFZ23,
  author  = {Brze\'zniak, Z. and Ferrario, B. and Zanella, M.},
  title   = {Ergodic results for the stochastic nonlinear {S}chr\"odinger equation with large damping},
  journal = {J. Evol. Equ.},
  volume  = {23},
  number  = {1},
  pages   = {Paper No. 19, 31},
  year    = {2023},
}

@article{NS25,
  TITLE = {Polynomial mixing for the stochastic {S}chr\"odinger equation with large damping in the whole space},
  AUTHOR = {Nguyen, H. and  Seong, K.},
  JOURNAL = {arXiv:2512.24599},
  Year = {2025},
}

@article{CHW17,
     AUTHOR = {Chen, C.  and Hong, J.  and Wang, X.},
     TITLE = {Approximation of invariant measure for damped stochastic
              nonlinear {S}chr\"odinger equation via an ergodic numerical
              scheme},
   JOURNAL = {Potential Anal.},
     VOLUME = {46},
      YEAR = {2017},
    NUMBER = {2},
     PAGES = {323--367},
 }

@article{BCK14,
  TITLE = {Space-time stationary solutions for the {B}urgers equation},
  AUTHOR = {Bakhtin, Y. and Cator, E. and Khanin, K.},
JOURNAL = {J. Amer. Math. Soc.},
  
    VOLUME = {27},
      YEAR = {2014},
    NUMBER = {1},
     PAGES = {193--238},
  
}

@book{KS12,
  TITLE = {Mathematics of two-dimensional turbulence},
  AUTHOR = {S. Kuksin and A. Shirikyan},
  publisher = {Cambridge University Press},
  Year = {2012},
 
}

@incollection{Shi08,
  author    = {Shirikyan, A.},
  title     = {Exponential mixing for randomly forced partial differential equations: method of coupling},
  booktitle = {Instability in Models Connected with Fluid Flows. {II}},
  series    = {Int. Math. Ser. (N.~Y.)},
  volume    = {7},
  pages     = {155--188},
  publisher = {Springer},
  address   = {New York},
  year      = {2008},
}

@book{HW19,
  author    = {Hong, J. and Wang, X.},
  title     = {Invariant measures for stochastic nonlinear {S}chr\"odinger equations},
  series    = {Lecture Notes in Math.},
  volume    = {2251},
  publisher = {Springer},
  address   = {Singapore},
  year      = {2019},
}

@article{GL26,
  TITLE = {Polynomial mixing for stochastic viscous conservation law equation on the whole line},
  AUTHOR = {Gao, P. and Li, L.},
  JOURNAL = {arXiv:2608.29824},
  Year = {2026},
}

@article{Bak13,
  author  = {Bakhtin, Y.},
  title   = {The {B}urgers equation with {P}oisson random forcing},
  journal = {Ann. Probab.},
  volume  = {41},
  number  = {4},
  pages   = {2961--2989},
  year    = {2013}
}

@article{BaL19,
  author  = {Bakhtin, Y. and Li, L.},
  title   = {Thermodynamic limit for directed polymers and stationary solutions of the {B}urgers equation},
  journal = {Comm. Pure Appl. Math.},
  volume  = {72},
  number  = {3},
  pages   = {536--619},
  year    = {2019}
}

@article{DGR21,
  author  = {Dunlap, A. and Graham, C. and Ryzhik, L.},
  title   = {Stationary solutions to the stochastic {B}urgers equation on the line},
  journal = {Comm. Math. Phys.},
  volume  = {382},
  number  = {2},
  pages   = {875--949},
  year    = {2021}
}

@incollection{Deb13,
  author    = {Debussche, A.},
  title     = {Ergodicity results for the stochastic {N}avier--{S}tokes equations: an introduction},
  booktitle = {Topics in Mathematical Fluid Mechanics},
  series    = {Lecture Notes in Math.},
  volume    = {2073},
  pages     = {23--108},
  publisher = {Springer},
  address   = {Heidelberg},
  year      = {2013}
}

@article{Ner22,
  author  = {Nersesyan, V.},
  title   = {Ergodicity for the randomly forced {N}avier--{S}tokes system in a two-dimensional unbounded domain},
  journal = {Ann. Henri Poincar\'e},
  volume  = {23},
  number  = {6},
  pages   = {2277--2294},
  year    = {2022}
}

@article{HM06,
  author  = {Hairer, M. and Mattingly, J. C.},
  title   = {Ergodicity of the 2{D} {N}avier--{S}tokes equations with degenerate stochastic forcing},
  journal = {Ann. of Math. (2)},
  volume  = {164},
  number  = {3},
  pages   = {993--1032},
  year    = {2006}
}

@article{EMS01,
  author  = {E, W. and Mattingly, J. C. and Sinai, Ya.},
  title   = {Gibbsian dynamics and ergodicity for the stochastically forced {N}avier--{S}tokes equation},
  journal = {Comm. Math. Phys.},
  volume  = {224},
  number  = {1},
  pages   = {83--106},
  year    = {2001}
}

@article{BDHPS03,
  author  = {Briand, Ph. and Delyon, B. and Hu, Y. and Pardoux, E. and Stoica, L.},
  title   = {{$L^p$ solutions of backward stochastic differential equations}},
  journal = {Stochastic Process. Appl.},
  volume  = {108},
  number  = {1},
  pages   = {109--129},
  year    = {2003}
}

@article{FM95,
  author  = {Flandoli, F. and Maslowski, B.},
  title   = {Ergodicity of the 2-{D} {N}avier--{S}tokes equation under random perturbations},
  journal = {Comm. Math. Phys.},
  volume  = {172},
  number  = {1},
  pages   = {119--141},
  year    = {1995}
}

@article{KS17,
  author  = {Kuksin, S. and Shirikyan, A.},
  title   = {Rigorous results in space-periodic two-dimensional turbulence},
  journal = {Physics of Fluids},
  volume  = {29},
  pages   = {125106},
  year    = {2017}
}

@article{KS02,
  author  = {Kuksin, S. and Shirikyan, A.},
  title   = {Coupling approach to white-forced nonlinear {PDE}s},
  journal = {J. Math. Pures Appl. (9)},
  volume  = {81},
  number  = {6},
  pages   = {567--602},
  year    = {2002}
}

@article{FP67,
  author  = {Foia{\c{s}}, C. and Prodi, G.},
  title   = {Sur le comportement global des solutions non-stationnaires des \'equations de {N}avier--{S}tokes en dimension 2},
  journal = {Rend. Sem. Mat. Univ. Padova},
  volume  = {39},
  pages   = {1--34},
  year    = {1967}
}

@article{Mar14,
  author  = {Martirosyan, D.},
  title   = {Exponential mixing for the white-forced damped nonlinear wave equation},
  journal = {Evol. Equ. Control Theory},
  volume  = {3},
  number  = {4},
  pages   = {645--670},
  year    = {2014}
}

@article{Kim06,
  author  = {Kim, J. U.},
  title   = {Invariant measures for a stochastic nonlinear {S}chr\"odinger equation},
  journal = {Indiana Univ. Math. J.},
  volume  = {55},
  number  = {2},
  pages   = {687--717},
  year    = {2006}
}
\end{document}